\documentclass[12pt,a4paper,reqno]{amsart}
\usepackage{color}
\usepackage{amsfonts,amsmath,amssymb,amsxtra,url,float} 
\allowdisplaybreaks[4] 
\usepackage[colorlinks,
  linkcolor=RoyalBlue,
  anchorcolor=Periwinkle,
  citecolor=Orange,
  urlcolor=Green
  ]{hyperref} 
\usepackage[usenames,dvipsnames]{xcolor} 
\usepackage{enumitem}
\usepackage{geometry} 
\usepackage{rotating} 
\usepackage{lscape} 
\usepackage{multirow}
\usepackage{graphicx} 
\usepackage{subfigure} 
\usepackage{tikz}
\usetikzlibrary{decorations.pathreplacing,decorations.markings}
 \tikzset{
  on each segment/.style={
    decorate,
    decoration={
      show path construction,
      moveto code={},
      lineto code={
        \path [#1]
        (\tikzinputsegmentfirst) -- (\tikzinputsegmentlast);
      },
      curveto code={
        \path [#1] (\tikzinputsegmentfirst)
        .. controls
        (\tikzinputsegmentsupporta) and (\tikzinputsegmentsupportb)
        ..
        (\tikzinputsegmentlast);
      },
      closepath code={
        \path [#1]
        (\tikzinputsegmentfirst) -- (\tikzinputsegmentlast);
      },
    },
  },
  mid arrow/.style={postaction={decorate,decoration={
        markings,
        mark=at position 0.6 with {\arrow[#1]{stealth}} 
      }}},
}
\usetikzlibrary{arrows}
\usetikzlibrary{trees}

\usetikzlibrary{matrix}
\usetikzlibrary{patterns}
\usetikzlibrary{shadings} 
\usepackage[all]{xy} 
\usepackage{mathdots} 
\usepackage{dsfont} 
\usepackage{cite}
\usepackage{mathrsfs} 
\usepackage{multicol} 
\numberwithin{figure}{section}
\usepackage{marginnote} 
\usepackage{graphicx} 
\usepackage{multicol} 

\usepackage{fancyhdr} 
\newcommand{\checks}[1]{{\color{black}{#1}}} 
\newcommand{\checkLiu}[1]{{\color{red}{#1}}}

\newtheorem{theorem}{Theorem}[section]
\newtheorem{assumption}[theorem]{Assumption}
\newtheorem{lemma}[theorem]{Lemma}
\newtheorem{corollary}[theorem]{Corollary}
\newtheorem{main theorem}[theorem]{Main Theorem}
\newtheorem{proposition}[theorem]{Proposition}
\newtheorem{definition}[theorem]{Definition}

\newtheorem{remark}[theorem]{Remark}
\newtheorem{example}[theorem]{Example}
\newtheorem{notation}[theorem]{Notation}
\newtheorem{question}[theorem]{Question}
\newtheorem{fact}[theorem]{Fact}

\usetikzlibrary{arrows}

\numberwithin{equation}{section}

\def\<{\langle} 
\def\>{\rangle} 
\def\NN{\mathbb{N}} 
\def\RR{\mathbb{R}} 
\def\II{\mathbb{I}}
\def\CC{\mathbb{C}}

\newcommand{\kk}{\mathds{k}} 
\newcommand{\Q}{\mathcal{Q}} 
\newcommand{\I}{\mathcal{I}} 
\newcommand{\Hom}{\mathrm{Hom}} %
\newcommand{\End}{\mathrm{End}} %
\def\s{\mathfrak{s}}
\def\t{\mathfrak{t}}

\def\itLamb{\mathit{\Lambda}}
\def\itPhi{\mathit{\Phi}}
\def\im{\mathrm{Im}}
\def\ker{\mathrm{Ker}}
\def\bfS{\mathbf{S}}
\def\bfW{\mathbf{W}}
\def\bfE{\pmb{E}}

\def\id{\mathbf{1}}
\def\ident{\mathrm{id}}
\def\scrN{\mathscr{N}\!\!or}
\def\scrA{\mathscr{A}}
\def\frakp{\mathfrak{p}}
\def\frakI{\mathfrak{I}}
\def\frakU{\mathfrak{U}}
\def\ideal{\mathfrak{i}}

\def\emb{\mathrm{emb}}
\def\dd{\mathrm{d}}
\def\rad{\mathrm{rad}}
\def\Leb{\mu_{\mathrm{L}}}
\def\rmL{\mathrm{L}}
\def\tri{\bigtriangleup}
\def\power{\mathrm{e}}
\def\rmi{\mathrm{i}}
\def\expaP{E_{\mathrm{pow}}}

\def\expaF{E_{\mathrm{Fou}}}
\def\expaSW{E_{\mathrm{S-W}}}
\def\ext{\dagger}
\def\compos{\ \lower-0.2ex\hbox{\tikz\draw (0pt, 0pt) circle (.1em);} \ }
\newcommand{\defines}{\it}
\newcommand{\normcat}{\mathsf{nor}}
\newcommand{\Ban}{\mathsf{Ban}}
\newcommand{\norm}[1]{\mathfrak{n}_{#1}}
\newcommand{\To}[1]{\mathop{-\!\!\!-\!\!\!-\!\!\!\longrightarrow}\limits^{#1}}
\newcommand{\bigoplusp}[2]{{\bigoplus_{#1}^{#2}{}_p}\ }
\newcommand{\w}[1]{\widehat{#1}}

\def\point{\checkLiu{}}

\begin{document}

\def\headertitle{Normed modules and the categorification of integrations, series expansions, and differentiations}

\title[\headertitle]{Normed modules and the categorification of integrations, series expansions, and differentiations}

\def\fstpage{1} 
\def\page{$\begin{matrix} {\color{white}0} \\ \thepage \end{matrix}$} 

\pagestyle{fancy}
\fancyhead[LO]{ }
\fancyhead[RO]{ }
\fancyhead[CO]{\ifthenelse{\value{page}=\fstpage}{\ }{\scriptsize{\headertitle}}}
\fancyhead[LE]{ }
\fancyhead[RE]{ }
\fancyhead[CE]{\scriptsize{Yu-Zhe LIU, Shengda LIU, Zhaoyong HUANG, $\&$ Panyue ZHOU}}
\fancyfoot[L]{ }
\fancyfoot[C]{\page} 
\fancyfoot[R]{ }
\renewcommand{\headrulewidth}{0.5pt} 



\author{Yu-Zhe Liu}
\address{School of Mathematics and Statistics, Guizhou University, Guiyang 550025, Guizhou, P. R. China}
\email{liuyz@gzu.edu.cn / yzliu3@163.com}

\author{Shengda Liu}
\address{The State Key Laboratory of Multimodal Artificial Intelligence Systems, Institute of Automation,
         Chinese Academy of Sciences, Beijing 100190, P. R. China.}
\email{thinksheng@foxmail.com}

\author{Zhaoyong Huang$^{\ast}$}\thanks{$^{\ast}$Corresponding author.}
\address{School of Mathematics, Nanjing University, Nanjing 210093, Jiangsu, P. R. China}
\email{huangzy@nju.edu.cn}

\author{Panyue Zhou}
\address{School of Mathematics and Statistics, Changsha University of Science and Technology, Changsha 410114, Hunan, P. R. China}
\email{panyuezhou@163.com}






\maketitle

\vspace{-3mm}

\begin{abstract}
We explore the assignment of norms to $\mathit{\Lambda}$-modules over a finite-dimensional algebra $\mathit{\Lambda}$,
resulting in the establishment of normed $\mathit{\Lambda}$-modules. Our primary contribution lies in constructing
two new categories $\mathscr{N}\!\!or^p$ and $\mathscr{A}^p$, where
each object in $\mathscr{N}\!\!or^p$ is a normed $\mathit{\Lambda}$-module $N$ limited by a special element $v_N\in N$
and a special $\mathit{\Lambda}$-homomorphism $\delta_N: N^{\oplus 2^{\dim\mathit{\Lambda}}} \to N$,
the morphism in $\mathscr{N}\!\!or^p$ is a $\mathit{\Lambda}$-homomorphism $\theta: N\to M$ such that
$\theta(v_N) = v_M$ and $\theta\delta_N = \delta_M\theta^{\oplus 2^{\dim\mathit{\Lambda}}}$,
and $\mathscr{A}^p$ is a full subcategory of $\mathscr{N}\!\!or^p$ generated by all Banach modules.
By examining the objects and morphisms in these categories.
We establish a framework for understanding the categorification of integration, series expansions, and derivatives.
Furthermore, we obtain the Stone--Weierstrass approximation theorem in the sense of $\mathscr{A}^p$.
\end{abstract}

\vspace{-3mm}

\setcounter{tocdepth}{3}
\setcounter{secnumdepth}{3}
\tableofcontents


\section{Introduction}

Mathematical analysis, encompassing branches such as integrations, differentiations, and series expansions,
is an integral component of mathematics and serves as an indispensable tool in various scientific domains
including physics, engineering, and life sciences. Traditionally founded on the $\epsilon$--$\delta$
definition of limits and the theory of completeness of the real numbers, mathematical analysis provides
a rich and diverse array of research topics within its sub-disciplines.
However, adaptation to different applications often obscures a unified understanding of its branches and their interconnections.
For \checks{example}, Lebesgue integration, introduced by Henri Lebesgue \cite{L1928} in 1902, represents
a critical advancement in mathematical analysis.
Understanding Lebesgue's approach to integrability on the real line involves methodical and
incremental steps beginning with the definition of measurable sets and null sets, followed
by exploring measure convergence. The journey continues through the exploration of step functions
and simple functions, progressing to sequences of their convergence and culminating
in the sophisticated construction of spaces for integrable functions and consistent integration methods.
This path, while comprehensive, paves a detailed route to fully appreciate the depth of Lebesgue integration,
as discussed in foundational texts such as \cite{Burk2011} and \cite{HM2014}.
However, the intricate methodologies developed do not directly translate to other branches of analysis,
making it challenging to apply these achievements uniformly across the field.

Category theory has evolved far beyond its original scope, now permeating nearly all branches of mathematics.
Initially formulated by Eilenberg and Mac Lane in the mid-20th century within the realm of algebraic topology \cite{EM1945},
a category fundamentally consists of objects and morphisms.
This framework facilitates a systematic and structural approach to analyzing a wide range of mathematical entities,
from algebraic structures to complex topological spaces.
The true utility of category theory lies in its ability to abstractly model and examine mathematical concepts through functors and natural transformations.
Functors are the ``morphisms between categories'',
systematically relating the objects and morphisms of one category to those of another, thereby uncovering deep interconnections
within mathematical frameworks. Natural transformations extend this by mapping between functors themselves,
ensuring consistency across categorical representations. This level of abstraction proves invaluable
in various mathematical applications, including the categorical descriptions of integration
\cite{Sag1965ana, Guo2013, CL2018Cart-int, CL2018int} and differentiation
\cite{BCS2006diff, BCS2015diff, Lemay2019, GP2020, APL2021diff, IL2023ana, L2023},
the categorical semantics of Differential Linear Logic \cite{BCLS2020, BCS2006diff},
 the Taylor series within Cartesian Differential Categories \cite{L2024}, preliminary categorifications of automorphic
 forms and the analytic continuation of L-functions \cite{Lan1997}, as well as providing cohesive frameworks
 for tackling complex problems such as quotient spaces, direct products, completions, and duality.
Furthermore, recent research has begun to explore the synergy between category theory and mathematical analysis
in the context of artificial intelligence. These advancements leverage categorical structures to
enhance machine learning models and develop more abstract frameworks for AI algorithms \cite{cru2024}.
Additionally, categorical semantics are being applied to better understand and design AI systems,
providing a robust mathematical foundation for their development and analysis \cite{shi2021}.
Through category theory, mathematicians gain a powerful tool for unifying and elucidating the intricacies
of diverse mathematical concepts. Building on the foundational work of Leinster \cite{Lei2023FA},
we describe integration, series expansions, and differentiation using the unified category $\scrA^p$.
Note that the Rota--Baxter algebra \cite{Bax1960,Rot1969-1,Rot1969-2} provides another algebraic description of integration,
but it is different from the categorification of integration given by the category $\scrA^p$.

As the landscape of integration theory expands, so too does the exploration into its algebraic facets, marking a significant
evolution in the approach to integration. Algebraic approaches to integration can be traced back at least to Segal's work \cite{Sag1965ana}.
Building upon the foundational works of Escard\'o--Simpson \cite{ES2001ana}, Freyd \cite{F2008ana} and Leinster \cite{Lei2023FA}
constructed a special category $\scrA^p$, where $p$ is a real number at least $1$.
In this category, objects are triples consisting of a Banach space $V$, an element
$v$ in $V$ with $|v|\leq 1$, and a $\kk$-linear map $\delta: V\oplus_p V \to V$ that satisfies $\delta(v, v) = v$.
Here, the notation $V_1\oplus_pV_2$ represents the direct sum of two normed spaces $V_1$ and $V_2$, where the norm is defined as
$|(v_1, v_2)| = \left(\frac{1}{2}(|v_1|^p + |v_2|^p)\right)^{1/p}$.
Furthermore, Leinster established three significant results as follows:
\begin{itemize}
  \item[(1)] $(L_p([0,1]), 1, \gamma)$ is the initial object in $\scrA^p$,
    where $\gamma:L_p([0,1])\oplus_p L_p([0,1]) \to L_p([0,1])$ is a special $\kk$-linear map
    (indeed, $\gamma$ is the map $\gamma_{\frac{1}{2}}$ given in Corollary \ref{coro:main1}\point);
  \item[(2)] $(\RR,1,m)$ is an object in $\scrA^1$, where $m: \RR\oplus_1\RR \to \RR$ sends $(x,y)$ to $\frac{1}{2}(x+y)$;
  \item[(3)] there exists a unique morphism
    \[H: (L_1([0,1]), 1, \gamma) \to (\RR,1,m)\]
    in $\scrA^1$,
\end{itemize}
see \cite[Theorem 2.1 and Proposition 2.2]{Lei2023FA}.
The homomorphism $H: L_1([0,1]) \to \kk$ is a $\kk$-linear map sending any function $ f $ in $ L_1([0,1])$ to its Lebesgue integral, \checks{i.e.,}
\[ H(f) = (\rmL)\int_0^1 f \, \mathrm{d}\Leb, \]
where $\Leb$ denotes the Lebesgue measure on $ \mathbb{R} $. This profound relationship illustrates Lebesgue integrability and integration
are not merely abstract constructs; rather, they naturally emerge from the foundational principles of Banach spaces. Consequently,
it can be logically inferred that the categorification of Lebesgue integration is inherently connected to, and can be derived from,
the categorification of Banach spaces.
However, we have discovered that Leinster's work can be extended to a more general setting of finite-dimensional algebras,
and it encompasses not only definite and indefinite integrals, but also includes key areas of mathematical analysis such as weak derivatives,
series expansions, and the Stone-Weierstrass approximation theorem.

Building upon Leinster's foundational work, we extend his categorical framework to encompass finite-dimensional algebras,
thereby creating a more versatile and unified approach to integration theory. By incorporating normed modules over these algebras
into our analysis, we bridge the gap between algebraic structures and analytical methods. This extension allows us to
reinterpret classical concepts of integration, differentiation, and series expansions within a broader categorical context.
Furthermore, our approach facilitates the seamless integration of algebraic techniques with analytical processes,
offering a cohesive framework that enhances the applicability and depth of mathematical analysis.
This novel categorical perspective not only unifies disparate areas of analysis but also opens new avenues
for research and application in related scientific fields.

This study aims to explore and construct a comprehensive theoretical framework specifically tailored for normed modules
in finite-dimensional algebras. We introduce and dissect a novel category, denoted by $\scrN^p$,
alongside its fully characterized subcategory, $\scrA^p$. This research endeavors to systematically categorize normed modules
and their operations, aiming to enhance our understanding of fundamental mathematical procedures such as integration,
series expansions, and differentiation. The specific research questions addressed are:

\begin{question} \label{quest} \rm \
\begin{enumerate}[label=(\arabic*)]
\item \label{item first}
  {\it How does the new categorical framework improve our comprehension
  of norm structures within various normed modules over an algebra? }
\item \label{item second}
  {\it What contributions do morphisms in the subcategory $\scrA^p$
  make towards advancing classical integration techniques?}
\item \label{item third}
  {\it What implications does the categorification of normed modules hold for
  the broader mathematical analysis landscape and its practical applications?}
\end{enumerate}
\end{question}

The investigation of these questions not only broadens the scope of category theory in mathematical analysis
and abstract algebra but also introduces novel theoretical tools and perspectives, potentially benefiting
other disciplines such as physics and automation engineering. To comprehensively address the aforementioned questions,
we will delineate the following  key topics in subsequent sections.

Firstly, we introduce functions defined on a finite-dimensional algebra  \(\itLamb\), along with the norm defined on \(\itLamb\)
and any \(\itLamb\)-module \(M\). It is pertinent to note that all \(\itLamb\)-modules considered in this paper are left \(\itLamb\)-modules.
The specifics of these structures are elaborated in Subsection \ref{subsect:normed alg.1}\point~and \ref{subsect:normed mod.1}\point, respectively.
A pivotal motivation for us to introduce normed modules is the pursuit of an integration definition that
transcends the conventional reliance on $L_p$ spaces. This approach is rooted in the understanding that an equivalent definition of
$L_p$ spaces can emerge through the integration itself. However, as highlighted by Leinster, the notion of Lebesgue integrals
is intrinsically linked to Banach spaces. Consequently, our investigation also necessitates considering the completions of
normed finite-dimensional algebras and normed modules, see Subsections \ref{subsect:normed alg.2}\point~and \ref{subsect:normed mod.2}.

Secondly, for a special subset $\II_{\itLamb}$ of $\itLamb$, denoted $\II_{\itLamb} \subseteq \itLamb$, we construct the category $\scrN^p$ in Subsection \ref{sect:two cats.1}.
Its object has the form $(N,v,\delta)$, where $N$ is a normed $\itLamb$-module, $v$ is an element in $N$ satisfying $|v| \le \mu(\II_{\itLamb})$,
$\mu$ is an arbitrary measure defined on $\II_{\itLamb}$
and $\delta: N^{\oplus_p 2^n} \to N$ is a $\itLamb$-homomorphism sending $(v,\ldots, v)$ to $v$. The morphism $h: (N,v,\delta)\to(N',v',\delta')$
is induced by a special $\itLamb$-homomorphism $N \to N'$ satisfying $h\delta = \delta' (h^{\oplus_p 2^n})$.
Furthermore, we consider the full subcategory $\scrA^p$ of $\scrN^p$, where each object $(N,v,\delta)$ consists of a Banach $\itLamb$-module $N$,
an element $v\in N$ and a $\itLamb$-homomorphism $\delta:N^{\oplus_p 2^n} \to N$.

Thirdly, we investigate the set $\bfS_{\tau}(\II_{\itLamb})$ of elementary simple functions (a special step function defined on $\itLamb$),
where $\tau$ is a homomorphism between two $\kk$-algebras. We demonstrate its structure as a $\itLamb$-module (Lemma \ref{lemm:normed mod}\point).
Consequently, we obtain an object $(\bfS_{\tau}(\II_{\itLamb}), \id, \gamma_{\xi})$ (Lemma \ref{lemm:S in Np}\point) in $\scrN^p$
and an object $(\w{\bfS_{\tau}(\II_{\itLamb})}, \id, \w{\gamma}_{\xi})$ in $\scrA^p$, where $\w{\bfS_{\tau}(\II_{\itLamb})}$
is the completion of $\bfS_{\tau}(\II_{\itLamb})$ and $\w{\gamma}_{\xi}$ is induced by $\gamma_{\xi}$.

Fourthly, we prove our main result in Section \ref{sect:Ap-init obj}\point~to answer Question \ref{quest}.\ref{item first},
which provides a unique homomorphism from the initial object in $\scrA^p$ to any normed module
to describe the properties of normed representations of algebra.

\begin{theorem}[{\rm Theorem \ref{thm:main1}\point~and Remark \ref{rmk:main1}\point}] \label{thm:A}
The triple $(\bfS_{\tau}(\II_{\itLamb}), \id, \gamma_{\xi})$ is an $\scrA^p$-initial object in $\scrN^p$,
\checks{i.e.,} for any object $(N,v,\delta)$ in $\scrA^p$, there exists a unique morphism
\[h\in \Hom_{\scrN^p}((\bfS_{\tau}(\II_{\itLamb}), \id, \gamma_{\xi}), (N,v,\delta))\]
such that the diagram
\[\xymatrix@C=2cm{
  (\bfS_{\tau}(\II_{\itLamb}), \id, \gamma_{\xi}) \ar[r]^h \ar[d]_{\subseteq}
& (N,v,\delta) \\
  (\w{\bfS_{\tau}(\II_{\itLamb})}, \id, \w{\gamma}_{\xi}) \ar[ru]_{\w{h}}
& }\]
commutes, where $\w{h}$ is given by the completion of $h:\bfS_{\tau}(\II_{\itLamb})\to N$.
\end{theorem}

Sections \ref{sect:7}\point, \ref{sect:series}\point~and \ref{sect:deriva}\point~realize integrations, series expansions
and derivatives as three morphisms in $\scrA^1$.

In \checks{Section} \ref{sect:7}\point, we construct an object $(\kk,\mu(\II_{\itLamb}), m)$ in $\scrA^p$,
where $m: \kk^{\oplus_p 2^n} \to \kk$ is a $\itLamb$-homomorphism whose definition is given in this section.
Take $(N,v,\delta)=(\kk,\mu(\II_{\itLamb}), m)$ as in Theorem \ref{thm:A}\point, we obtain the following result to
answer Question \ref{quest}.\ref{item second},
which describes numerous integrations by using category $\scrA^p$ in a unified way since $\mu$ is an arbitrary measure.

\begin{theorem}[{\rm Theorem \ref{thm:main2}\point}]  \label{thm:B}
If $\kk=(\kk,|\cdot|,\preceq)$ is an extension of $\RR$, then there exists a unique $\itLamb$-homomorphism
$T: \bfS_{\tau}(\II_{\itLamb}) \to \kk$ such that
\[T: (\bfS_{\tau}(\II_{\itLamb}), \id, \gamma_{\xi}) \to (\kk, \mu(\II_{\itLamb}), m)\]
is a morphism in $\Hom_{\scrN^p}((\bfS_{\tau}(\II_{\itLamb}), \id, \gamma_{\xi}), (\kk, \mu(\II_{\itLamb}), m))$
and the diagram
\[\xymatrix@C=2cm{
  (\bfS_{\tau}(\II_{\itLamb}), \id, \gamma_{\xi}) \ar[r]^{T} \ar[d]_{\subseteq}
& (\kk, \mu(\II_{\itLamb}),m) \\
  (\w{\bfS_{\tau}(\II_{\itLamb})}, \id, \w{\gamma}_{\xi}) \ar[ru]_{\w{T}}
& }\]
commutes, where $\w{T}$ is the unique morphism lying in
$\Hom_{\scrA^p}((\w{\bfS_{\tau}(\II_{\itLamb})}, \id, \w{\gamma}_{\xi}), (\kk, \mu(\II_{\itLamb}), m))$.
Furthermore, if $p=1$, then we have the following three properties of $\w{T}$ by the direct limits
$\underrightarrow{\lim}T_i: \w{T}=\underrightarrow{\lim} E_i \to \kk$
{\rm(}The definitions of $E_i$ and $T_i$ are given in Notation \ref{notation}\point~and Section \ref{sect:7}\point, respectively{\rm)}:
\begin{itemize}
  \item[\rm(1)] {\rm(Formula (\ref{formula:Tu(1)}))} $\w{T}(\id) = \mu(\II_{\itLamb})$;
  \item[\rm(2)] {\rm(Lemma \ref{lemm:Tu}\point)} $\w{T}: \bfS_{\tau}(\II_{\itLamb}) \to \kk$
    is a homomorphism of $\itLamb$-modules;
  \item[\rm(3)] {\rm(Proposition \ref{prop:tri inq}\point)} $|\w{T}(f)|\le \w{T}(|f|)$.
\end{itemize}
\end{theorem}
The morphism $\w{T}$ provides a categorification of integration, and we denote
\begin{align} \label{int}
 \w{T}(f) =: (\scrA^1)\int_{\II_{\itLamb}} f \dd\mu.
\end{align}
The above (1), (2), and (3) show that
\[ (\scrA^1)\int_{\II_{\itLamb}} \id \dd\mu = \mu(\II_{\itLamb}) , \]
\begin{align}\label{Lambda-linear}
 (\scrA^1)\int_{\II_{\itLamb}} (\lambda_1 f_1 + \lambda_2 f_2) \dd\mu
= \lambda_1 \cdot (\scrA^1)\int_{\II_{\itLamb}} f_1 \dd\mu
+ \lambda_2 \cdot (\scrA^1)\int_{\II_{\itLamb}} f_2 \dd\mu \ \ (\lambda_1,\lambda_2\in \itLamb),
\end{align}
and
\[ \bigg|(\scrA^1)\int_{\II_{\itLamb}} f \dd\mu \bigg|
\le (\scrA^1)\int_{\II_{\itLamb}} |f| \dd\mu, \]
respectively.

Let $\kk[X_1, \cdots, X_N]$ ($=\kk[\pmb{X}]$ for short) be the $N$ variables polynomial ring
over a field $\kk$ with $N\ge \dim_{\kk}\itLamb = n$.
Then $\kk[\pmb{X}]$ can be seen as a normed left $\itLamb$-module,
where the norm $\Vert\cdot\Vert_{\kk[\pmb{X}]}$ is either (\ref{formula: poly norm 1})
or (\ref{formula: poly norm 2}).
In Section \ref{sect:series}, we get two corollaries as follows to answer Question \ref{quest}.\ref{item third}.

\begin{corollary} \label{coro:C}  Let $\scrA^p$ satisfy $p=1$.
\begin{itemize}
  \item[{\rm(1)}] {\rm(Corollary \ref{coro:power series}\point/Weierstrass Approximation Theorem)}
    If $N=n$ and $\Vert\cdot\Vert_{\kk[\pmb{X}]}$ is defined by (\ref{formula: poly norm 1}),
    then the unique morphism in
    \[\w{\expaP} \in
      \Hom_{\scrA^1}(
        (\w{\bfS_{\tau}(\II_{\itLamb})}, \id, \gamma_{\xi}),
        (\w{\kk[\pmb{X}]}, \id, \w{\gamma_{\xi}}|_{\w{\kk[\pmb{X}]}})
        )\]
    shows that for any function $f \in \w{\bfS_{\tau}(\II_{\itLamb})}$,
    there exists a sequence $\{P_i\}_{i\in\NN}$ of polynomials such that
    \[ \w{\expaP}(f) = \underleftarrow{\lim} P_i. \]

  \item[{\rm(2)}] {\rm (Corollary \ref{coro:Fourier series}\point)}
    If $\kk=\CC$, $N=2n$ and $\Vert\cdot\Vert_{\kk[\pmb{X}]}$
    is defined by (\ref{formula: poly norm 2}), then the unique morphism in
    \[ \w{\expaF} \in
       \Hom_{\scrA^1}(
         (\bfS_{\tau}(\II_{\itLamb}), \id, \gamma_{\xi}),
         (\w{\CC[\power^{\pm2\pi\rmi\pmb{X}}]}, \id,
           \w{\gamma_{\xi}}|_{\w{\CC[\power^{\pm2\pi\rmi\pmb{X}}] }})
         ) \]
    shows that for any function $f \in \w{\bfS_{\tau}(\II_{\itLamb})}$,
    there exists a sequence $\{P_i\}_{i\in\NN}$ of triangulated polynomials such that
    \[ \w{\expaF}(f) = \underleftarrow{\lim} P_i. \]
\end{itemize}
\end{corollary}

\noindent
Furthermore, we show the Stone--Weierstrass approximation theorem in Subsection \ref{subsect:S-W thm}\point,
see Corollary \ref{coro:S-W approx}.

\begin{corollary}[{Corollary \ref{coro:S-W approx}\point,
Stone$-$Weierstrass Approximation Theorem}]
There exists a unique morphism
\[ \expaSW: (\bfS_{\tau}(\II_{\itLamb}), \id, \gamma_{\xi})
\to (\bfW, \id, \w{\gamma_{\xi\dagger}}) \]
in $\Hom_{\scrN^1}((\bfS_{\tau}(\II_{\itLamb}), \id, \w{\gamma}_{\xi}),
(\bfW, \id, \w{\gamma_{\xi\dagger}}) )$ such that the diagram
\[\xymatrix@C=2cm{
  (\bfS_{\tau}(\II_{\itLamb}), \id, \gamma_{\xi}) \ar[r]^{\expaSW} \ar[d]_{\subseteq}
& (\bfW, \id, \w{\gamma_{\xi\dagger}}) \\
  (\w{\bfS_{\tau}(\II_{\itLamb})}, \id, \w{\gamma}_{\xi}) \ar[ru]_{\w{\expaSW}}
& }\]
commutes, where the definition of $\bfW$ is a direct limit defined in Subsection \ref{subsect:S-W thm}\point;
and $\w{\expaSW}$ is the unique extension of $\expaSW$ lying in $\Hom_{\scrA^1}((\w{\bfS_{\tau}(\II_{\itLamb})}, \id, \w{\gamma}_{\xi}),$
$(\bfW, \id, \w{\gamma_{\xi\dagger}}) )$.
\end{corollary}

In Section \ref{sect:deriva}\point, we recall some works of Leinster and Meckes.
Based on their work, we show the following theorem.
\begin{theorem} \label{thm:1.6}
Let $p=1$, $\itLamb=\RR=\kk$, $\tau=\ident: \RR \to \RR, x\mapsto x$, $\II_{\itLamb}=[0,1]$ and $\xi=\frac{1}{2}$;
and, for simplification, we write $\w{\bfS}:=\w{\bfS_{\tau}(\II_{\itLamb})}$.
\begin{itemize}
  \item[{\rm(1)}] {\rm(Theorem \ref{thm:diff}\point)}
    \begin{itemize}
      \item[{\rm(i)}] A morphism in $\Hom_{\scrA^1}((\w{\bfS},\id, \w{\gamma}_{\frac{1}{2}}), (N,v,\delta))$ is zero if and only if $v=0$.
      \item[{\rm(ii)}] Furthermore, there is no morphism in $\scrA^1$ starting with $(\w{\bfS},\id, \w{\gamma}_{\frac{1}{2}})$ such that
          this morphism sends any almost everywhere differentiable function $f(x)$ to its weak derivative $\frac{\dd f}{\dd x}$.
    \end{itemize}

  \item[{\rm(2)}] {\rm(Theorem \ref{thm:diff-non initial}\point)}
    The differentiation $D$ is a morphism in $\scrA^p$ ending with the initial object $(\w{\bfS}, \id, \w{\gamma}_{\frac{1}{2}})$.
\end{itemize}
\end{theorem}

Recall that any function $f:\II_{\itLamb}=[0,1] \to \RR$
in $\bfS$ is a {\defines step function},
\checks{i.e.,} there is a dissection $[0,1] = \bigcup_{i=1}^t \II_i$ of $[0,1]$ with
$\II_i \cap \II_j = \varnothing$ for any $1\le i\ne j\le t$,
such that each $f|_{\II_i}(x)$ is a constant in $\RR$.
Then $\frac{\dd f}{\dd x}$ almost everywhere equals to $0$.
It follows that if the completion $D$ of $\frac{\dd}{\dd x}$
sends every function in $\w{\bfS}$ to zero. Furthermore, if $D$ is a morphism in
$\Hom_{\scrA^1}((\w{\bfS},\id, \w{\gamma}_{\frac{1}{2}}), (N,v,\delta))$,
then we have $v = D(\id) = 0$, it follows that $D=0$ by Theorem \ref{thm:1.6}\point~(1) (i),
this is a contradiction.
Therefore, Theorem \ref{thm:1.6}\point~(1) (i) shows that
differentiation, i.e., the homomorphism $D$,
is not a morphism in $\scrA^1$ with domain the initial object of $\scrA^1$.
Then we obtain the statement (ii) of Theorem \ref{thm:1.6}\point~(1).
Naturally, we would ask whether $D$ can be characterized by $\scrA^p$?
To do this, we prove Theorem \ref{thm:1.6}\point~(2) in Subsection \ref{subsect:deri 2}\point,
and show that $D$ is a morphism in a category $\scrA^p$.

Finally, we provide some applications for our main results in Section \ref{sect:9}.
In Subsection \ref{subsect:app1}\point, we assume $\kk=\RR$, $(\itLamb, \prec, \Vert\cdot\Vert_{\itLamb}) = (\RR,\le,|\cdot|)$,
$B_{\RR}=\{1\}$, $\norm{}: B_{\RR}\to \{1\}\subseteq\RR^{\ge 0}$,
$\II_{\RR}=[0,1]$, $\xi=\frac{1}{2}$, $\kappa_0(x)=\frac{x}{2}$, $\kappa_1(x)=\frac{x+1}{2}$ and
$\tau=\text{id}_{\RR}:\RR\to \RR$, and let $\Leb$ be the Lebesgue measure.
Then (\ref{int}) is a Lebesgue integration
\[(\scrA^1)\int_{\II_{\RR}=[0,1]} f\dd\Leb = (\rmL)\int_0^1 f \dd\Leb, \]
and (\ref{Lambda-linear}) shows that Lebesgue integration is $\RR$-linear.
This result provides a categorification of Lebesgue integration.
In Subsection \ref{subsect:app2}\point, we provide two examples for Corollary \ref{coro:C}\point~to
show that the Taylor series and Fourier series can be realized as two morphisms in $\scrA^1$ with domain the initial object.

\section{Preliminaries} \label{sec:preliminaries}

In this section we recall some concepts in the category theory and representation theory of algebras, including limits in the category theory
(cf. \cite[Chapter 5]{R1979} and \cite[Chapter III, pages 62--74]{L2013}),
$\kk$-algebras (cf. \cite[Chapter I]{ASS2006}),
some methods to establish topologies on algebras
(cf. \cite[Chapter 10]{AM2018}).
These concepts are familiar to algebraists but may not be as familiar to those in the field of analysts.

\subsection{Categories and limits} \label{sect:lim compl}

Recall that a {\defines category} $\mathcal{C}$ consists of three ingredients: a class of {\defines objects},
a set $\Hom_{\mathcal{C}}(X,Y)$ of {\defines morphisms} for any objects $X$ and $Y$ in $\mathcal{C}$,
and the {\defines composition} $\Hom_{\mathcal{C}}(X,Y)\times\Hom_{\mathcal{C}}(Y,Z)\to\Hom_{\mathcal{C}}(X,Z)$,
denoted by
\[ (f:X\to Y, g:Y\to Z) \mapsto gf: X\to Z, \]
for any objects $X$, $Y$ and $Z$ in $\mathcal{C}$. These ingredients are subject to the following axioms:
\begin{itemize}
  \item[(1)] the $\Hom$ sets are pairwise disjoint;
  \item[(2)] for any object $X$, the {\defines identity morphism} $1_X: X\to X$ in $\Hom_{\mathcal{C}}(X,X)$ exists;
  \item[(3)] the composition is associative: given morphisms $\xymatrix{U\ar[r]^{f} & V\ar[r]^{g} & W\ar[r]^{h} & X,}$
    we have \[ h(gf) = (hg)f. \]
\end{itemize}

Next, we review the limits in the category theory.

\begin{definition}[cf. {\cite[Chapter 5, Section 5.2]{R1979}}] \rm
Let $\frakI=(\frakI,\preceq)$ be a partially ordered set, and let $\mathcal{C}$ be a category.
A {\defines direct system} in $\mathcal{C}$ over $\frakI$ is an ordered pair $((M_i)_{i\in\frakI}, (\varphi_{ij})_{i\preceq j})$,
where $(M_i)_{i\in\frakI}$ is an indexed family of objects in $\mathcal{C}$
and $(\varphi_{ij}:M_i\to M_j)_{i\preceq j}$ is an indexed family of morphisms for which $\varphi_{ii}=1_{M_i}$ for all $i$,
such that the following diagram
\[\xymatrix{
M_i \ar[rd]_{\varphi_{ij}} \ar[rr]^{\varphi_{ik}} & & M_k \\
 & M_j \ar[ru]_{\varphi_{jk}} &
}\]
commutes whenever $i\preceq j\preceq k$. Furthermore, for the above direct system $((M_i)_{i\in\frakI}, (\varphi_{ij})_{i\preceq j})$,
the {\defines direct limit} (also called {\defines inductive limit} or {\defines colimit})
is an object, say $\underrightarrow{\lim} M_i$, and {\defines insertion morphisms}
$(\alpha_i: M_i \to \underrightarrow{\lim} M_i)_{i\in\frakI}$ such that
\begin{itemize}
  \item[(1)] $\alpha_j\varphi_{ij}=\alpha_i$ whenever $i\preceq j$;
  \item[(2)] for any object $X$ in $\mathcal{C}$ such that
    there are given morphisms $f_i: M_i \to X$ satisfying $f_j\varphi_{ij}=f_i$ for all $i\preceq j$,
    there exists a unique morphism $\theta: \underrightarrow{\lim} M_i \to X$ making the following diagram
    \[ \xymatrix{
        \underrightarrow{\lim} M_i \ar@{-->}[rr]^{\theta}_{(\exists!)} & & X \\
        & M_i \ar[lu]^{\alpha_i} \ar[ru]_{f_i} \ar[d]^(0.35){\varphi_{ij}}_(0.35){(i\preceq j)}  & \\
        & M_j \ar@/^1pc/[luu]^{\alpha_j} \ar@/_1pc/[ruu]_{f_j}  &
      }\]
commutes.
\end{itemize}
\end{definition}

\begin{example} \rm
Let $\{x_n\}_{n\in\NN^+}$ be a monotonically increasing sequence of real numbers,
and let $\RR$ be the partially ordered category $(\RR,\le)$, in which the elements are real numbers
and the morphisms are of the form $\le_{r_2r_1}: r_1 \to r_2$ ($r_2\le r_1$).
If $\{x_n\}_{n\in\NN^+}$ has limit $x$ as in analysis, i.e., for any $\epsilon>0$,
there exists $N\in\NN^+$ such that $|x_n-x|<\epsilon$ holds for all $n>N$,
then $x=\underrightarrow{\lim} x_n$.
Indeed, for any $x'\in\RR$ such that the morphisms $(\alpha_i=\le_{x_ix'} :x_i\to x')_{i\in\NN^+}$ exist,
there is a morphism $\theta=\le_{xx'}:x\to x'$ such that the following diagram
\[ \xymatrix@C=1.3cm@R=1.3cm{
   x \ar@{-->}[rr]^{\theta=\le_{xx'}} & & x' \\
   & x_i \ar[lu]^(0.3){\le_{x_ix}} \ar[ru]_(0.3){\le_{x_ix'}} \ar[d]^(0.35){\le_{x_ix_j}}  & \\
   & x_j \ar@/^1.3pc/[luu]^{\le_{x_jx}} \ar@/_1.3pc/[ruu]_{\le_{x_jx'}}  &
}\]
commutes. It is clear that the morphism $\theta$ is unique in this example.
Furthermore, $x\le x'$ holds because if $x'<x$ then we can find some $x_t$
such that $x_t>x'$, i.e., $\alpha_t\in \Hom_{(\RR,\le)}(x',x_t)=\varnothing$, this is a contradiction.
\end{example}

\begin{definition}[cf. {\cite[Chapter 5, Section 5.2]{R1979}}] \rm \label{exp:colim}
Let $\frakI=(\frakI,\preceq)$ be a partially ordered set, and let $\mathcal{C}$ be a category.
An {\defines inverse system} in $\mathcal{C}$ over $\frakI$ is an ordered pair $((M_i)_{i\in\frakI}, (\psi_{ij})_{j\preceq i})$,
where $(M_i)_{i\in\frakI}$ is an indexed family of objects in $\mathcal{C}$
and $(\psi_{ij}:M_j\to M_i)_{j\preceq i}$ is an indexed family of morphisms for which $\psi_{ii}=1_{M_i}$ for all $i$,
such that the following diagram
\[\xymatrix{
M_i \ar@{<-}[rd]_{\psi_{ij}} \ar@{<-}[rr]^{\psi_{ik}} & & M_k \\
 & M_j \ar@{<-}[ru]_{\psi_{jk}} &
}\]
commutes whenever $i\preceq j\preceq k$. Furthermore, for the above direct system $((M_i)_{i\in\frakI}, (\psi_{ij})_{j\preceq i})$,
the {\defines inverse limit} (also called {\defines projective limit} or {\defines limit})
is an object, say $\underleftarrow{\lim} M_i$, and {\defines projects morphisms}
$(\alpha_i: \underleftarrow{\lim}M_i \to  M_i)_{i\in\frakI}$ such that
\begin{itemize}
  \item[(1)] $\psi_{ji}\alpha_j=\alpha_i$ whenever $i\preceq j$;
  \item[(2)] for any object $X$ in $\mathcal{C}$ such that
    there are given morphisms $f_i: X\to M_i$ satisfying $\psi_{ji}f_j=f_i$ for all $i\preceq j$,
    there exists a unique morphism $\vartheta: X\to\underrightarrow{\lim} M_i$ making the following diagram
    \[ \xymatrix{
        \underleftarrow{\lim} M_i \ar@{<--}[rr]^{\vartheta}_{(\exists!)} & & X \\
        & M_i \ar@{<-}[lu]^{\alpha_i} \ar@{<-}[ru]_{f_i} \ar@{<-}[d]^(0.35){\psi_{ij}}_(0.35){(i\preceq j)}  & \\
        & M_j \ar@{<-}@/^1pc/[luu]^{\alpha_j} \ar@{<-}@/_1pc/[ruu]_{f_j}  &
      }\]
commutes.
\end{itemize}
\end{definition}

\begin{example} \rm \label{exp:lim}
Let $\{x_n\}_{n\in\NN^+}$ be a monotonically decreasing sequence of real numbers,
and let $\RR$ be the partially ordered category $(\RR,\le)$.
If $\{x_n\}_{n\in\NN^+}$ has limit $x$ as in analysis,
then we have $x=\underleftarrow{\lim} x_n$ by a way similar to that in Example \ref{exp:colim}.
\end{example}

\subsection{\texorpdfstring{$\kk$-algebras and their completions}{}}

Let $\kk$ be a field. In this subsection we recall the definitions of $\kk$-algebras and the completions of $\kk$-algebras.
All concepts in this subsection are parallel to those in \cite[Chapter 10, Section 10.1]{AM2018}
which extracts some important results about the completions of Abelian groups.

\subsubsection{\texorpdfstring{$\kk$-algebras}{}}

\begin{definition} \rm
A {\defines $\kk$-algebra} $A$ defined over $\kk$ is both a ring and a $\kk$-vector space such that
\[ k(aa') = (ka)a' = a(ka'). \]
In particular,
\begin{itemize}
  \item[(1)] if $A$ is a commutative ring, \checks{i.e.,} $a_1a_2=a_2a_1$ holds for all $a_1,a_2\in A$,
    then we call that $A$ is {\defines commutative}, otherwise, we call that it is {\defines non-commutative};
  \item[(2)] if the $\kk$-dimension $\dim_{\kk}A$ of $A$,
    i.e., the dimension of $\itLamb$ as a $\kk$-vector space,
    is finite, then we call that $A$ is a {\defines finite-dimensional $\kk$-algebra},
    otherwise, we call that it is {\defines an infinite-dimensional $\kk$-algebra}.
\end{itemize}
\end{definition}

In this paper, we do not require the commutativity of $\kk$-algebras,
but we always suppose that every $\kk$-algebra in our paper is a finite-dimensional $\kk$-algebra with identity $1$.

Recall that an {\defines idempotent} of a $\kk$-algebra $A$ is an element $e$ in $A$ such that $e^2=e$.
Obviously, $0$ and $1$ are idempotents. If an idempotent $e$ has a decomposition
\[ e = e'+e'' \]
such that
\begin{itemize}
  \item[(1)] $e'$ and $e''$ are non-zero idempotents;
  \item[(2)] $e'$ and $e''$ are {\defines orthogonal}, \checks{i.e.,} $e'e''=0=e''e'$,
\end{itemize}
then we call $e$ {\defines decomposable}.
We call $e$ a {\defines primitive idempotent} if it is not decomposable.
Furthermore, one can prove that $1$ has a decomposition
\[ 1 = e_1 + e_2 + \cdots + e_t \]
such that all $e_i$ are primitive idempotents and
$e_ie_j=0$ holds for all $i\ne j$,
we call $\{e_1,\ldots, e_t\}$ a {\defines complete set of primitive orthogonal idempotents},
see \cite[Chapter I, Page 18]{ASS2006}.

Let $e_1$, $\ldots$, $e_t$ be the complete set of primitive orthogonal idempotents.
Then $A$ has a decomposition $A=\bigoplus_{i=1}^t Ae_i$, where
each direct summand $Ae_i$ is an indecomposable left $A$-module.
We say $A$ is {\defines basic} if $Ae_i\not\cong Ae_j$ for all $1\le i\ne j\le t$.

\begin{example} \rm \label{exp:k-alg}
The set $\mathbf{M}_n(\kk)$ of all $n\times n$ matrices over $\kk$,
the polynomial ring $\kk[x_1, \cdots, x_n]$, and the field $\kk$ itself are $\kk$-algebras.
In particular, $\mathbf{M}_n(\kk)$ and $\kk$ are finite-dimensional,
and $\kk[x_1, \cdots, x_n]$ is infinite-dimensional.
\end{example}

Recall that a quiver is a quadruple $\Q=(\Q_0, \Q_1, \s, \t)$ where $\Q_0$ is the set of vertices,
$\Q_1$ is the set of arrows, and $\s,\t:\Q_1\to\Q_0$ are functions respectively sending each arrow to its starting point and ending point.
Then any vertex $v\in\Q_0$ can be seen as a path on $\Q$ whose length is zero,
and any arrow $\alpha\in\Q_1$ can be seen as a path on $\Q$ whose length is one.
A path $\wp$ of length $l$, denoted $\ell(\wp)$, is the composition $\alpha_l\cdots\alpha_2\alpha_1$ of arrows $\alpha_1$, $\ldots$, $\alpha_l$,
where $\t(\alpha_i)=\s(\alpha_{i+1})$ for all $1\le i< l$.
Then, naturally, we define the composition of two paths $\wp_1 = \alpha_l \cdots \alpha_1$ and $\wp_2 = \beta_{\ell} \cdots \beta_1$ as:
\[ \wp_2\wp_1 = \beta_{\ell} \cdots \beta_1 \alpha_l \cdots \alpha_1 \]
provided that the ending point $\t(\wp_1)$ of $\wp_1$ coincides with the starting point $\s(\wp_2)$ of $\wp_2$,
otherwise (i.e., $\t(\wp_1) \ne \s(\wp_2)$), then the composition is defined to be zero.
Consequently, let $\Q_l$ be the set of all paths of length $l$.
Then $\kk\Q := \mathrm{span}_{\kk}(\bigcup_{l \ge 0} \Q_l)$,
known as the {\defines path algebra} of $\Q$, is a $\kk$-algebra whose multiplication
is defined as follows:
\[ \kk\Q \times \kk\Q \to \kk\Q\ \text{via}\ (k_1\wp_1,k_2\wp_2)\mapsto
\begin{cases}
k_1k_2\cdot \wp_2\wp_1, & \text{ if } \t(\wp_1)=\s(\wp_2);  \\
\ \ \ \ \ \  0, & \text{ otherwise.}
\end{cases}
\]
The following result shows that we can describe all finite-dimensional $\kk$-algebras using quivers,
see \cite[page 43]{Ringel1984} and \cite[Theorem 1.9]{ARS1995}.
The idea of such a graphical representation seems to go back to
Gabriel \cite{Gab1960}, Grothendieck \cite{Gro1957}, and Thrall \cite{Thra1947},
but it became widespread in the early seventies, mainly due to Gabriel \cite{Gab1972, Gab1973}.

\begin{theorem}[{Gabriel \cite[Chapter II, Theorem 3.7]{ASS2006}}]
\label{Gab's thm}
For any finite-dimensional $\kk$-algebra $A$, there is a finite quiver $\Q$,
i.e., the vertex set and arrow set are finite sets, and an admissible
ideal\footnote{An admissible ideal $\I$ of $\kk\Q$ is an ideal such that
$R_{\Q}^m \subseteq\I \subseteq R_{\Q}^2$ holds for some $m\ge 2$, see \cite[Chapter II, Section II.1, page 53]{ASS2006},
where $R_{\Q}^t$ is the ideal of $\kk\Q$ generated by all paths of length $\ge t$.}
$\I$  of $\kk\Q$ such that
the module category of $A$ is equivalent to that of $\kk\Q/\I$.
Furthermore, if $A$ is basic, we have $A\cong\kk\Q/\I$.
\end{theorem}

\begin{remark} \rm
We provide a remark for the isomorphism $A\cong\kk\Q/\I$ given in Theorem \ref{Gab's thm}\point~
here: the existence of the quiver $\Q$ is unique if $A$ is basic and $\I$ is admissible;
the definition of admissible can be found in \cite[Chapter I, Section I.6]{ASS2006}.
\end{remark}

\subsubsection{\texorpdfstring{Topologies on $\kk$-algebras}{}}
Now we recall the topologies of $\kk$-algebras $A$ (not necessarily basic or finite-dimensional).
Let $\ideal(A)$ be the set of all ideals of $A$, which forms a partially ordered set $\ideal(A) = (\ideal(A), \preceq)$
with the partial order defined by the inclusion, i.e., for any $A_1, A_2\in\ideal(A)$, we have
\[ A_1\preceq A_2 \text{ if and only if } A_1 \subseteq A_2. \]
Notice that $A$ has two trivial ideal $0$ and $A$, then we have $\ideal(A)\ne\varnothing$
and have a descending chain $A_0=A \succeq A_1=0 \succeq A_2=0 \succeq \cdots$.
Thus, there is at least one descending chain of ideals.
Let $\mathcal{J}$ be a descending chain
\begin{center}
  $A_0 = A \succeq A_1 \succeq A_2 \succeq \cdots$
\end{center}
of ideals.
We say a subset $U$ of $A$ satisfies the {\defines N-condition}, if it meets the following criteria:
\begin{itemize}
  \item[(N1)] $U$ contains the zero of $A$;
  \item[(N2)] there exists some $j\in\NN$ such that $U\supseteq A_j$.
\end{itemize}
Furthermore, we denote by $\frakU_A(0)$ the set of all subsets satisfying the $N$-condition,
which forms a partially ordered set with the partial order $\preceq$ given by $\subseteq$.

\begin{lemma} \label{lemm:top}
The set $\frakU_A(0)$ is a topology defined on $A$, in other words,
it satisfies the following four conditions.
\begin{itemize}
  \item[\rm(1)] For any $U \in \frakU_A(0)$, we have $0\in U$.
  \item[\rm(2)] $\frakU_A(0)$ is closed under finite intersection, \checks{i.e.,} for any $U_1, \ldots, U_t \in \frakU_A(0)$,
    we have $\bigcap_{1\le j\le t} U_j \in \frakU_A(0)$.
  \item[\rm(3)] If $U\in\frakU_A(0)$ and $U\subseteq V\subseteq A$, then $V\in \frakU_A(0)$.
  \item[\rm(4)] If $U\in\frakU_A(0)$, then there is a set $V\in\frakU_A(0)$ such that $V\subseteq U$
   and $U-y:=\{u-y\mid u\in U\}\in \frakU_A(0)$ for all $y\in V$.
\end{itemize}
\end{lemma}

\begin{proof}
First, (1) is trivial by the condition (N1).

Second, for arbitrary two subset $U_1$ and $U_2$, there are $A_{j_1}$ and $A_{j_2}$ such that
$U_1\supseteq A_{j_1}$ and $U_2\supset A_{j_2}$. Then $U_1\cap U_2\supseteq A_{j_1}\cap A_{j_2}$.
By the definition of $A_j$, we have $A_{j_1}\cap A_{j_2} = A_{\min\{j_1, j_2\}}$,
\checks{i.e.,} \[U_1\cap U_2\supseteq A_{\min\{j_1,j_2\}}. \]
Since $0\in U_1\cap U_2$ trivially, we have $U_1\cap U_2\in \frakU_A(0)$.
By induction, we obtain (2).

Third, assume $U\in\frakU_A(0)$ and $U\subseteq V\subseteq A$.  By the definition of $\frakU_A(0)$,
we have  $0\in U$ and $U\supseteq A_j$ for some $j$.
Thus, $0\in V$ and $V\supseteq A_j$, so we obtain (3).

Finally, for each $U\in\frakU_A(0)$, we can find $V$ in the following way.
There exists an index $\jmath$ such that $U\not\supseteq A_{\jmath-1}$ and
$U \supseteq A_{\jmath} \supseteq  A_{\jmath+1} \supseteq \cdots$.
Take $V = \bigcap_{j\le\jmath} A_j$ ($= A_{\jmath} \subseteq U$).
For any $y\in V$, we have (N1), \checks{i.e.,} $0=y-y\in U-y=\{u-y\mid u\in U\}$ by $y\in V\subseteq U$;
and have (N2) since $a=(a+y)-y$ holds for any $a\in V$ and $a+y\in V$.
Then we obtain $U-y\in \frakU_A(0)$, \checks{i.e.,} (4) holds.
\end{proof}

\begin{definition} \rm \label{def:J-top}
The set $\frakU_A(0)$ is called the {\defines $\mathcal{J}$-topology} of $A$.
Furthermore, we can define open sets on $A$.
\begin{itemize}
\item[(1)] The subset in $\frakU_A(0)$ is called a {\defines neighborhood} of $0$.
For any $U\in\frakU_A(0)$, the union $\bigcup_V V$ of all subsets $V$ given in
Lemma \ref{lemm:top}\point~(4) is called the {\defines interior} of $U$ and denote $\bigcup_V V$ by $U^{\circ}$.

\item[(2)] A neighborhood $U$ is called {\defines open} if $U=U^{\circ}$.
An {\defines open set} $O$ defined on $A$ is one of the following cases:
\begin{itemize}
  \item[(a)] $O$ equals either $A$ or $\varnothing$;
  \item[(b)] $O$ is the intersection of a finite number of open neighborhoods;
  \item[(c)] $O$ is the union of any number of open neighborhoods.
\end{itemize}
\end{itemize}
\end{definition}

It induces the definitions of continuous homomorphisms of $\kk$-algebras.

\begin{definition} \rm
Let $A_1$ and $A_2$ be two $\kk$-algebras, and let $\mathcal{J}_1$ and $\mathcal{J}_2$
be two descending chains of ideals in $A_1$ and $A_2$, respectively.
Let $\frakU_{A_1}(0)$ and $\frakU_{A_2}(0)$ be the $\mathcal{J}_1$-topology and
$\mathcal{J}_2$-topology given by $\mathcal{J}_1$ and $\mathcal{J}_2$, respectively.
A homomorphism $h:A_1\to A_2$ of $\kk$-algebras is called {\defines continuous} if
the preimage of an arbitrary open set on $A_2$ is an open set on $A_1$.
\end{definition}

\begin{lemma} \label{lemm:contin}
Let $A$ be a $\kk$-algebra with a $\mathcal{J}$-topology.
Then the addition $+: A\times A\to A$ and each $\kk$-linear transformation $h_{\lambda}: A\to A$
defined by $a \mapsto \lambda a$ $(\lambda\in A)$ are continuous.
\end{lemma}

\begin{proof}
It is obvious that $\text{id}_A=h_1: A\to A$ via $a\mapsto a$ is continuous.
The continuity of $h_{\lambda}$ can be given by $\text{id}_A$.

Let $\mathcal{J}=$
$$A=A_0 \succeq A_1 \succeq A_2 \succeq \cdots.$$
For any open neighborhood $U$ of $0$, its preimage is
\[+^{-1}(U)=\{ (x_1, x_2)\mid x_1+x_2\in U \} =: \widetilde{U}. \]
We need to show that $\widetilde{U}\in\frakU_{A\times A}((0,0))$ and $\widetilde{U}^{\circ}= \widetilde{U}$
in the case for $A\times A$ being a $\kk$-algebra,
where the descending chain, say $\mathcal{J}_{A\times A}$, of $A\times A$ is induced by $\mathcal{J}$ as follows.
\[ A\times A = A_0 \times A_0 \succeq A_1\times A_1 \succeq A_2\times A_2 \succeq \cdots.\]
First of all, the zero element of $A\times A$ is $(0,0)$ which satisfies $0\in U$ and $0+0=0\in U$, then $(0,0)\in\widetilde{U}$.

Secondly, since $U$ is a neighborhood of $0$, there exists an ideal $A_j$ of $\mathcal{J}$ such that $U\supseteq A_j$.
Then for any $x_1, x_2 \in A_j$, we have $x_1+x_2\in A_j\subseteq U$, \checks{i.e.,} $(x_1, x_2) \in \widetilde{U}$.
It follows that $A_j\times A_j \subseteq \widetilde{U}$. We obtain $\widetilde{U}\in\frakU_{A\times A}((0,0))$.

Thirdly, for any $(y_1,y_2)\in\widetilde{U}$, we have $y_1+y_2\in U$ by the definition of $\widetilde{U}$, then,
\[(0,0) = (y_1-y_1, y_2-y_2) \in \widetilde{U}-(y_1,y_2) = \{ (x_1-y_1,x_2-y_2) \mid x_1+x_2\in U \}, \]
\checks{i.e.,} (N1) holds.
On the other hand, for any $(z_1,z_2)\in A_j\times A_j$, we have
\[(z_1,z_2)=((z_1+y_1)-y_1, (z_2+y_2)-y_2).\]
Note that $z_1+y_1+z_2+y_2 = (y_1+y_2)+(z_1+z_2)$ is an element lying in $U+(z_1+z_2)$.
Since $U$ is open, we have
\[U+(z_1+z_2) = U^{\circ}-(-(z_1+z_2)) = \{u+(z_1+z_2) \mid u\in U\} \in \frakU_A(0) \]
by Lemma \ref{lemm:top}\point~(4) and Definition \ref{def:J-top}\point,
\checks{i.e.,} $U+(z_1+z_2)$ is a set satisfying Lemma \ref{lemm:top}\point~(4). Then
\[U^{\circ} = \bigcup_{
    {V \subseteq U,\ V \text{ satisfies}}
 \atop {\text{Lemma \ref{lemm:top}\point~(4)}}
}V \supseteq U+(z_1+z_2), \]
and so, we obtain $(y_1+y_2)+(z_1+z_2) \in U+(z_1+z_2) \subseteq U^{\circ}$, \checks{i.e.,} $(y_1+y_2)+(z_1+z_2)\in U$.
Thus, $(z_1,z_2) \in \widetilde{U}$. It follows that $A_j\times A_j\subseteq \widetilde{U}-(y_1,y_2)$, and thus (N2) holds.
Therefore, $\widetilde{U}-(y_1,y_2)\in \frakU_{A\times A}((0,0))$.
In summary, we have that $\widetilde{U}$ satisfies Lemma \ref{lemm:top}\point~(4),
and so, by Definition \ref{def:J-top}\point, it is clear that $\widetilde{U}^{\circ}=\widetilde{U}$.
\end{proof}

\begin{definition}[{cf. \cite[Chapter 10, page 101]{AM2018}}] \rm
A {\defines topological $\kk$-algebra} is a $\kk$-algebra  equipped with a topology such that
the addition $+:A\times A\to A$ and each $\kk$-linear transformation $-h_1:A\to A$ via $a\mapsto -a$ are continuous.
\end{definition}

The following result is a consequence of Lemma \ref{lemm:contin}.

\begin{proposition}
Given an arbitrary $\kk$-algebra $A$ and its descending chain $\mathcal{J}$ of ideals.
Then $A$ becomes a topological $\kk$-algebra with the $\mathcal{J}$-topology $\frakU_A(0)$.
\end{proposition}

In this paper, we refer to $A$ as a {\defines $\mathcal{J}$-topological $\kk$-algebra}.

\subsubsection{\texorpdfstring{Completions induced by $\mathcal{J}$-topologies}{}}

Assume that $|\cdot|:\kk\to\RR^{\ge 0}$ is a norm defined on the field $\kk$ in this subsection, i.e., $|\cdot|$ is a map satisfying
\begin{itemize}
  \item[(1)] $|k|=0$ if and only if $k=0$;
  \item[(2)] $|k_1k_2|=|k_1||k_2|$ holds for all $k_1,k_2\in\kk$;
  \item[(3)] and the triangle inequality $|k_1+k_2|\le |k_1|+|k_2|$ holds for all $k_1,k_2\in\kk$.
\end{itemize}
Then $\{\mathfrak{B}_r=\{a\in\kk\mid |a|<r\} \mid r\in\RR^+\}$ induces a standard topology $\frakU_{\kk}(0)$ on $\kk$
whose elements are called the {\defines neighborhoods} of $0\in\kk$.

Let $A$ be a $\mathcal{J}$-topological $\kk$-algebra whose dimension is finite and let $B_A=\{b_1,\ldots, b_n\}$ be a basis of $A$.
Then, naturally, we can define the Cauchy sequence by the $\mathcal{J}$-topology.
More precisely, a sequence $\{x_i\}_{i\in\NN}$ in $A$ is called a
{\defines $\mathcal{J}$-Cauchy sequence} if for any $U$, lying in $\frakU_A(0)$, containing some subset
$\sum_{i=1}^n \mathfrak{u}_ib_i$ of $A$ with $\mathfrak{u}_i\in\frakU_{\kk}(0)$ ($1\le i\le n$),
there is $m\in\NN$ such that $x_s-x_t\in U$ holds for all $s,t\ge m$.
Two $\mathcal{J}$-Cauchy sequences $\{x_i\}_{i\in\NN}$ and $\{y_i\}_{i\in\NN}$ are called
{\defines equivalent}, denoted by $\{x_i\}_{i\in\NN}$ $\sim$ $\{y_i\}_{i\in\NN}$,
if for any $U\in\frakU_A(0)$, there is an integer $m\in\NN$ such that $x_i-y_i\in U$ holds for all $i\ge m$.
It is easy to see that ``$\sim$'' is an equivalence relation.
We use $[\{x_i\}_{i\in\NN}]$ to denote the equivalence class containing $\{x_i\}_{i\in\NN}$,
and use $\mathfrak{C}_{\mathcal{J}}(A)$ to denote the set of all equivalence classes of $\mathcal{J}$-Cauchy sequences.
Then we have three families of $A$-homomorphisms:
\begin{itemize}
  \item[(1)] $(\varphi_{ji}: A/A_j \to A/A_i)_{j\ge i}$, where
    all $\varphi_{ji}$ are naturally induced by $A_i\supseteq A_j$;
  \item[(2)] $(p_i:\mathfrak{C}_{\mathcal{J}}(A) \to A/A_i)_{i\in\NN}$,
    where $p_i(x_0, \ldots, x_{i-1}, x_i, x_{i+1}, \ldots ) = x_i+A_i$
    ($p_i$ is called the $i$-th {\defines projection});
  \item[(3)] $(u_i: A/A_i \to \mathfrak{C}_{\mathcal{J}}(A))_{i\in\NN},$
    where $u_i(a+A_i) = (0,\ldots, \mathop{0}\limits^{i-1}, a, \mathop{0}\limits^{i+1}, 0 \ldots)$
    ($u_i$ is called the $i$-th {\defines injection}).
\end{itemize}
Let $\mathcal{X}$ be the category whose object set is $\{A/A_i\mid i\in\NN\}\cup\{\mathfrak{C}_{\mathcal{J}}(A)\}$
and morphism set is the collection of all $A$-homomorphisms as above.
Then we obtain the following commutative diagram
    \[ \xymatrix{
        \mathfrak{C}_{\mathcal{J}}(A) \ar@{<--}[rr]^{u_h}_{(\exists!)} & & A/A_h \\
        & A/A_i \ar@{<-}[lu]^{p_i}
                \ar@{<-}[ru]_{\varphi_{hi}}
                \ar@{<-}[d]^(0.35){\varphi_{ji}}_(0.35){(i\le j)}
        & \\
        & A/A_j. \ar@{<-}@/^1pc/[luu]^{p_j}
                \ar@{<-}@/_1pc/[ruu]_{\varphi_{hj}}
        &
      }\]
It follows from the above construction that the following proposition holds.

\begin{proposition}[{cf. \cite[Chapter 10, page 103]{AM2018}}]
Using the notations as above, we have
\[\underleftarrow{\lim} A/A_i \cong \mathfrak{C}_{\mathcal{J}}(A). \]
\end{proposition}

We write $\w{A}:=\mathfrak{C}_{\mathcal{J}}(A)$ and call it the {\defines completion} of $A$.
We say that $A$ is {\defines complete} if $\w{A}=A$.
In particular, if $A=\kk$, then the descending chain $\mathcal{J}:$
\[ A_0=\kk \succeq A_1=0 \]
induces a $\mathcal{J}$-topology
\[\frakU_{\kk}(0) = \{ \text{the neighborhood of } 0 \} \]
of $\kk$. In this case, the $\mathcal{J}$-Cauchy sequence coincides with the usual Cauchy sequence.

\begin{proposition} \label{prop:compl1}
Let $A$ be a basic finite-dimensional $\kk$-algebra and let $\mathcal{J}$ be the descending chain
\[ A_0 = A = \rad^0 A \succeq  A_1=\rad A  \succeq A_2=\rad^2 A \succeq \cdots\checks{.}  \]
Then $A$ is complete {\rm(}in the sense of $\mathcal{J}$-topology{\rm)} if and only if $\kk$ is complete.
\end{proposition}

\begin{proof}
Let $A$ be a basic finite-dimensional $\kk$-algebra. Then, by Theorem \ref{Gab's thm}\point,
there is a finite quiver $\Q$ and an ideal $\I$ of $\kk\Q$ such that
\[A\cong \kk\Q/\I = \bigoplus_{l\in\NN} \kk\Q_l.\]
Thus, up to isomorphism, each element $a\in A$ can be written as $\sum_{j=1}^{n} k_j\wp_j$,
where $n$ is the dimension of $A$, $k_u\in\kk$ and $\wp_u$ is a path on $\Q$.

Assume that $\kk$ is complete. Since $A$ is finite-dimensional,
we have $\rad^lA = \mathrm{span}_{\kk}\{\Q_i \mid i\ge l\}$.
Thus, $\rad^{L+1} A=0$, where $L=\max_{\wp\in\Q_{\ge 0}} \ell(\wp)$, \checks{i.e.,}
\begin{center}
  $\mathcal{J}=$
  \ \
  $A \succeq  \rad A  \succeq \rad^2 A \succeq \cdots \rad^L A \succeq 0 \succeq 0 \succeq \cdots$.
\end{center}
Let $\{x_i = \sum_{j=1}^n k_{ij}\wp_j\}_{i\in\NN}$ be a $\mathcal{J}$-Cauchy sequence in $A$.
Take
\[U=\bigg\{\sum\nolimits_{\ell(\wp)=L} k_{\wp}\wp \mid k_{\wp}
\text{ lie in some neighborhood in } \frakU_{\kk}(0)\bigg\} \ \ (\supsetneq \rad^{L+1}A=0).\]
Then, there is $N(U)\in\NN$ such that
\[x_s-x_t = \sum_{j=1}^n (k_{sj}-k_{tj})\wp_j \in \rad^L A \text{ holds for all } s,t\ge N(U).\]
Thus, $k_{sj}-k_{tj}$ lies in some neighborhood in $\frakU_{\kk}(0)$, and so, for all $i$, $\{k_{ij}\}_{i\in\NN}$ is a Cauchy sequence in $\kk$.
Then it is clear that $A$ is complete.

Conversely, if $A$ is complete, we assume that $\kk$ is not complete,
and $\w{\kk}$ is the completion of $\kk$. Then we have a natural $\kk$-linear embedding
$\mathfrak{e}:\kk\to\w{\kk}$ sending $k\in\kk$ to $\{k_i\}_{i\in\NN}$, where $k_1=k_2=\cdots=k$.
Then there is a Cauchy sequence $\{x_i\}_{i\in\NN}\in\w{\kk}\backslash \mathfrak{e}(\kk)$.
Consider the sequence $\{x_i\cdot\wp\}_{i\in\NN}$ in $A$, where $\wp\in\rad^L A$ is a path of length $L$.
Then $\{x_i\cdot\wp\}_{i\in\NN}$ is a $\mathcal{J}$-Cauchy sequence in $A$. However, we have
$\{x_i\cdot\wp\}_{i\in\NN} \in \w{A}\backslash A $ in this case, which contradicts that $A$ is complete.
\end{proof}

\subsection{\texorpdfstring{The total order of $\kk$-algebras}{}}

Recall that a field $\kk$  equipped with a total order $\preceq$ is an {\defines ordered field}
if it satisfies the following four conditions:
\begin{itemize}
  \item[(1)] for any $a,b\in\kk$, either $a\preceq b$, $b\preceq a$ or $a=b$ holds;
  \item[(2)] if $a\preceq b$, $b\preceq c$, then $a\preceq c$;
  \item[(3)] if $a\preceq b$, then $a+c\preceq b+c$ for all $c\in\kk$;
  \item[(4)] if $a\preceq b$ and $0\preceq c$, then $ac\preceq bc$.
\end{itemize}
In order to give the definition of integration defined on a finite-dimensional $\kk$-algebra $\itLamb$,
we need to assume that $\kk$ is a field with the total order $\preceq$.
However, it is well-known that $\kk$ might not always be an ordered field,
as the case for $\kk$ being the complex field $\mathbb{C}$.
Interestingly, for our purposes, the existence of such a total order is not a prerequisite.
We only require that the finite-dimensional $\kk$-algebra involved in our study,
encompasses certain partially ordered subsets. Specifically, the subset $\II_{\itLamb}$ outlined in
Subsection \ref{subsect:normed alg.3}\point~is sufficient. For the sake of simplicity,
we assume that $\kk$ is fully ordered, although this assumption does not sacrifice generality.
This simplification aids in our definition of integration within the context of category theory.

\begin{remark} \rm
We provide a remark to show that if $\kk$ is total ordered,
then any finite-dimensional $\kk$-algebra $\itLamb$ can be endowed with a total order.
Let $B_{\itLamb}=\{b_i \mid 1\le i\le n\}$ be a $\kk$-basis of $\itLamb$.
If $B_{\itLamb}$ is totally ordered (assuming $b_i\preceq b_j$ if and only if $i\le j$),
then we can define a total order for $\itLamb$ as follows.

{\bf Step 1.}  For any two elements $a, a'\in \itLamb$,
we define $a\prec_p a'$ if and only if $\varphi(a)<\varphi(a')$,
where $\varphi$ is a map $\varphi: \itLamb \to \RR^{\ge 0}$
(for example, $\varphi$ is the norm $\Vert\cdot\Vert_p$ defined in Section \ref{sect:normed alg}\point).

{\bf Step 2.} Assume $a=\sum_{i=1}^m k_ib_i$ and $a'=\sum_{i=1}^m k_i'b_i$ ($0\le m\le n$) such that $k_i=k_i'$ holds for all $i<m$.
If $\varphi(a) = \varphi(a')$, then we define $a\preceq_p a'$ if and only if $k_m\preceq k_m'$.
\end{remark}

\section{\texorpdfstring{Normed $\kk$-algebras}{}} \label{sect:normed alg}

In this section, let $\itLamb$ be a finite-dimensional $\kk$-algebra with a $\kk$-basis $B_{\itLamb}=\{b_i\mid 1\le i\le n\}$.
Then any element $a\in\itLamb$ is of the form $a = \sum_{i=1}^n k_ib_i$.
In this section, we define some algebraic structures on $\itLamb$.

\subsection{\texorpdfstring{Norms of $\kk$-algebras}{}} \label{subsect:normed alg.1}
~{For a map $\norm{}: B_{\itLamb} \to \RR^+$}
and any $p\ge 1$, we have $||\cdot||_p: \itLamb\to\RR^{\ge 0}$ as the function
\begin{align}\label{formula:norm}
  \Vert a \Vert_p = \Big\Vert \sum_{i=1}^n k_ib_i \Big\Vert_p
  := \big( (|k_1|\norm{}(b_1))^p + \cdots + (|k_n|\norm{}(b_n))^p \big)^{\frac{1}{p}}.
\end{align}

\begin{proposition} \label{prop:triple}
Any triple $(\itLamb, \norm{}, \Vert\cdot\Vert_p)$ {\rm(}=$\itLamb$ for short{\rm)} is a normed $\kk$-vector space.
\end{proposition}

\begin{proof}
First of all, for any $a=\sum_{i=1}^n k_ib_i\in\itLamb$, we have $\Vert a \Vert_p\ge 0$
because $\norm{}(b_i)>0$ and $|k_i|\ge 0$ ($1\le i\le n$).
In particular, if $\Vert a \Vert_p=0$, then
\begin{center}
  $(|k_1|\norm{}(b_1))^p + \cdots + (|k_n|\norm{}(b_n))^p=0$.
\end{center}
Since $|k_i|\norm{}(b_i)\ge 0$ and $\norm{}(b_i)>0$ hold for all $1\le i\le n$,
we obtain $|k_i|\norm{}(b_i)=0$, and so $k_i=0$.
Thus, $a=\sum_{i=1}^n 0b_i=0$. Then it is easy to see that $\Vert a\Vert_p=0$ if and only if $a=0$.

Next, for any $k\in \kk$ and $a=\sum_{i=1}^n k_ib_i\in\itLamb$, we have
\begin{align}
      \Vert ka\Vert_p
  & = \Vert k(k_1b_1+\cdots+k_nb_n)\Vert_p  \nonumber \\
  & = \Big(\sum_{i=1}^n (|kk_i|\norm{}(b_i))^p\Big)^{\frac{1}{p}}
    = \Big(\sum_{i=1}^n |k|^p(|k_i|\norm{}(b_i))^p\Big)^{\frac{1}{p}} \nonumber \\
  & = |k|\Big(\sum_{i=1}^n(|k_i|\norm{}(b_i))^p\Big)^{\frac{1}{p}} = |k|\cdot \Vert a\Vert_p \nonumber.
\end{align}

Finally, we prove the triangle inequality $\Vert a+a'\Vert_p\le \Vert a\Vert_p + \Vert a'\Vert_p$
for arbitrary two elements $a=\sum_{i=1}^n k_ib_i$ and $a'=\sum_{i=1}^n k_i'b_i$.
It can be induced by the discrete Minkowski inequality $(\sum_{i=1}^n x_i^p)^{\frac{1}{p}} + (\sum_{i=1}^n y_i^p)^{\frac{1}{p}}
\ge (\sum_{i=1}^n (x_i+y_i)^p)^{\frac{1}{p}}$ as follows:
\begin{align}
      \Vert a \Vert_p + \Vert a' \Vert_p
 &  = \Big(\sum_{i=1}^n(|k_i| \norm{}(b_i))^p \Big)^{\frac{1}{p}} +
      \Big(\sum_{i=1}^n(|k_i'|\norm{}(b_i))^p \Big)^{\frac{1}{p}} \nonumber \\
 &\ge \Big(\sum_{i=1}^n(|k_i| \norm{}(b_i)+k_i'\norm{}(b_i))^p \Big)^{\frac{1}{p}} \nonumber \\
 &  = \Big(\sum_{i=1}^n(|k_i +k_i'|\norm{}(b_i))^p \Big)^{\frac{1}{p}} = \Vert a+a'\Vert_p \nonumber.
\end{align}
Therefore, $(\itLamb, \norm{}, \Vert\cdot\Vert_p)$ is a normed space.
\end{proof}

\begin{definition} \rm
A {\defines normed $\kk$-algebra} is a triple $(\itLamb, \norm{}, \Vert\cdot\Vert_p)$,
where $\norm{}: B_{\itLamb} \to \RR^+$ and $\Vert\cdot\Vert_p: \itLamb\to\RR^{\ge0}$
are called the {\defines normed basis function} and {\defines norm} of $\itLamb$, respectively.
\end{definition}

\subsection{\texorpdfstring{Completions of normed $\kk$-algebras}{}} \label{subsect:normed alg.2}

We can define open neighborhoods $B(0,r)$ of $0$ for any normed $\kk$-algebra $(\itLamb, \norm{}, \Vert\cdot\Vert_p)$ by
\[ B(0,r) := \{ a\in \itLamb \mid \Vert a \Vert_p < r \}.  \]
Let $\frakU_{\itLamb}^B(0)$ be the class of all subsets $U$ of $\itLamb$ satisfying the following conditions.
\begin{itemize}
  \item[(1)] $U$ is the intersection of a finite number of $B(0,r)$;
  \item[(2)] $U$ is the union of any number of $B(0,r)$.
\end{itemize}
Then $\frakU_{\itLamb}^B(0)$ is a topology defined on $\itLamb$, called the $\Vert\cdot\Vert_p$-topology
and we can define Cauchy sequences called $\Vert\cdot\Vert_p$-Cauchy sequences by the above topology.

Recall that $\itLamb$ has a $\mathcal{J}$-topology $\frakU_{\itLamb}(0)$ given by the descending chain
$$\itLamb=\rad^0\itLamb \succeq \rad^1\itLamb \succeq \rad^2\itLamb \succeq \cdots.$$
Thus, we obtain two completions $\w{\itLamb}{}^B$ and $\w{\itLamb}$ by the $\Vert\cdot\Vert_p$-topology and the $\mathcal{J}$-topology, respectively.
The following lemma establishes the relation between $\w{\itLamb}{}^B$ and $\w{\itLamb}$ in the case for $\kk$ being complete.

\begin{proposition} \label{prop:compl2}
Assume that $\kk$ is complete.
Let $\itLamb=(\itLamb, \norm{}, \Vert\cdot\Vert_p)$ be an $n$-dimensional normed $\kk$-algebra with the $\mathcal{J}$-topology
$\frakU_{\itLamb}(0)$ given by $\itLamb=\rad^0\itLamb \succeq \rad^1\itLamb \succeq \rad^2\itLamb \succeq \cdots$
{\rm(}$\Vert\cdot\Vert_p$ is the norm defined on $\itLamb$ given in Proposition \ref{prop:triple}\point{\rm)}.
Then $\w{\itLamb}{}^B = \w{\itLamb}$.
\end{proposition}

\begin{proof}
Similar to Proposition \ref{prop:compl1}\point~we can show that $\w{\itLamb}{}^B=\itLamb$ (i.e., $\itLamb$ is complete) if and only if $\w{\kk}=\kk$.
By using Proposition \ref{prop:compl1}\point~again, we have that $\w{\itLamb}=\itLamb$ if and only if $\w{\kk}=\kk$. Then
$\w{\kk}=\kk$ if and only if $\w{\itLamb}{}^B = \itLamb = \w{\itLamb}$.
Equivalently,
\[\w{\itLamb}{}^B = \bigg(\w{\sum_{i=1}^n\kk b_i}\bigg)^B = \sum_{i=1}^n\w{\kk} b_i = \w{\sum_{i=1}^n\kk b_i} = \w{\itLamb}. \]
\end{proof}

\begin{remark} \rm
\begin{itemize}
\item[]
\item[(1)] Note that the norms defined on $\itLamb$ is not unique. In Section \ref{sect:normed mod}\point,
we will introduce normed $\itLamb$-modules $N$ over any finite-dimensional normed $\kk$-algebra $\itLamb$.
In this case, we need a homomorphism $\tau:\itLamb\to\itLamb'$ between two
finite-dimensional normed $\kk$-algebras $\itLamb$ and $\itLamb'$,
and the norms $\Vert\cdot\Vert$ and $\Vert\cdot\Vert'$ respectively defined on $\itLamb$ and $\itLamb'$
may not necessarily be of the form $\Vert\cdot\Vert_p$.
\item[(2)] If $\itLamb=\kk$ and $\norm{}(1)=1$, then the norm $\Vert\cdot\Vert_p$ given in Proposition \ref{prop:triple}\point~
is the norm $|\cdot|$, i.e, $\Vert a\Vert_p = (|a|^p)^{\frac{1}{p}} = |a|$.
\end{itemize}
\end{remark}

\subsection{\texorpdfstring{Elementary simple functions}{}}
\label{subsect:normed alg.3}

Let $\II$ be a subset of $\kk$. Denote $\II_{\itLamb}$ by the subset
  \[ \bigg\{ \sum_{i=1}^nk_ib_i\mid k_i\in\II \bigg\}
    \mathop{\longleftrightarrow}\limits^{1-1}
    \prod_{i=1}^n(\II\times\{b_i\})\]
of $\itLamb$. A {\defines function defined on $\II_{\itLamb}$} is a map $f: \II_{\itLamb}\to\kk$.
Since $(\itLamb, \norm{}, \Vert\cdot\Vert_p)$ is a normed space,
$\itLamb$ is also a topological space induced by the norm $\Vert\cdot\Vert_p$, and so is $\II_{\itLamb}$.
Thus, we can define an open set for every subset of $\itLamb$, including $\II_{\itLamb}$.
The function $f$ is said to be {\defines continuous} if the preimage of any open subset of $\kk$ is an open set of $\II_{\itLamb}$.

Let $\II:=[a,b]_{\kk}$ be a fully ordered subset of $\kk$ whose minimal element and maximal element are $a$ and $b$, respectively.
In our paper, we assume that $\kk$ and $[a,b]_{\kk}$ are infinite sets
and consider only the case for $\II=[a,b]_{\kk}$ with $a\prec b$
such that there exists an element $\xi$ with $a\prec \xi\prec b$ and the order-preserving bijections
$\kappa_a:\II\to[a,\xi]_{\kk}$ and $\kappa_b:\II\to[\xi,b]_{\kk}$ exist
(for example, the case of the cardinal number of $\II$ is either $\aleph_0$ or $\aleph_1$).

An {\defines elementary simple function} on $\II_{\itLamb}$ is a finite sum
\[\sum_{i=1}^t k_i\id_{I_i},\]
where
\begin{itemize}
  \item[(1)] for any $1\le i\le t$, $k_i\in\kk$;
  \item[(2)] $I_i=I_{i1}\times \cdots \times I_{in}$, and, for any $1\le j\le n$, $I_{ij}$ is a subset of $\II$ which is one of the following forms
    \begin{itemize}
      \item[(a)] $(c_{ij},d_{ij})_{\kk}:=\{k\in\kk \mid c_{ij}\prec k\prec d_{ij}\}$,
      \item[(b)] $[c_{ij},d_{ij})_{\kk}:=\{k\in\kk \mid c_{ij}\preceq k\prec d_{ij}\}$,
      \item[(c)] $(c_{ij},d_{ij}]_{\kk}:=\{k\in\kk \mid c_{ij}\prec k\preceq d_{ij}\}$,
      \item[(d)] $[c_{ij},d_{ij}]_{\kk}:=\{k\in\kk \mid c_{ij}\preceq k\preceq d_{ij}\}$,
    \end{itemize}
    where $a\preceq c_{ij} \prec d_{ij} \preceq b$;
  \item[(3)] and $\id_{I_i}$ is the function $I_i \to \{1\}$ such that $I_i\cap I_j=\varnothing$ holds for all $1\le i\ne j\le t$.
\end{itemize}
We denote $\bfS(\II_{\itLamb})$ by the set of all elementary simple functions.
Then $\bfS(\II_{\itLamb})$ is a $\kk$-vector space,
and $\bfS(\II_{\itLamb})$ induces the direct sum $\bfS(\II_{\itLamb})^{\oplus 2^n}$ whose element can be seen as the sequence
\[\bigg\{f_{(\delta_1, \ldots, \delta_n)}\bigg(\sum_{i=1}^n k_ib_i\bigg)\bigg\}_{
    (\delta_1, \ldots, \delta_{n})\in \{a,b\}\times\cdots\times\{a,b\}
    }
   =: \pmb{f}(k_1,\ldots,k_{n}), \]
$\sum_{i=1}^n k_ib_i$ is written as $(k_1,\ldots,k_n)$ since $\{b_1\mid 1\le i\le n\}=B_{\itLamb}$
is the $\kk$-basis of $\itLamb$. Then we can characterize $\bfS(\II_{\itLamb})$ together with two further pieces of data:
the function $\id_{\II_{\itLamb}}: \II_{\itLamb}\to\{1\}$, 
and the map
\begin{align}\label{map:gamma}
  \gamma_{\xi}: \bfS(\II_{\itLamb})^{\oplus 2^n} \to \bfS(\II_{\itLamb}),
\end{align}
called the {\defines juxtaposition map}, sending $\pmb{f}$ to the function
\[\gamma_{\xi}(\pmb{f}) (k_1,\ldots,k_n)
 = \sum_{(\delta_1, \ldots, \delta_n)}\id_{\kappa_{\delta_1}(\II)\times \cdots\times
   \kappa_{\delta_n}(\II)}\cdot f_{(\delta_1, \ldots, \delta_n)} (\kappa_{\delta_1}^{-1}(k_1),\ldots,\kappa_{\delta_n}^{-1}(k_n)),\]
\begin{center}
  ($k_1\ne\xi$, $\ldots$, $k_n\ne\xi$),
\end{center}
where $\xi$ is an element with $a\prec \xi\prec b$ such that the order-preserving bijections
\begin{center}
$\kappa_a:\II\to[a,\xi]_{\kk}$ and $\kappa_b:\II\to[\xi,b]_{\kk}$
\end{center}
exist.

\begin{example} \rm
(1) Take $\itLamb$ be the $\kk$-algebra whose dimension is $2$, and assume that $\{b_1, b_2\}$ is a basis of $\itLamb$.
Then $\II_{\itLamb} \cong_{\kk} [a,b]_{\kk} b_1 \times [a,b]_{\kk} b_2$.
For any element
\[\pmb{f}=(f_{(a,a)}, f_{(b,a)}, f_{(a,b)}, f_{(b,b)}) \in \bfS(\II_{\itLamb})^{\oplus 4},\]
where $f_{(\delta_1,\delta_2)}: \II_{\itLamb} \to \kk$ is a function in $\bfS(\II_{\itLamb})$
sending each $k_1b_1+k_2b_2$ to the element $f_{(\delta_1,\delta_2)}(k_1,k_2)$ in $\kk$,
$(\delta_1,\delta_2) \in \{a,b\}\times\{a,b\} = \{(a,a), (b,a), (a,b), (b,b)\}$,
$\gamma_{\xi}$ juxtaposes $f_{(a,a)}$, $f_{(b,a)}$, $f_{(a,b)}$ and $f_{(b,b)}$ into a new function
\begin{align*}
   & \gamma_{\xi}(f_{(a,a)}, f_{(b,a)}, f_{(a,b)}, f_{(b,b)})(k_1,k_2) \\
=\ & \tilde{f}_{(a,a)}(k_1,k_2) + \tilde{f}_{(b,a)}(k_1,k_2)
 + \tilde{f}_{(a,b)}(k_1,k_2) + \tilde{f}_{(b,b)}(k_1,k_2)
\end{align*}
as shown in Figure \ref{fig:jux.map}\point, where
\begin{align*}
  \tilde{f}_{(a,a)}(k_1,k_2)
& = \id_{[a,\xi)\times [a,\xi)}\cdot f_{(a,a)}(\kappa_a^{-1}(k_1), \kappa_a^{-1}(k_2)), \\
  \tilde{f}_{(b,a)}(k_1,k_2)
& = \id_{(\xi,b]\times [a,\xi)}\cdot f_{(b,a)}(\kappa_b^{-1}(k_1), \kappa_a^{-1}(k_2)), \\
  \tilde{f}_{(a,b)}(k_1,k_2)
& = \id_{[a,\xi)\times (\xi,b]}\cdot f_{(a,b)}(\kappa_a^{-1}(k_1), \kappa_b^{-1}(k_2)), \\
  \tilde{f}_{(b,b)}(k_1,k_2)
& = \id_{(\xi,b]\times (\xi,b]}\cdot f_{(b,b)}(\kappa_b^{-1}(k_1), \kappa_b^{-1}(k_2)). \\
\end{align*}

\begin{figure}[htbp]
\begin{center}
\definecolor{pistachio}{rgb}{0.75,1,0.75}
\begin{tikzpicture}[scale=1.2]
\draw [pistachio][line width=8pt][shift={(0, 3)}] (-3,0) -- (3,0);
\draw [pistachio][line width=8pt][shift={(0,-3)}] (-3,0) -- (3,0);
\draw [pistachio][line width=8pt][shift={( 3,0)}] (0,-3) -- (0,3);
\draw [pistachio][line width=8pt][shift={(-3,0)}] (0,-3) -- (0,3);
\fill [pistachio][shift={(-3,-3)}]
     (-1.25,-1.25) -- ( 1.25,-1.25) -- ( 1.25, 1.25) -- (-1.25, 1.25) -- (-1.25,-1.25);
\fill [pistachio][shift={( 3,-3)}]
     (-1.25,-1.25) -- ( 1.25,-1.25) -- ( 1.25, 1.25) -- (-1.25, 1.25) -- (-1.25,-1.25);
\fill [pistachio][shift={( 3, 3)}]
     (-1.25,-1.25) -- ( 1.25,-1.25) -- ( 1.25, 1.25) -- (-1.25, 1.25) -- (-1.25,-1.25);
\fill [pistachio][shift={(-3, 3)}]
     (-1.25,-1.25) -- ( 1.25,-1.25) -- ( 1.25, 1.25) -- (-1.25, 1.25) -- (-1.25,-1.25);
\fill [pistachio][shift={(-3,-3)}]
     (-1,-0.3) -- (-1,-1) -- (-2.5,-0.66)
     node[left,opacity=0.25]{$\pmb{\bfS(\II_{\itLamb})^{\oplus 4}}$}
     node[left,black]{$\bfS(\II_{\itLamb})^{\oplus 4}$};
\fill [pink]
     (-1.25,-1.25) -- ( 1.25,-1.25) -- ( 1.25, 1.25) -- (-1.25, 1.25) -- (-1.25,-1.25);
\fill [pink] (1,-0.3) -- (1,-1) -- (5,-0.66)
     node[right,opacity=0.25]{$\pmb{\bfS(\II_A)}$}
     node[right,black]{$\bfS(\II_A)$};
\fill[left color=   red!37, right color=   red!17] (-1  ,-1  ) -- (-0.3,-1  ) -- (-0.3,-0.3) -- (-1  ,-0.3);
\fill[ top color=  blue!37,bottom color=  blue!17] (-0.3,-1  ) -- ( 1  ,-1  ) -- ( 1  ,-0.3) -- (-0.3,-0.3);
\fill[ top color=orange!37,bottom color=orange!17] (-0.3,-0.3) -- ( 1  ,-0.3) -- ( 1  , 1  ) -- (-0.3, 1  );
\fill[left color=violet!37, right color=violet!17] (-0.3,-0.3) -- (-1  ,-0.3) -- (-1  , 1  ) -- (-0.3, 1  );
\draw [black][line width=1pt]
     (-1,-1) node[ left]{$a$} node[below]{$a$}
  -- ( 1,-1) node[right]{$b$}
  -- ( 1, 1)
  -- (-1, 1) node[above]{$b$}
  -- (-1,-1);
\draw [black][line width=1pt][shift={(-3, 3)}][fill=violet!25]
     (-1,-1) node[ left]{$a$} node[below]{$a$}
  -- ( 1,-1) node[right]{$b$}
  -- ( 1, 1)
  -- (-1, 1) node[above]{$b$}
  -- (-1,-1);
\draw [black][line width=1pt][shift={( 3, 3)}][fill=orange!25]
     (-1,-1) node[ left]{$a$} node[below]{$a$}
  -- ( 1,-1) node[right]{$b$}
  -- ( 1, 1)
  -- (-1, 1) node[above]{$b$}
  -- (-1,-1);
\draw [black][line width=1pt][shift={( 3,-3)}][fill=blue!25]
     (-1,-1) node[ left]{$a$} node[below]{$a$}
  -- ( 1,-1) node[right]{$b$}
  -- ( 1, 1)
  -- (-1, 1) node[above]{$b$}
  -- (-1,-1);
\draw [black][line width=1pt][shift={(-3,-3)}][fill=red!25]
     (-1,-1) node[ left]{$a$} node[below]{$a$}
  -- ( 1,-1) node[right]{$b$}
  -- ( 1, 1)
  -- (-1, 1) node[above]{$b$}
  -- (-1,-1);
\draw[shift={( 0, 3)}][<->] (-1.8,0)--(1.8,0);
\draw[shift={( 0, 3)}]      (0,0) node[above]{$\oplus$};
\draw[shift={( 0,-3)}][<->] (-1.8,0)--(1.8,0);
\draw[shift={( 0,-3)}]      (0,0) node[below]{$\oplus$};
\draw[shift={(-3, 0)}][<->] (0,-1.8)--(0,1.8);
\draw[shift={(-3, 0)}]      (0,0) node[ left]{$\oplus$};
\draw[shift={( 3, 0)}][<->] (0,-1.8)--(0,1.8);
\draw[shift={( 3, 0)}]      (0,0) node[right]{$\oplus$};
\draw[white][line width=1pt] (-0.3,-1)--(-0.3, 1);
\draw[white][line width=1pt] (-1,-0.3)--( 1,-0.3);
\draw[cyan]
  (-0.3,-1) node{$\bullet$} node[below]{$\xi$}
  (-1,-0.3) node{$\bullet$} node[ left]{$\xi$};
\draw[red] (-2,-2) -- (-0.3,-0.3);
\draw[red] (-2.5,-2.0) node[above]{$\kappa_a$};
\draw[red] (-2.0,-2.5) node[right]{$\kappa_a$};
\draw[red][postaction={on each segment={mid arrow = red}}] (-4,-2) -- (-1  ,-0.3);
\draw[red][postaction={on each segment={mid arrow = red}}] (-2,-4) -- (-0.3,-1  );
\draw[red][shift={(-3,-3)}] (0,0) node{$f_{(a,a)}$};
\draw[red][shift={(-0.6,-0.6)}] (0,0) node{\scriptsize$\tilde{f}_{(a,a)}$};
\draw[blue] ( 2,-2) -- (-0.3,-0.3);
\draw[blue] ( 2.5,-2.0) node[above]{$\kappa_b$};
\draw[blue] ( 2.0,-2.5) node[ left]{$\kappa_a$};
\draw[blue][postaction={on each segment={mid arrow = blue}}] ( 4,-2) -- ( 1  ,-0.3);
\draw[blue][postaction={on each segment={mid arrow = blue}}] ( 2,-4) -- (-0.3,-1  );
\draw[blue][shift={( 3,-3)}] (0,0) node{$f_{(b,a)}$};
\draw[blue][shift={( 0.3,-0.6)}] (0,0) node{\scriptsize$\tilde{f}_{(b,a)}$};
\draw[orange] ( 2, 2) -- (-0.3,-0.3);
\draw[orange] ( 2.5, 2.0) node[below]{$\kappa_b$};
\draw[orange] ( 2.0, 2.5) node[ left]{$\kappa_b$};
\draw[orange][postaction={on each segment={mid arrow = orange}}] ( 4, 2) -- ( 1  ,-0.3);
\draw[orange][postaction={on each segment={mid arrow = orange}}] ( 2, 4) -- (-0.3, 1  );
\draw[orange][shift={( 3, 3)}] (0,0) node{$f_{(b,b)}$};
\draw[orange][shift={( 0.3, 0.3)}] (0,0) node{\scriptsize$\tilde{f}_{(b,b)}$};
\draw[violet] (-2, 2) -- (-0.3,-0.3);
\draw[violet] (-2.5, 2.0) node[below]{$\kappa_a$};
\draw[violet] (-2.0, 2.5) node[right]{$\kappa_b$};
\draw[violet][postaction={on each segment={mid arrow = violet}}] (-4, 2) -- (-1  ,-0.3);
\draw[violet][postaction={on each segment={mid arrow = violet}}] (-2, 4) -- (-0.3, 1  );
\draw[violet][shift={(-3, 3)}] (0,0) node{$f_{(a,b)}$};
\draw[violet][shift={(-0.6, 0.3)}] (0,0) node{\scriptsize$\tilde{f}_{(a,b)}$};
\end{tikzpicture}
\caption{Juxtaposition map}
\label{fig:jux.map}
\end{center}
\end{figure}

(2) This example is used to establish the relation between Banach space and Lebesgue intersections in \cite{Lei2023FA}.
Take $\kk=\RR$, $\II=[0,1]$, $\xi=\frac{1}{2}$, $\itLamb=\RR$
and the order-preserving bijections $\kappa_{0}:\II=[0,1]\to \kk=\RR$ and $\kappa_{1}:\II=[0,1]\to \kk=\RR$
are given by $x\mapsto\frac{x}{2}$ and $\frac{1+x}{2}$, respectively.
Then $\bfS(\II_{\RR}) = \bfS([0,1])$ is a normed space together with two further pieces of data:
the function $\id_{[0,1]}:[0,1]\to\{1\}$ and the juxtaposition map
\[\gamma_{\frac{1}{2}}: \bfS([0,1])\oplus\bfS([0,1]) \to \bfS([0,1])\]
sending $(f_1, f_2)$ to the following function
\begin{align}
     \gamma_{\frac{1}{2}}(f_1,f_2)(x)
 & = \id_{\kappa_{0}([0,1))}\cdot f_1({\kappa_{0}^{-1}(x)})
   + \id_{\kappa_{1}((0,1])}\cdot f_1({\kappa_{1}^{-1}(x)})
     \nonumber \\
 & = \begin{cases}
       f_1(2x)   & x\in \kappa_{0}([0,1)) = [0,\tfrac{1}{2});\\
       f_2(2x-1) & x\in \kappa_{1}((0,1]) = (\tfrac{1}{2},1].
     \end{cases}
     \nonumber
\end{align}
\end{example}

\begin{lemma}\label{lemm:gamma linear}
The map $\gamma_{\xi}$ is a $\kk$-linear map.
\end{lemma}

\begin{proof}
Take $a,b\in\kk$, $f,g\in\bfS(\II_{\itLamb})$ and
let $(k_i)_i$, $\id$ and $(\delta_i)_i$ be
the element $(k_1,\ldots,k_n)$ in $\bfS(\II_{\itLamb})^{\oplus 2^n}$,
the identity function $\id_{\kappa_{\delta_1}(\II)\times\cdots\times\kappa_{\delta_n}(\II)}$
and the $n$-multiple $(\delta_1\times\cdots\times\delta_n)$, respectively.
Then
\begin{align}
  \gamma_{\xi}(af+bg)((k_i)_i)
& = \sum_{(\delta_i)_i}\id\cdot (af+bg)_{(\delta_i)_i} ((\kappa_{\delta_i}^{-1}(k_i))_{i}) \nonumber \\
& = \sum_{(\delta_i)_i} \big( \id \cdot af_{(\delta_i)_i} ((\kappa_{\delta_i}^{-1}(k_i))_{i})
    + \id \cdot bg_{(\delta_i)_i} ((\kappa_{\delta_i}^{-1}(k_i))_{i}) \big) \nonumber \\
& = a \sum_{(\delta_i)_i} \id \cdot f_{(\delta_i)_i} ((\kappa_{\delta_i}^{-1}(k_i))_{i})
    + b \sum_{(\delta_i)_i} \id \cdot g_{(\delta_i)_i} ((\kappa_{\delta_i}^{-1}(k_i))_{i}) \nonumber \\
& = a\gamma_{\xi}(f)((k_i)_i) + b\gamma_{\xi}(g)((k_i)_i).   \nonumber
\end{align}
Thus, $\gamma_{\xi}$ is a $\kk$-linear map.
\end{proof}


\section{\texorpdfstring{Normed modules over $\kk$-algebras}{}} \label{sect:normed mod}

Let $\II$ be a subset of the field $\kk=(\kk,\preceq)$ with totally ordered $\preceq$.
Then $\II$ is also a totally ordered set. For simplicity,
we denote by $[x,y]_{\kk}$ the set of all elements $k\in\kk$ with $x\preceq k\preceq y$, \checks{i.e.,}
\[[x,y]_{\kk} := \{ k\in\kk \mid x\preceq k\preceq y \}. \]
In particular, if $x=y$ then $[x,y]_{\kk}=\{x\}=\{y\}$ is a set containing only one element.

In this section, we introduce the category $\scrN^p$, which is used to explore the categorification of integration.

\subsection{\texorpdfstring{Norms of $\itLamb$-modules}{}} \label{subsect:normed mod.1}

Recall that a {\defines left $A$-module} (=$A$-module for short) over a $\kk$-algebra $A$ is a $\kk$-vector space $V$
with a $\kk$-linear map $h: A\to \End_{\kk}V$ sending $a$ to $h_a$.
Thus, $h$ provides a right action $A\times V \to V$,
$(a,v)\mapsto va:=h_a(v)$ which satisfies the following properties:
\begin{itemize}
  \item[(1)] $a(v+v')=av+av'$ for any $v,v'\in V$ and $a\in A$;
  \item[(2)] $(a+a')v=av+a'v$ for any $v\in V$ and $a,a'\in A$;
  \item[(3)] $a'(av)=(a'a)v$ for any $v\in V$ and $a,a'\in A$;
  \item[(4)] $1v=v$ for any $v\in V$;
  \item[(5)] $(ka)v=k(av)=a(kv)$ for any $v\in V$, $a\in A$ and $k\in\kk$.
\end{itemize}
Take $A=\itLamb$ to be the normed $\kk$-algebra whose norm $\Vert\cdot\Vert_p:\itLamb\to\RR^+$ given by (\ref{formula:norm}),
where the $\kk$-basis of $\itLamb$ is $B_{\itLamb} = \{b_i \mid 1\le i\le n=\dim_{\kk}\itLamb\}$.

\begin{definition} \rm \label{def:normed mod}
%
Let $\tau:\itLamb\to \kk$ be a homomorphism between two normed $\kk$-algebras
$(\itLamb, \Vert\cdot\Vert_p)$ and $(\kk, |\cdot|)$.
A {\defines $\tau$-normed $\itLamb$-module} is a $\itLamb$-module $M$ with a norm $\Vert\cdot\Vert: M\to\RR^{\ge 0}$ such that
\begin{align}\label{formula:normedmod}
  \Vert am\Vert = |\tau(a)| \cdot\Vert m\Vert \text{ holds for all } a\in\kk \text{ and } m\in M.
\end{align}
Thus, each normed $\itLamb$-module can be seen as a triple $(M, h, \Vert\cdot\Vert)$ of the $\kk$-vector space $M$,
the $\kk$-linear map $h: M \to \End_{\kk}M$ and a norm $\Vert\cdot\Vert: M \to \RR^{\ge 0}$.
For simplification, $\tau$-normed modules are called {\it normed modules}.
\end{definition}

The norms of $\itLamb$-modules yield the following fact.

\begin{fact} \rm \
\begin{itemize}
  \item[(1)] Note that $\Vert\cdot\Vert_p$ defined by (\ref{formula:norm}) is the norm of $\itLamb$ as a $\kk$-vector space.
      It is easy to see that $\itLamb$ is also a left $\itLamb$-module, called the {\defines regular module}, where the scalar multiplication
      is given by the multiplication $\itLamb \times \itLamb \to \itLamb, (a,x)\mapsto ax$ of $\itLamb$ as a finite-dimensional $\kk$-algebra.
      Thus, it is natural to ask whether $\Vert\cdot\Vert_p$ is a norm of $\itLamb$ as a $\itLamb$-module.
      Indeed, the norm of $\itLamb$ as a finite-dimensional $\kk$-algebra {may not be equal to} the norm $\Vert\cdot\Vert$ of $\itLamb$ as a regular module.
      However, if $\itLamb$ as the left $\itLamb$-module defined by
      \begin{align}\label{formula:Lambda}
        \itLamb \times \itLamb \to \itLamb, (a, x) \mapsto a \star x:= \tau(a)x,
      \end{align}
      where $\tau(a)x$ is defined by the scalar multiplication of $\itLamb$ as the $\kk$-vector space ${_\kk}\itLamb$,
      then, for any $x = \sum_{i=1}^n k_ib_i\in\itLamb$, we obtain
      \begin{align}
          \Vert a\star x\Vert_p
      & = \bigg\Vert \tau(a)\sum_{i=1}^n k_ib_i \bigg\Vert_p
        = \bigg(\sum_{i=1}^n |\tau(a)k_i|^p\norm{}(b_i)^p\bigg)^{\frac{1}{p}} \nonumber \\
      & = |\tau(a)|\bigg(\sum_{i=1}^n |k_i|^p\norm{}(b_i)^p\bigg)^{\frac{1}{p}}
        = |\tau(a)| \Vert x \Vert_p. \nonumber
      \end{align}
    To be more precise, $\itLamb$ is a $(\itLamb, \itLamb)$-bimodule with two norms,
    and $\itLamb$ is a normed module satisfying Definition \ref{def:normed mod}\point~when it is considered as a module defined in (\ref{formula:Lambda}).

  \item[(2)] For any $\itLamb$-homomorphism $f:M\to N$ of two $\itLamb$-modules $M$ and $N$,
    if $M$ and $N$ are normed $\itLamb$-modules, \checks{i.e.,} $M=(M,h_M,\Vert\cdot\Vert_M)$ and $N=(N,h_N,\Vert\cdot\Vert_N)$, then we have
    \[\Vert f(am)\Vert_N = \Vert af(m)\Vert_N = |\tau(a)| \cdot \Vert f(m)\Vert_N \]
\end{itemize}
\end{fact}

\begin{example} \label{exp:1} \rm
Let $\itLamb=\left(\begin{smallmatrix}
\kk & 0 \\
\kk & \kk
\end{smallmatrix}\right).$
Then a $\kk$-basis of $\itLamb$ is $B_{\itLamb}=\{\bfE_{11}, \bfE_{21}, \bfE_{22}\}$,
where $\bfE_{11} = \big({}_0^1\ {}_0^0\big)$, $\bfE_{21} = \big({}_1^0\ {}_0^0\big)$, $\bfE_{22} = \big({}_0^0\ {}_1^0\big)$.
Take $\norm{}$ be the map $B_{\itLamb}\to\RR^+$ defined by $\norm{}(\bfE_{11})=\norm{}(\bfE_{21})=\norm{}(\bfE_{22})=1$,
then for any element $x=\big({}_{k_{21}}^{k_{11}}\ {}_{k_{22}}^0\big)$ in $\itLamb$,
we have $\Vert x\Vert_p= (|k_{11}|^p+|k_{21}|^p+|k_{22}|^p)^{\frac{1}{p}}$.
There are three indecomposable $\itLamb$-modules up to $\itLamb$-isomorphisms:
\[ P(1)=\left(\begin{smallmatrix}
\kk & 0 \\
\kk & 0
\end{smallmatrix}\right)
\cong
\left(\begin{smallmatrix}
0 & \kk \\
0 & \kk
\end{smallmatrix}\right),
P(2)=\left(\begin{smallmatrix}
 0 & 0 \\
 0 & \kk
\end{smallmatrix}\right),
\]
\[\text{and the cokernel }
\mathrm{coker}\left(P(2)\to P(1)\cdot\left(\begin{smallmatrix}
 0 & 1 \\
 1 & 0
\end{smallmatrix}\right)\right).\]
Then each $\itLamb$-module $M$ is isomorphic to the direct sum
$P(1)^{\oplus t_1}\oplus P(2)^{\oplus t_2}\oplus (P(1)/P(2))^{\oplus t_3}$ for some $t_1, t_2, t_3\in\NN$.
Assume that $M=(M,h_M,\Vert\cdot\Vert_M)$ and $N=(N,h_N,\Vert\cdot\Vert_N)$ are two normed $\itLamb$-modules.
Then, naturally, $M\oplus N$ is also a $\itLamb$-module, where the left $\itLamb$-action is the map
\[h_M\oplus h_N:= \left(
\begin{smallmatrix}
h_M & 0\\
0 & h_N
\end{smallmatrix}\right):
\itLamb \times M\oplus N \to M\oplus N\]
which sends $(a,({}_n^m))$ to
\begin{center}
  $\left(\begin{smallmatrix} h_M & 0\\ 0 & h_N \end{smallmatrix}\right)({}_n^m) = \Big({}_{(h_N)_a(n)}^{(h_M)_a(m)}\Big) = ({_{an}^{am}})$.
\end{center}
Furthermore, we can use the $\tau$-norms of $M$ and $N$, \checks{i.e.,} $\Vert\cdot\Vert_{M}$ and $\Vert\cdot\Vert_{N}$,
to define a $\tau$-norm $\Vert\cdot\Vert_{M\oplus N}$ of $M\oplus N$ by
\begin{center}
  $\Vert(m,n)\Vert_{M\oplus N} := \left(|k|(\Vert m\Vert_{M}^p + \Vert n\Vert_{N}^p)\right)^{\frac{1}{p}}$ for given $k\in\kk\backslash\{0\}$.
\end{center}
Then we have
\begin{align}
     \Vert a(m,n)\Vert_{M\oplus N}
 & = \left(|k|(\Vert am\Vert_{M}^p + \Vert an\Vert_{N}^p)\right)^{\frac{1}{p}}
   = \left(|k|(|\tau(a)|^p\Vert m\Vert_{M}^p + |\tau(a)|^p\Vert n\Vert_{N}^p)\right)^{\frac{1}{p}} \nonumber  \\
 & = |\tau(a)| \left(|k|(\Vert m\Vert_{M}^p + \Vert n\Vert_{N}^p)\right)^{\frac{1}{p}}
   = |\tau(a)| \Vert (m,n)\Vert_{M\oplus N} \nonumber
\end{align}
for any $a\in\itLamb$.
\end{example}

\begin{example} \label{exp:2} \rm
The quiver of the $\kk$-algebra $\itLamb$ given in Example \ref{exp:1}\point~is $\Q= 1 \To{\alpha} 2$.
By representation theory, all $\itLamb$-modules $M$ can be represented by $M_1 \To{\varphi_a} M_2$,
where $M_1$ and $M_2$ are two $\kk$-vector spaces and $\varphi_a$ is a $\kk$-linear map.
Indeed, the identity element of $\itLamb$ is $\bfE=\bfE_{11}+\bfE_{22}$,
where $\{\bfE_{11},\bfE_{22}\}$ is the complete set of primitive orthogonal idempotents.
Thus, $M$, as a $\kk$-vector space, has a decomposition $M = \bfE_{11}M \oplus \bfE_{22}M$
(because $\bfE_{11}\bfE_{22}=0$ yields $\bfE_{11}M\cap\bfE_{22}M = 0$).
For any $a=k_{11}\bfE_{11}+k_{22}\bfE_{22}+k_{21}\bfE_{21}$ and $m\in M$, we have
\begin{align}
am & = (k_{11}\bfE_{11} + k_{22}\bfE_{22} + k_{21}\bfE_{21}) (\bfE_{11}m + \bfE_{22}m) \nonumber \\
   & = k_{11}\bfE_{11}(\bfE_{11}m) + k_{22}\bfE_{22}(\bfE_{22}m) + k_{21}\bfE_{21}(\bfE_{11}m) \nonumber \\
   & = k_{11}(h_M)_{\bfE_{11}}(\bfE_{11}m) + k_{22}(h_M)_{\bfE_{22}}(\bfE_{22}m) + k_{21}(h_M)_{\bfE_{21}}(\bfE_{11}m) \nonumber \\
   & = (h_M)_{\bfE_{11}}(k_{11}\bfE_{11}m) + (h_M)_{k_{22}\bfE_{22}}(\bfE_{22}m) + (h_M)_{\bfE_{21}}(k_{21}\bfE_{11}m), \label{formula:action}
\end{align}
where
\begin{itemize}
  \item[(a)] $h_M: \itLamb\to \End_{\kk}M$ is a homomorphism of $\kk$-algebras sending $a$ to $(h_M)_a$,
    which satisfies $\id_M=(h_M)_{\bfE} = (h_M)_{\bfE_{11}}+(h_M)_{\bfE_{22}}$;
  \item[(b)] $(h_M)_{\bfE_{ii}} = \id_{\bfE_{ii}M}$ ($i=1,2$);
  \item[(c)] $(h_M)_{\bfE_{12}}: \bfE_{11}M \to \bfE_{22}M$ is a $\kk$-linear map (this is equivalent to (\ref{formula:action})).
\end{itemize}
Therefore, we obtain that the representation corresponding to $M=\bfE_{11}M \oplus \bfE_{22}M$ is
\[\bfE_{11}M \To{\bfE_{21}} \bfE_{22}M.\]
Generally, $M_1 \To{\varphi_a} M_2$ corresponds to the module $M_1\oplus M_2$, where the $\itLamb$-action
$\itLamb \times M_1\oplus M_2 \to M_1\oplus M_2$ is defined by
\begin{center}
  $\bfE_{11}(m_1,m_2)=(m_1,0)$, $\bfE_{22}(m_1,m_2)=(0,m_2)$,
\end{center}
\begin{center}
  and $\bfE_{12}(m_1,m_2)=(0,\varphi_{\alpha}(m_1))$.
\end{center}
Without loss of generality, for any representation $M_1 \To{\varphi_a} M_2$ of $\Q$,
assume that $M_1=\kk^{\oplus t_1}$, $M_2=\kk^{\oplus t_2}$ and $\varphi_a\in\mathbf{Mat}_{t_2\times t_1}(\kk)$ (up to $\itLamb$-isomorphism),
and for any $i=1,2$, $M_i$ is a normed space equipping with the norm $\Vert\cdot\Vert_{M_i}: M_i=\kk^{\oplus t_i}\to\RR^+$
sending $m_i=(m_{ij})_{1\le j\le t_i}$ to $\left(\sum_{j=1}^{t_i} |m_{ij}|^p\right)^{\frac{1}{p}}$.
Then we can define a norm $\Vert\cdot\Vert_{M_1\oplus M_2}$ by
\[ \Vert(m_1,m_2)\Vert_{M_1\oplus M_2} = (|k|(\Vert m_1 \Vert_{M_1}^p + \Vert m_2 \Vert_{M_2}^p))^{\frac{1}{p}}, \]
where $k$ is a given element in $\kk\backslash\{0\}$.
The direct sum $\oplus$ of $\kk$-vector spaces is the $p$ powers of the norm preserving in the case for $k=1$,
\checks{i.e.,} $\Vert(m_1,m_2)\Vert_{M_1\oplus M_2}^p = \Vert m_1\Vert_{M_1}^p + \Vert m_2\Vert_{M_2}^p$.
Furthermore, if $\Vert\cdot\Vert_{M_1}$ and $\Vert\cdot\Vert_{M_2}$ are $\tau$-norms of $M_1$ and $M_2$, respectively,
then, for any $a\in\itLamb$, we have
\begin{align*}
    \Vert a(m_1,m_2)\Vert_{M_1\oplus M_2}
& = \big(|k|(\Vert am_1 \Vert_{M_1}^p + \Vert am_2 \Vert_{M_2}^p)\big)^{\frac{1}{p}} \\
& = \big(|k|(|\tau(a)|^p\Vert m_1 \Vert_{M_1}^p + |\tau(a)|^p \Vert m_2 \Vert_{M_2}^p)\big)^{\frac{1}{p}} \\
& = |\tau(a)| \big(|k|(\Vert m_1 \Vert_{M_1}^p + \Vert m_2 \Vert_{M_2}^p )\big)^{\frac{1}{p}} \\
& = |\tau(a)| \Vert a(m_1,m_2)\Vert_{M_1\oplus M_2}.
\end{align*}
\end{example}

\subsection{\texorpdfstring{Completions of normed $\itLamb$-modules}{}} \label{subsect:normed mod.2}

Let $N=(N,h,\Vert\cdot\Vert)$ be a normed $\itLamb$-module. In this part we construct its completion.
For us, we need only the completion in the finite-dimensional $\kk$-algebra $\itLamb$ case.
Otherwise, there is at least one $\itLamb$-module which is not complete,
for instance, $\itLamb$ is a non-complete $\itLamb$-module.
Therefore, we assume that $\kk$ is complete in this subsection  by Propositions \ref{prop:compl1}\point~and \ref{prop:compl2}.

Similar to finite-dimensional $\kk$-algebras, we can define open neighborhoods $B(0,r)$ of $0$
for any normed $\itLamb$-module $N=(N,h,\Vert\cdot\Vert)$ by
\[ B(0,r) := \{ x\in N \mid \Vert x \Vert < r \}.  \]
Let $\frakU_{N}^B(0)$ be the class of all subsets $U$ of $N$ satisfying the following conditions.
\begin{itemize}
  \item[(1)] $U$ is the intersection of a finite number of $B(0,r)$;
  \item[(2)] $U$ is the union of any number of $B(0,r)$.
\end{itemize}
Then  $\frakU_{N}^B(0)$ is a topology defined on $N$,
and we can define the Cauchy sequence by the above topology.

\begin{lemma} \label{lemm:C(N)}
Let $\mathfrak{C}^{*}(N)$ be the set of all Cauchy sequences in the normed $\itLamb$-module
$N=(N,h,\Vert\cdot\Vert)$. Then $\mathfrak{C}^{*}(N)$ is a $\itLamb$-module.
\end{lemma}

\begin{proof}
First of all, $\mathfrak{C}^{*}(N)$ is a $\kk$-vector space whose addition and $\kk$-action are given by
\[\{x_i\}_{i\in\NN}+\{y_i\}_{i\in\NN} = \{x_i+y_i\}_{i\in\NN}
\ \ (\forall \{x_i\}_{i\in\NN}, \{y_i\}_{i\in\NN} \in \mathfrak{C}^{*}(N)) \]
\[ \text{ and } k \{x_i\}_{i\in\NN} = \{k x_i\}_{i\in\NN} \ \ (\forall k\in\kk),  \]
respectively. Furthermore, define
\[ \itLamb \times \mathfrak{C}^{*}(N) \to \mathfrak{C}^{*}(N),
\ (a, \{x_i\}_{i\in\NN}) \mapsto a\cdot\{x_i\}_{i\in\NN}:=\{a\cdot x_i\}_{i\in\NN}, \]
where $a\cdot x_i = h_a(x_i)$. Then $\mathfrak{C}^{*}(N)$ is a $\itLamb$-module.
\end{proof}

Two Cauchy sequences $\{x_i\}_{i\in\NN}$ and $\{y_i\}_{i\in\NN}$ in $N$ are called {\defines equivalent},
denoted by $\{x_i\}_{i\in\NN}\sim \{y_i\}_{i\in\NN}$, if for any $U\in\frakU_N^B(0)$,
there is $r\in\NN$ such that $x_s-x_t\in U$ holds for all $s,t\ge r$.
It is easy to see that ``$\sim$'' is an equivalence relation.
Let $[\{x_i\}_{i\in\NN}]$ be the equivalent class of Cauchy sequences containing $\{x_i\}_{i\in\NN}$
and let $\mathfrak{C}(N)$ be the set of all equivalent classes.
We naturally obtain a map
\[h:\mathfrak{C}^{*}(N) \to \mathfrak{C}(N),\ \ \{x_i\}_{i\in\NN} \mapsto [\{x_i\}_{i\in\NN}].\]
We can show that $\mathfrak{C}(N)$ is a $\itLamb$-module
by using an argument similar to that in the proof of Lemma \ref{lemm:C(N)}\point, and further obtain
$\ker(h: \mathfrak{C}^{*}(N) \to \mathfrak{C}(N)) = [\{0\}_{i\in\NN}]$.
Thus we have
\[ \mathfrak{C}(N) \cong \mathfrak{C}^{*}(N)/[\{0\}_{i\in\NN}]. \]
Then $\mathfrak{C}(N)$ is complete, and we call it the {\defines completion} of $N$.
We use $\w{N}$ to denote the completion $\mathfrak{C}(N)$ of $N$.
The $\itLamb$-module $\w{N}$ is a normed $\itLamb$-module, where the norm defined on $\w{N}$
is induced by the norm $\Vert\cdot\Vert: N\to\RR^{\ge 0}$ defined on $N$.

\begin{definition} \rm \label{def:Banach mods}
Assume that $\itLamb$ is complete. A normed $\itLamb$-module $N$ is called a {\defines Banach $\itLamb$-module}
if $\w{N}=N$ (i.e.{,} $N$ is complete).
\end{definition}

\subsection{
\texorpdfstring{$\sigma$-algebras and
the elementary simple function set $\bfS_{\tau}(\II_{\itLamb})$}{}}
\label{subsect:normed mod.3}

\begin{lemma}
Take $\tau$ to be a homomorphism of $\kk$-algebras $\tau: \itLamb\to \kk$.
Then the elementary simple function set $\bfS(\II_{\itLamb})$ with the above homomorphism $\tau$,
denoted by $\bfS_{\tau}(\II_{\itLamb})$, is a $\itLamb$-module,
where the $\itLamb$-action $\itLamb\times \bfS(\II_{\itLamb})\to \bfS(\II_{\itLamb})$ is given by
\begin{center}
  $(a, f=\sum_{i=1}^t k_i\id_{I_i}) \mapsto af:= \sum_{i=1}^t \tau(a)k_i\id_{I_i}$.
\end{center}
\end{lemma}

\begin{proof}
For all $a\in\itLamb$, $a'\in\itLamb$, $k\in\kk$, $f=\sum_i k_i\id_{I_i}\in\bfS(\II_{\itLamb})$
and $f'=\sum_j k_j'\id_{I_j'}\in\bfS(\II_{\itLamb})$,
the following conditions are satisfied:
\begin{itemize}
  \item[(1)] 
    $a(f+f')=af+af'$ (trivial);
  \item[(2)] 
    $(a+a')f=af+a'f$ (trivial);
  \item[(3)] $(aa')f= a(a'f)$ because
    \begin{align}
      (aa')f & =(aa')\sum\nolimits_i k_i\id_{I_i}
               = \sum\nolimits_i \tau(aa')k_i\id_{I_i}
               = \sum\nolimits_i \tau(a)\tau(a')k_i\id_{I_i} \nonumber \\
             & = a\sum\nolimits_i\tau(a')k_i\id_{I_i}
               = a(a'\sum\nolimits_ik_i\id_{I_i}) = a(a'f) \nonumber
    \end{align}
  \item[(4)] $1f=f$ (trivial);
  \item[(5)] We have
             \begin{itemize}
               \item $(ka)f=(ka)\sum_i k_i\id_{I_i} = \sum_i \tau(ka)(k_i\id_{I_i})$,
               \item $k(af)=k(a\sum_i k_i\id_{I_i}) = k\sum_i \tau(a)k_i\id_{I_i} = \sum_i k(\tau(a)(k_i\id_{I_i}))$,
               \item and $a(kf)=a\sum_i k(k_i\id_{I_i}) = \sum_i \tau(a)(k(k_i\id_{I_i}))$.
             \end{itemize}
    Since $\tau$ is a homomorphism of $\kk$-algebras, we have
    \[\tau(ka)(k_i\id_{I_i})=k(\tau(a)(k_i\id_{I_i}))=\sum\nolimits_i\tau(a)(k(k_i\id_{I_i}))
      = \sum\nolimits_i kk_i\tau(a)\id_{I_i}, \]
    for all $i$. Then $(ka)f=k(af)=a(kf)$.
\end{itemize}
\end{proof}

Now, we introduce a norm for $\bfS_{\tau}(\II_{\itLamb})$ such that it is a normed $\itLamb$-module.
To do this, we first recall the definition of $\sigma$-algebras.
The main use of $\sigma$-algebras is in the definition of measures.
It is important in mathematical analysis and probability theory.
In mathematical analysis, it is the foundation for Lebesgue integration,
and in probability theory, it is interpreted as the collection of events that can be assigned probabilities.
See for example \cite[Page 12]{Eric2009}, \cite[Page 10]{K2001} and \cite[Page 8]{Rudin1987}.

\begin{definition} \rm
Let $S$ be a set and let $P(S)$ be the set of all subsets of $S$, which is called the power set of $S$.
A {\defines $\sigma$-algebra} is a subset $\mathcal{A}$ of $P(S)$ satisfying the following conditions:
\begin{itemize}
  \item[(1)] $\varnothing$ and $S$ lie in $\mathcal{A}$;
  \item[(2)] for any $X\in\mathcal{A}$, the complement set $X^c := S\backslash X$ of $X$ lies in $\mathcal{A}$;
  \item[(3)] for any $X_1,\ldots, X_n\checks{,} \ldots\in \mathcal{A}$, the union $\bigcup_{i=1}^{\infty} X_i$ is an element in $\mathcal{A}$.
\end{itemize}
For a class $\mathcal{C}$ of some sets lying in $P(S)$,
we call $\mathcal{A}$ a {\defines $\sigma$-algebra generated by $\mathcal{C}$}
if $\mathcal{A}$ is the minimal $\sigma$-algebra containing $\mathcal{C}$.
\end{definition}

Let $\Sigma_{\kk}$ be the $\sigma$-algebra generated by $\{(a,b)_{\kk}, [a,b)_{\kk}, (a,b]_{\kk}, [a,b]_{\kk} \mid a\preceq b\}$,
and let $\mu: \Sigma_{\kk}\to\RR^{\ge0}$ be a {\defines measure} such that $\mu(\{k\})=0$ holds for any $k\in\kk$,
\checks{i.e.,} $\mu$ is a function satisfying the following conditions:
\begin{itemize}
  \item[(1)] $\mu(\varnothing)=0$;
  \item[(2)] $\mu(\bigcup_{i\in\NN}X_i) = \sum_{i\in\NN} \mu(X_i)$ holds for all sets $X_1,X_2,\ldots$ satisfying $X_i\cap X_j=\varnothing$ ($i\ne j$).
\end{itemize}

Any two functions $f$ and $g$ in $\bfS(\II_{\itLamb})$ are called {\defines equivalent} if
\begin{center}
  $\mu(\{\pmb{k}=(k_1,\ldots,k_{n}) \in \kk^{\oplus n} \mid f(\pmb{k})\ne g(\pmb{k}) \})=0$.
\end{center}
The equivalent class containing $f$ is written as $[f]$. Then we obtain an epimorphism
\[ \bfS(\II_{\itLamb}) \to \overline{\bfS(\II_{\itLamb})} := \{[f] \mid f\in\bfS(\II_{\itLamb})\}\]
sending each function to its equivalent classes. It is easy to see that the kernel of the above epimorphism is $[0]$.
Then we have
\[ \overline{\bfS(\II_{\itLamb})} \cong \bfS(\II_{\itLamb})/[0]. \]
For simplification, we do not differentiate between two equivalent functions under the above isomorphism.
Therefore, we treat $\bfS(\II_{\itLamb})$ and the quotient $\overline{\bfS(\II_{\itLamb})}$ equivalently.

\begin{lemma} \label{lemm:normed mod}
Let $\tau:\itLamb\to\kk$ be a homomorphism between two $\kk$-algebras.
Then the $\itLamb$-module $\bfS_{\tau}(\II_{\itLamb})$ with the map
\[
  \Vert\cdot\Vert_p: \bfS_{\tau}(\II_{\itLamb}) \to \RR^{\ge 0}, \
  f=\sum_{i=1}^t k_i\id_{I_i} \mapsto \bigg(\sum_{i=1}^t (|k_i|\mu(I_i))^p\bigg)^{\frac{1}{p}}
\]
is normed.
\end{lemma}

\begin{proof}
Let $f$ be an arbitrary function lying in $\bfS(\II_{\itLamb})$.
It is trivial that $\Vert f\Vert_p$ is non-negative.
Let $a$ be an arbitrary element in $\itLamb$ and assume $f=\sum_{i=1}^t k_i\id_{I_i}$. We have
\begin{align}
\Vert af\Vert_p & = \Big\Vert \sum_{i=1}^t \tau(a)k_i\id_{I_i}\Big\Vert_p
                  = \Big(\sum_{i=1}^t |\tau(a)k_i|^p\mu(\id_{I_i})^p\Big)^{\frac{1}{p}} \nonumber \\
                & = |\tau(a)|\cdot\Big(\sum_{i=1}^t |k_i|^p\mu(\id_{I_i})^p\Big)^{\frac{1}{p}}
                  = |\tau(a)|\cdot\Vert f\Vert_p \nonumber
\end{align}
which satisfies the formula (\ref{formula:normedmod}).
In particular, if $\Vert f\Vert_p=0$, then so is $(|k_i|\mu(I_i))^p = 0$ for all $i$,
and we have $|k_i|=0$ in the case for $\mu(I_i)\ne 0$.
If $\mu(I_j)=0$ holds for some $j\in J$ $(\subseteq \{1,2,\ldots, t\})$, then we have
$f = \sum_{j\in J} k_j\id_{I_j}$. Clearly,
\[ \mu(\{ x \in \II_{\itLamb} \mid f(x) \ne 0 \}) = \sum_{j\in J}\mu(I_j) = 0, \]
\checks{i.e.,} $f=0$ in treating $\bfS(\II_{\itLamb})$ and the quotient $\overline{\bfS(\II_{\itLamb})}$ equivalently.
Thus, $\Vert f\Vert_p=0$ if and only if $f=0$.

Next, we prove the triangle inequality. For two arbitrary functions $f=\sum_i k_i\id_{I_i}$ and $g=\sum_j l_j\id_{I_j'}$,
we have
\begin{align}\label{formula:f+g}
  f+g = \sum\nolimits_i k_i\id_{I_i \backslash \bigcup\nolimits_jI_j'}
      + \sum\nolimits_j l_j\id_{I_j'\backslash \bigcup\nolimits_i I_i}
      + \sum_{I_i\cap I_j'\ne\varnothing} (k_i\id_{I_i\cap I_j'}+l_j\id_{I_i\cap I_j'})
\end{align}
by $I_i\cap I_{\imath}=\varnothing$ ($\forall i\ne \imath$) and $I_j'\cap I_{\jmath}'=\varnothing$ ($\forall j\ne \jmath$).
Then we can compute the norm of $f+g$ by (\ref{formula:f+g}) as the following formula:
\[\Vert f+g \Vert_p = ({R+G+B)^{\frac{1}{p}}}, \]
where
\begin{align}
 R & = \sum\nolimits_i |k_i|^p\mu\big(I_i \big\backslash \bigcup\nolimits_j I_j'\big)^p; \nonumber \\
 G & = \sum\nolimits_j |l_j|^p\mu\big(I_j'\big\backslash \bigcup\nolimits_i I_i \big)^p; \nonumber \\
 B & =  \sum_{I_i\cap I_j'\ne\varnothing} (|k_i|^p+|l_j|^p)\mu(I_i\cap I_j')^p. \nonumber
\end{align}
On the other hand, we have the following inequality by the discrete Minkowski inequality:
\begin{align}
      \Vert f \Vert_p + \Vert g \Vert_p
  = & \bigg(\sum\nolimits_i |k_i|^p\mu(I_i)^p\bigg)^{\frac{1}{p}}
      + \bigg(\sum\nolimits_j |l_i|^p\mu(I_i')^p\bigg)^{\frac{1}{p}} \nonumber \\
\ge & \bigg(\sum\nolimits_i |k_i|^p\mu(I_i)^p + \sum\nolimits_j |l_i|^p\mu(I_i')^p\bigg)^{\frac{1}{p}} =: \mathfrak{S}.
\end{align}
Since, by the definition of measure, $\mu(X\cup Y)=\mu(X)+\mu(Y)$ holds for any $X, Y$ with $X\cap Y=\varnothing$,
we obtain
\begin{align}\label{formula:mu-addi}
  \mu(X\cup Y)^p \ge \mu(X)^p+\mu(Y)^p,
\end{align}
then
\[\mu(I_i)^p\ge\mu\big(I_i\backslash\bigcup\nolimits_j I_j'\big)^p + \mu\big(I_i\cap\bigcup\nolimits_j I_j'\big)^p. \]
Thus,
\begin{align} \label{formula:red}
        \sum\nolimits_i |k_i|^p\mu(I_i)^p
\ge\ &  \sum\nolimits_i |k_i|^p\mu\big(I_i\backslash\bigcup\nolimits_j I_j'\big)^p
        + \sum\nolimits_i |k_i|^p\mu\big(I_i\cap\bigcup\nolimits_j I_j'\big)^p
        \nonumber \\
  =\ &  R
        + \sum\nolimits_i |k_i|^p
        \bigg(\mathop{\sum\nolimits_{j}}\limits_{I_i\cap I_j'\ne\varnothing}\mu(I_i \cap I_j')\bigg)^p
        \nonumber \\
  \mathop{\ge}\limits^{(\ref{formula:mu-addi})}
     &  R + \sum_{I_i\cap I_j'\ne\varnothing} |k_i|^p\mu(I_i\cap I_j')^p.
\end{align}
Similarly,
\begin{align} \label{formula:green}
       \sum\nolimits_j |l_j|^p\mu(I_j')^p
\ge G + \sum_{I_j'\cap I_i\ne\varnothing} |l_j|^p\mu(I_j'\cap I_i)^p.
\end{align}
Notice that
\[\sum_{I_i\cap I_j'\ne\varnothing} |k_i|^p\mu(I_i\cap I_j')^p
+ \sum_{I_j'\cap I_i\ne\varnothing} |l_j|^p\mu(I_j'\cap I_i)^p
= \sum_{I_i\cap I_j'\ne\varnothing} (|k_i|^p+|l_j|^p)\mu(I_i\cap I_j')^p
= B, \]
then (\ref{formula:red})$+$(\ref{formula:green}) induces
$\mathfrak{S}^p \ge {R+G+B}$.
Thus, the triangle inequality $\Vert f \Vert_p+\Vert g\Vert_p \ge \Vert f+g \Vert_p$ holds.
\end{proof}

\section{\texorpdfstring{The categories $\scrN^p$ and $\scrA^p$}{}} \label{sect:two cats}

Recall that a measure defined on $\Sigma_{\kk}$ is a countable additive function $\mu:\Sigma_{\kk}\to\RR^{\ge 0}$ with $\mu(\varnothing)=0$.
Naturally, it induces a measure, still written as $\mu$, defined on some $\sigma$-algebra of $\itLamb$ such that,
for any $\sum_{i=1}^n I_i b_i$ ($I_i\in\Sigma_{\kk}$ is measurable), the equation $\mu(\sum_{i=1}^n I_i b_i) = \prod_{i=1}^n \mu(I_i)$ holds.

Let $\dim_{\kk}\itLamb=n$, and let $N$ be a normed $\itLamb$-module equipped with two additional pieces of data:
an element $v\in N$ such that $\Vert v\Vert\le \mu(\II_{\itLamb})$, and a continuous $\itLamb$-homomorphism
$\delta: N^{\oplus_p 2^n} \to N$. Here, $\oplus_p$ denotes the direct sum of $2^n$ normed $\itLamb$-modules $X_1, \ldots, X_{2^n}$
with the norm defined as follows:
\[ \Vert\cdot\Vert_p: \bigoplusp{i=1}{2^n} X_i \to \RR^{\ge0},
   (x_1,x_2,\ldots, x_{2^n})
     \mapsto
       \bigg(
             \bigg(\frac{\mu(\II)}{\mu(\II_{\itLamb})}\bigg)^n
             \sum_{i=1}^{2^n}\Vert x_i\Vert^p
       \bigg)^{\tfrac{1}{p}}.\]

\subsection{\texorpdfstring{The categories $\scrN^p$ and $\scrA^p$}{}} \label{sect:two cats.1}
Let $\scrN^p$ be a class of triples which are of the form $(N, v, \delta)$,
where $N$ is a normed $\itLamb$-module, $v\in N$ is an element with $\Vert v\Vert_p \le \mu(\II_{\itLamb})$
and $\delta: N^{\oplus_p 2^n}\to N$ is a $\itLamb$-homomorphism satisfying $\delta(v,v,\ldots,v)=v$
such that for any Cauchy sequence $\{x_i\}_{i\in\NN} \in \w{N^{\oplus_p 2^n}} \cong \w{N}^{\oplus_p 2^n}$,
the commutativity
\begin{align}\label{formula:comm}
 \underleftarrow{\lim}\delta(x_i) = \delta(\underleftarrow{\lim}x_i)
\end{align}
of the inverse limit and the $\itLamb$-homomorphism holds.
For any two triples $(N,v,\delta)$ and $(N',v',\delta')$ in $\scrN^p$,
we define the morphism $(N,v,\delta) \to (N',v',\delta')$ to be the $\itLamb$-homomorphism $\theta: N\to N'$ with $\theta(v)=v'$
such that the following diagram
\[\xymatrix@R=1.5cm@C=1.5cm{
  N^{\oplus_p 2^n} \ar[r]^{\delta}
  \ar[d]_{ \theta^{\oplus 2^n}
          =\left(\begin{smallmatrix}
          \theta &&\\
          &\ddots&\\
          && \theta
          \end{smallmatrix}\right)_{2^n\times 2^n}}
& N \ar[d]^{\theta\ } \\
  N'^{\oplus_p 2^n} \ar[r]_{\delta'}
& N'
}\]
commutes, \checks{i.e.,} for any $(v_1,\ldots, v_{2^n})\in N^{\oplus_p 2^n}$,
$\theta(\delta(v_1,\ldots,v_{2^n}))=\delta'(\theta(v_1),\ldots,\theta(v_{2^n}))$.
Then it is easy to check that $\scrN^p$ is a category.

\begin{lemma} \label{lemm:S in Np pre}
Let
\begin{itemize}
  \item[\rm (1)] $\xi$ be an element in $\II=[a,b]_{\kk}$ with $a\prec \xi\prec b$ such that the order-preserving bijections
  $\kappa_a:\II\to[a,\xi]_{\kk}$ and $\kappa_b:\II\to[\xi,b]_{\kk}$ exist,
  \item[\rm (2)] $\id$ be the identity function $\id_{\II_{\itLamb}}: \II_{\itLamb} \to \{1\}$,
  \item[\rm (3)] $\gamma_{\xi}$ be the map given in (\ref{map:gamma}),
  \item[\rm (4)] $\tau:\itLamb\to\kk$ be the homomorphism of $\kk$-algebras given in Lemma \ref{lemm:normed mod}.
\end{itemize}
Then the following statements hold.
\begin{itemize}
  \item[\rm (a)] $\gamma_{\xi}(\id, \id, \cdots, \id) = \id$;
  \item[\rm (b)] $\gamma_{\xi}$ is a $\itLamb$-homomorphism.
\end{itemize}
\end{lemma}

First, we provide a remark for the above lemma.

\begin{remark} \rm
Indeed, $(\bfS_{\tau}(\II_{\itLamb}), \id, \gamma_{\xi})$ is an object in the category $\scrN^p$.
However, Lemma \ref{lemm:S in Np pre}\point~points out that $\gamma_{\xi}(\id, \id, \cdots, \id) = \id$
and $\gamma_{\xi}$ is a $\itLamb$-homomorphism.
Thus, we need to show that the commutativity of the inverse limit and $\gamma_{\xi}$ holds.
We will prove this result in the following content, as shown in Lemma \ref{lemm:S in Np}.
\end{remark}

Next, we prove Lemma \ref{lemm:S in Np pre}.

\begin{proof}
(a) We have that $\bfS_{\tau}(\II_{\itLamb})$ is a normed $\itLamb$-module by Lemma \ref{lemm:normed mod}\point,
and $\gamma_{\xi}$ is a $\kk$-linear map by Lemma \ref{lemm:gamma linear}.
The formula $\gamma_{\xi}(\id,\ldots,\id)=\id$ can be directly induced by the definition of $\gamma_{\xi}$.

(b) Take $\lambda\in\itLamb$, $f\in\bfS(\II_{\itLamb})$ and let $(k_i)_i$, $\id$ and $(\delta_i)_i$ be
an arbitrary element $(k_1,\ldots,k_n)$ in $\bfS(\II_{\itLamb})^{\oplus 2^n}$,
the identity function $\id_{\kappa_{\delta_1}(\II)\times\cdots\times\kappa_{\delta_n}(\II)}$
and the $n$-multiple $(\delta_1\times\cdots\times\delta_n)$, respectively.
Then we have
\begin{align}
   & \gamma_{\xi}(\lambda\cdot f)((k_i)_i) \nonumber \\
=\ & \sum_{(\delta_i)_i}\id\cdot (\tau(\lambda)f)_{(\delta_i)_i} ((\kappa_{\delta_i}^{-1}(k_i))_{i}) \nonumber \\
=\ & \tau(\lambda)\gamma_{\xi}(f)((k_i)_i)
 \ \ (\text{similar to Lemma \ref{lemm:gamma linear}}\point)  \nonumber \\
=\ & \lambda\cdot \gamma_{\xi}(f)((k_i)_i).  \nonumber
\end{align}
Thus $\gamma_{\xi}$ is a $\itLamb$-homomorphism.
\end{proof}

Let $\scrA^p$ denote a class of triples which are of the form $(\w{N}, v, \w{\delta})$,
where $\w{N}$ is a Banach $\itLamb$-module (see Definition \ref{def:Banach mods}\point),
$v\in \w{N}$ is an element with $\Vert v\Vert \le \mu(\II_{\itLamb})$
and $\w{\delta}: \w{N}^{\oplus_p 2^n}\to \w{N}$ is a $\itLamb$-homomorphism satisfying $\w{\delta}(v,v,\ldots,v)=v$.
Obviously, $\scrA^p$ is a full subcategory of $\scrN^p$.

\subsection{\texorpdfstring{The triple $(\bfS_{\tau}(\II_{\itLamb}), \id, \gamma_{\xi})$}{} }

Let $(N,v,\delta)$ be an object in $\scrN^p$ and $\w{N}$ the completion of the $\itLamb$-module $N$.
Then $\w{N}$, as a $\kk$-vector space, is a Banach space which is a Banach $\itLamb$-module.
And, naturally, we obtain the $\itLamb$-homomorphism
\[\w{\delta}: \w{N}^{\oplus_p 2^n} \to \w{N}\]
induced by the $\itLamb$-homomorphism $\delta$. Furthermore, we have that $(\w{N}, v,\w{\delta})$
is also an object in $\scrN^p$, and there is a naturally embedding morphism
\[ \emb: (N,v,\delta) \hookrightarrow (\w{N}, v,\w{\delta}) \]
which is induced by $N \subseteq \w{N}$.

\begin{notation} \rm \label{notation}
Keep the notations $\xi=:\xi_{11}$, $\kappa_a$, $\kappa_b$,
$\id$, $\gamma_{\xi}$ and $\tau$ as in Lemma \ref{lemm:S in Np pre}.
Then $\xi_{11}$ divides $\II=:\II^{(01)}$ into two subsets $[a,\xi_{11}]_{\kk}=:\II^{(11)}$ and $[\xi_{11},b]_{\kk}=:\II^{(12)}$.
Next, let $\xi_{22}=\xi_{11}$ $(=\xi)$, and denote by $\xi_{21}$ and $\xi_{23}$ the two elements in $\II_{\itLamb}$ such that
\begin{itemize}
  \item $a\prec\xi_{21}=\kappa_a\kappa_a(b) = \kappa_a\kappa_b(a) = \kappa_b\kappa_a(a) = \kappa_a(\xi_{11})\prec\xi_{22}$;
  \item $\xi_{22}\prec\xi_{23}=\kappa_b\kappa_b(a)=\kappa_b\kappa_a(b)=\kappa_b\kappa_a(b) = \kappa_b({\xi_{11}})\prec b$.
\end{itemize}
Then $\II$ is divided into four subsets which are of the form
$\II^{(2\ t+1}) = [\xi_{2t}, \xi_{2\ t+1}]_{\kk}$ ($0\le t\le 3$)
by $a=\xi_{20} \prec \xi_{21} \prec \xi_{22} \prec \xi_{23} \prec \xi_{24}=b$.
Repeating the above step $t$ times, we obtain a sequence of $2^t-1$ elements lying in $\II_{\itLamb}$
\[a=\xi_{t0} \prec \xi_{t1} \prec \xi_{t2} \prec \cdots \prec \xi_{t2^t}=b, \]
all $2^t$ subsets which are of the form $\II^{(t\ s+1)}=[\xi_{ts}, \xi_{t\ s+1}]_{\kk}$,
and $2^t$ order-preserving bijections $\kappa_{\xi_{ts}}: \II^{(t\ s+1)} \to \II^{(01)}$.

For any family of subsets $(\II^{(u_iv_i)})_{1\le i\le n}$ ($1\le v_i\le 2^{u_i}$),
we denote by $\id_{(u_iv_i)_i}$ the function
\[ \id_{(u_iv_i)_i} := \id_{\II_{\itLamb}}\Big|{}_{\prod_{i=1}^{n}\II^{(u_iv_i)}}:
\ \II_{\itLamb} \to \{0,1\},\
x \mapsto \begin{cases}
            1, & x\in \prod_{i=1}^{n}\II^{(u_iv_i)}; \\
            0, & \text{otherwise},
          \end{cases} \]
where $\II^{(u_iv_i)} \cong \II^{(u_iv_i)}\times\{b_i\}\subseteq \II_{\itLamb}$ holds for all $i$
and $B_{\itLamb}=\{b_i \mid 1\le i\le n\}$ is the $\kk$-basis of $\itLamb$.

Let $E_{u}$ be the set of all step functions constant on each of $\prod_{i=1}^{n}\II^{(u_iv_i)}$
($1\le v_i\le 2^{u_i}$ for all $i$),
\checks{i.e.,} every step function in $E_{u}$ is of the form
\[\sum_{(u_iv_i)_i} k_{(u_iv_i)_i}\id_{(u_iv_i)_i},\]
where each $k_{(u_iv_i)_i}$ lies in $\kk$, the number of summands is $(2^u)^n=2^{un}$,
and each $(u_iv_i)_i$ corresponds to the Cartesian product $\prod_{i=1}^{n}\II^{(u_iv_i)}$.
Then it is easy to check that each $E_u$ is a normed submodule of $\bfS(\II_{\itLamb})$,
and $E_{u}\subseteq E_{u+1}$ because each step function constant on each of $\II^{(uv)}$
is equivalent to a step function constant on each of $\II^{(u+1\ v)}$. Thus,
\[\kk \cong E_{0} \subseteq E_{1} \subseteq \ldots \subseteq E_{t} \subseteq \ldots
\subseteq \bfS(\II_{\itLamb}) \subseteq \w{\bfS(\II_{\itLamb})}.\]
Moreover, for any $\II^{(uv)}=[\xi_{u\ v-1}, \xi_{uv}]_{\kk}$, we have two cases
(i) $\xi_{uv} \preceq \xi$ and (ii) $\xi \preceq \xi_{u\ v-1}$ by the definition of $E_u$.
Therefore, we obtain a map
\[\frakp: \{\II^{(uv)} \mid u\in \NN\} \to \{a,b\}, \
\II^{(uv)} \mapsto \begin{cases}
a, & \II^{(uv)} \text{ lies in case (i); } \\
b, & \II^{(uv)} \text{ lies in case (ii).}
\end{cases}
\]
\end{notation}

Now we use the above map to prove the following lemma.

\begin{lemma} \label{lemm:inverse gamma}
The map $\gamma_{\xi}: \bfS(\II_{\itLamb})^{\oplus_p 2^n} \to \bfS(\II_{\itLamb})$ induces the following $\kk$-linear map
\[\gamma_{\xi}:E_u^{\oplus_p 2^n} \mathop{\longrightarrow}\limits^{\cong} E_{u+1}\]
which is an isomorphism of $\itLamb$-modules.
\end{lemma}

\begin{proof}
The $\kk$-vector space $E_u$ is a $\itLamb$-module, where $\itLamb\times E_u\to E_u$ is defined by
\begin{center}
$(a,f=\sum_i 1\cdot\id_{I_i}) \mapsto a\cdot f = \sum_i \tau(a)\cdot \id_{I_i}$.
\end{center}
Then it is easy to see that $\gamma_{\xi}$ is a $\itLamb$-homomorphism.
Since $\ker(\gamma_{\xi}) = 0$, we have that $\gamma_{\xi}$ is injective.
Next, we prove that it is also surjective.

Any step function $f: \kk^{\oplus n} \to \kk$ lying in $E_{u+1}$ can be written as
\begin{align}
  f(k_1,\ldots,k_n)
= \sum_{(u_iv_i)_i} f_i 
= \sum_{ (\omega_1, \ldots, \omega_n) \in \{a,b\}\times\cdots\times \{a,b\} } f_{(\omega_1, \ldots, \omega_n)}
\nonumber
\end{align}
where
\begin{itemize}
  \item $f_i=k_{(u_iv_i)_i}\id_{(u_iv_i)_i}$;
  \item \[ f_{(\omega_1,\ldots,\omega_n)}(k_1,\ldots,k_n) = \sum_{ \Pi_{i=1}^{n} \frakp(\II^{(u_iv_i)}) = (\omega_1, \ldots, \omega_n) }  f_i,\]
    thus, the number of all summands of it is $(2^u)^n=2^{un}$;
  \item the number of all summands of $\sum\limits_{ (\omega_1, \ldots, \omega_n) \in \{a,b\}\times\cdots\times \{a,b\} } f_{(\omega_1,...,\omega_n)}$ is $2^n$
    (thus the number of all summands of $\sum\limits_{(u_iv_i)_i} f_i$ is $2^{un}\cdot 2^n = 2^{(u+1)n}$).
\end{itemize}
Then
\[\tilde{f}_{(\omega_1,\ldots,\omega_n)}(k_1,\ldots,k_n) =
f_{(\omega_1,\ldots,\omega_n)}(\kappa_{\omega_1}^{-1}(k_1),\ldots, \kappa_{\omega_n}^{-1}(k_n))\in E_u,\]
and $\gamma_{\xi}$ sends $\{f_{(\omega_1,\ldots, \omega_n)}\}_{(\omega_1,\ldots, \omega_n)\in \{a,b\}\times\cdots\times \{a,b\}}$
to $f$ by the definition of $\gamma_{\xi}$, see (\ref{map:gamma}).
We obtain that $\gamma_{\xi}$ is surjective. Therefore, $\gamma_{\xi}$ is a $\itLamb$-isomorphism.
\end{proof}

By Lemma \ref{lemm:inverse gamma}\point, the following result holds.

\begin{lemma} \label{lemm:S in Np}
The commutativity of the inverse limit and the map
$\w{\gamma}_{\xi}: \w{\bfS_{\tau}(\II_{\itLamb})}{}^{\oplus_p 2^n} \to \w{\bfS_{\tau}(\II_{\itLamb})}$
induced by the completion of $\bfS_{\tau}(\II_{\itLamb})$ holds,
\checks{i.e.,} for any sequence $\{\pmb{f}_i\}_{i\in\NN^+}$ in $\w{\bfS_{\tau}(\II_{\itLamb})}{}^{\oplus_p 2^n}$,
if its inverse limit exists, then we have
\[\w{\gamma}_{\xi}(\underleftarrow{\lim}\pmb{f}_i) = \underleftarrow{\lim}\w{\gamma}_{\xi}(\pmb{f}_i).\]
Furthermore, $(\bfS_{\tau}(\II_{\itLamb}), \id, \gamma_{\xi})$ is an object in $\scrN^p$.
\end{lemma}

\begin{proof}
Since $\gamma_{\xi}$ is a $\itLamb$-isomorphism, it is clear that $\w{\gamma}_{\xi}$ is also a $\itLamb$-isomorphism.
Then, the commutativity of the inverse limit and the map $\w{\gamma}_{\xi}$ holds.
Thus, for any sequence $\{\pmb{f}_i\}_{i\in\NN^+}$ in $\bfS_{\tau}(\II_{\itLamb}){}^{\oplus_p 2^n}$, if its inverse limit exists,
then this inverse limit is also an element in $\w{\bfS_{\tau}(\II_{\itLamb})}{}^{\oplus_p 2^n}$, and so
\[ \gamma_{\xi} (\underleftarrow{\lim} \pmb{f}_i)
 = \w{\gamma}_{\xi} (\underleftarrow{\lim} \pmb{f}_i)
 \mathop{=}^{\spadesuit} \underleftarrow{\lim} \w{\gamma}_{\xi} (\pmb{f}_i)
 = \underleftarrow{\lim} \gamma_{\xi}(\pmb{f}_i), \]
where $\spadesuit$ holds since $\w{\gamma}_{\xi}$ is a $\itLamb$-isomorphism (see Lemma \ref{lemm:inverse gamma}\point).
Therefore, by Lemma \ref{lemm:S in Np pre}\point, $(\bfS_{\tau}(\II_{\itLamb}), \id, \gamma_{\xi})$ is an object in $\scrN^p$.
\end{proof}

\subsection{\texorpdfstring{$\w{\bfS_{\tau}(\II_{\itLamb})}$ is a direct limit}{}} \label{sssect:S=direclim}

Let $\normcat\itLamb$ be the category of normed $\itLamb$-modules and $\itLamb$-homomorphisms between them.
Then it is easy to check that all $E_u$ are objects in $\normcat\itLamb$.
Furthermore, for any $u\le v$, we have a $\itLamb$-homomorphism $\varphi_{uv}: E_u\to E_v$ which is induced by $E_u\subseteq E_v$.
Thus we obtain a direct system $((E_i)_{i\in\NN}, (\varphi_{uv})_{u\le v})$ in $\normcat\itLamb$ over $\NN$.
Let $\Ban(\itLamb)$ be the category of Banach $\itLamb$-modules and continuous $\itLamb$-homomorphisms between them.
Then $\Ban(\itLamb)$ is a full subcategory of $\normcat(\itLamb)$,
and so, naturally, we obtain a direct system $((E_i)_{i\in\NN}, (\varphi_{uv})_{u\le v})$ in $\Ban(\itLamb)$
if $\itLamb$ is a complete $\kk$-algebra.

The following lemma establishes the relation between $E_n$ and $\bfS(\II_{\itLamb})$.

\begin{lemma} \label{lemm:limE}
Let $\itLamb$ be a complete $\kk$-algebra.
Consider the category $\Ban(\itLamb)$ and take $(\alpha_i: E_i \to \w{\bfS_{\tau}(\II_{\itLamb})})_{i\in\NN}$,
where every $\alpha_i$ is the embedding given by $E_i\subseteq \w{\bfS_{\tau}(\II_{\itLamb})}$.
Then $\underrightarrow{\lim} E_i \cong \w{\bfS_{\tau}(\II_{\itLamb})}$.
\end{lemma}


\begin{proof}
Let $X$ be an arbitrary object in $\normcat\itLamb$ such that there is $(f_i: E_i\to X)_{i\in\NN}$ satisfying
$f_i\varphi_{ij}=f_j$ for all $i\le j$.
Then we can find the $\itLamb$-homomorphism $\theta: \w{\bfS_{\tau}(\II_{\itLamb})} \to X$ in the following way.

For any $x\in \w{\bfS_{\tau}(\II_{\itLamb})}$, there exists a sequence $\{x_t\}_{t\in\NN}$ in $\bigcup_{i}E_i$
such that $\{\Vert x_t - x\Vert_p\}_{t}$ is a monotonically decreasing sequence of positive real numbers.
Then we have $\underleftarrow{\lim} \{\Vert x_t - x\Vert_p\}_{t} = 0$ by Example \ref{exp:lim}\point~
which induces $\underleftarrow{\lim} x_t = x.$
Since $\alpha_i$, $\alpha_j$ and $\varphi_{ij}$ ($\forall i\le j$) are $\itLamb$-homomorphisms induced by $\subseteq$
(thus they are $\kk$-linear maps induced by $\subseteq$)
and every $x_t$ has a preimage in some $E_{u(t)}$,
then $\itLamb$-homomorphisms $(f_i)_{i\in\NN}$ send $\{x_t\}_{t\in\NN}$ to a sequence $\{f_{u(t)}(x_t)\}_{t\in\NN}$ in $X$.
By the completeness of $X$, $\underleftarrow{\lim} f_{u(t)}(x_t) \in X$ holds. Define
\[\theta(x) = \underleftarrow{\lim} f_{u(t)}(x_t) = \underleftarrow{\lim} f|_{E_{u(t)}}(x_t)
   = \underleftarrow{\lim} f(x_t), \]
where $f$ is the map $\underleftarrow{\lim}E_u \to X$ induced by the direct limit of $((E_i)_{i\in\NN}, (\varphi_{uv})_{u\le v})$.
Then one can check that $\theta$ is well-defined and is a $\itLamb$-homomorphism making the following diagram commute.
\[\xymatrix{
   \w{\bfS_{\tau}(\II_{\itLamb})} \ar@{-->}[rr]^{\theta} & & X \\
 & E_i \ar[lu]^{\alpha_i} \ar[ru]_{f_i} \ar[d]^(0.35){\varphi_{ij}}_(0.35){(i\preceq j)}  & \\
 & E_j \ar@/^1pc/[luu]^{\alpha_j} \ar@/_1pc/[ruu]_{f_j}  &
}\]

Next, we show that the existence of $\theta$ is unique.
Assume that $\theta'$ is also a $\itLamb$-homomorphism with $\theta'\alpha_i=f_i$ for all $i$.
Note that all morphisms in $\Ban(\itLamb)$ are continuous,
which ensures the commutativity $\underleftarrow{\lim}\vartheta(x_i) = \vartheta(\underleftarrow{\lim}x_i)$
between the inverse limit and any morphism $\vartheta$ starting from $\w{\bfS_{\tau}(\II_{\itLamb})}$.
Then for any $x\in\w{\bfS_{\tau}(\II_{\itLamb})}$,
taking the sequence $\{x_t\}_{t\in\NN}$ in $\bigcup_iE_i$ satisfying $\underleftarrow{\lim} x_t = x$,
we have \[   \theta'(x)=\theta'\big(\underleftarrow{\lim}\alpha_i( x_t)\big)
           = \underleftarrow{\lim}\theta'(\alpha_i( x_t))
           = \underleftarrow{\lim} f_i(x_t)
           = \underleftarrow{\lim}\theta (\alpha_i(x_t))
           = \theta \big(\underleftarrow{\lim}\alpha_i(x_t)\big) = \theta(x),\]
\checks{i.e.,} $\theta=\theta'$. Therefore, by the definition of direct limits, we have
$\underrightarrow{\lim} E_i \cong \w{\bfS_{\tau}(\II_{\itLamb})}$.
\end{proof}

\section{\texorpdfstring{The $\scrA^p$-initial object in $\scrN^p$}{}} \label{sect:Ap-init obj}

Let $\mathcal{C}$ be a category.
Recall that an object $O$ in $\mathcal{C}$ is an {\defines initial object}
if for any object $Y$, we have $\Hom_{\mathcal{C}}(O, Y)$ contains only one morphism,
\checks{i.e.,} there is a unique morphism $O\to Y$ in $\mathcal{C}$.
Obviously, if $\mathcal{C}$ has initial objects, then the initial object is unique up to isomorphism,
see \cite[Chapter 5, Lemma 5.3]{R1979}.
Let $\mathcal{D}$ be a full subcategory of $\mathcal{C}$.
An object $C \in \mathcal{C}$ is called a {\defines $\mathcal{D}$-initial object}
if for any $D\in\mathcal{D}$, there is a unique morphism $h\in \Hom_{\mathcal{C}}(C,D)$ such that
the following diagram
\[
\xymatrix@C=1.5cm{
 C \ar[r]^{h} \ar[d]_{\subseteq} & D \\
 D' \ar[ru]_{h'} &
}
\]
commutes, where $D'$ is an initial object in $\mathcal{D}$ and $h'$ is a morphism in $\mathcal{D}$.
See \cite[Page 216]{R1979}.
It is trivial that an initial object in $\mathcal{C}$ is a $\mathcal{C}$-initial object.

Let $\itLamb$ be a complete $\kk$-algebra. In this section,
we show that $(\w{\bfS_{\tau}(\II_{\itLamb})}, \id, \w{\gamma}_{\xi})$ is an $\scrA^p$-initial object in $\scrN^p$.
The proof is divided into two parts: (1) there is at least one morphism from
$(\w{\bfS_{\tau}(\II_{\itLamb})}, \id, \w{\gamma}_{\xi})$ to any object in $\scrA^p$;
(2) the above morphism is unique.

\subsection{
\texorpdfstring{The existence of morphism from $(\w{\bfS_{\tau}(\II_{\itLamb})}, \id, \w{\gamma}_{\xi})$}{}
}

In this subsection we show that $\Hom_{\scrA^p}((\w{\bfS_{\tau}(\II_{\itLamb})}, \id, \w{\gamma}_{\xi}), (V, v, \delta))$
is not empty for every object $(V, v, \delta)$ in $\scrA^p$.

\begin{lemma} \label{lemm:existence}
For any object $(V, v, \delta) \in \scrA^p$, we have
\[\Hom_{\scrA^p}((\w{\bfS_{\tau}(\II_{\itLamb})}, \id, \w{\gamma}_{\xi}), (V, v, \delta))\ne \varnothing.\]
\end{lemma}

\begin{proof}
For each $u\in\NN$, consider the map $\theta_u: E_u\to V$ as follows:
\begin{itemize}
  \item[(i)]
    $\theta_0: E_0 \to V$ is a map induced by the $\kk$-linear map $\kk\to V$ sending $1$ to $v$ (note that $E_0\cong\kk$).
    Then one can check that $\theta$ is a $\itLamb$-homomorphism.
  \item[(ii)]
    $\theta_{u+1}$ is induced by $\theta_{u}$ through the composition
    \[\theta_{u+1}:= \bigg( E_{u+1}
      \To{\gamma_{\xi}^{-1}} E_u^{\oplus_p 2^n}
      \To{\theta_u^{\oplus 2^n}} V^{\oplus_p 2^n}
      \To{\delta} V \bigg),\]
    where the inverse $\gamma_{\xi}^{-1}$ of the map $\gamma_{\xi}$ is given in Lemma \ref{lemm:inverse gamma}.
\end{itemize}
Notice that $\gamma_{\xi}^{-1}(f) \in E_{u-1}$ for any $f\in E_u\subseteq E_{u+1}$, then for the case $u=0$,
we have that $f = k\id_{E_0}$ is a constant defined on $E_0$, and
\[ \theta_1(f) = \delta (\theta_0^{\oplus 2^n} (\gamma_{\xi}^{-1}(f)))
=\delta(\theta_0(k\id_{E_0}), \theta_0(k\id_{E_0}), \ldots,  \theta_0(k\id_{E_0})) = kv, \]
\checks{i.e.,} $\theta_1$ is an extension of $\theta_0$. It yields $\theta_1(\id_{E_1})=v$ by $\theta_0(\id_{E_0})=v$ (see (i)).
Furthermore, we can check that $\theta_{u+1}$ is an extension of $\theta_u$ and
\begin{align}\label{formula:theta(1)=v}
 \theta_{u}(\id_{E_u}) = v\ (\forall u\in \NN)
\end{align}
by induction, \checks{i.e.,} the following diagram
\[\xymatrix{
   \underrightarrow{\lim}E_i  & & V \\
 & E_{u} \ar@{^(->}[lu]^(0.35){\alpha_u} \ar[ru]_{\theta_u} \ar@{^(->}[d]|{\alpha_{u\ u+1}}  & \\
 & E_{u+1} \ar@{^(->}@/^1pc/[luu]^{\alpha_{u+1}} \ar@/_1pc/[ruu]_{\theta_{u+1}}  &
}\]
commutes, where $\alpha_i: E_i\to \underrightarrow{\lim}E_i$ and $\alpha_{ij}: E_i\to E_j$ ($i\le j$)
are the embeddings induced by $E_i\subseteq \underrightarrow{\lim}E_i$ and $E_i\subseteq E_j$, respectively.
Then, for any $i\le j$, there is a unique $\itLamb$-homomorphism $\theta$ such that the  following diagram
\[\xymatrix{
   \underrightarrow{\lim}E_i \ar@{-->}[rr]^{\theta} & & V \\
 & E_{i} \ar@{^(->}[lu]^(0.35){\alpha_i} \ar[ru]_{\theta_i} \ar@{^(->}[d]|{\alpha_{ij}}  & \\
 & E_{j} \ar@{^(->}@/^1pc/[luu]^{\alpha_j} \ar@/_1pc/[ruu]_{\theta_j}  &
}\]
commutes. By Lemma \ref{lemm:limE}\point, we have that
$\theta: \underrightarrow{\lim}E_i\cong \w{\bfS_{\tau}(\II_{\itLamb})} \to V$ is a $\itLamb$-homomorphism in
$\Hom_{\itLamb}(\w{\bfS_{\tau}(\II_{\itLamb})}, V)$.

Next, we prove that $\theta$ is a morphism in $\scrN^p$.
First of all, 
we have
\[\theta(\id) 
= \underleftarrow{\lim}\theta|_{E_i}(\id_{E_i})
= \underleftarrow{\lim}\theta(\alpha_i(\id_{E_i})) = \underleftarrow{\lim} \theta_i(\id_{E_i})
\mathop{=\!=\!=}\limits^{(\ref{formula:theta(1)=v})} \underleftarrow{\lim} v = v. \]
In the following, we show that the following diagram
\begin{align}\label{mainresult:diagram}
\xymatrix{
  \w{\bfS_{\tau}(\II_{\itLamb})}{}^{\oplus_p 2^n}
  \ar[r]^{\gamma_{\xi}}
  \ar[d]_{\theta^{\oplus 2^n}}
& \w{\bfS_{\tau}(\II_{\itLamb})}
  \ar[d]^{\theta} \\
  V^{\oplus_p 2^n}
  \ar[r]_{\delta} & V
}
\end{align}
commutes. Notice that each $\pmb{f}=(f_1,\ldots, f_{2^n})\in \w{\bfS_{\tau}(\II_{\itLamb})}{}^{\oplus_p 2^n}$
can be seen as the inverse limit $\underleftarrow{\lim}\pmb{f}_i$ of
some sequence $\{\pmb{f}_i=(f_{1i},\ldots, f_{2^ni})\}_{i\in\NN}$ in $\bigcup_{u\in\NN}E_u^{\oplus_p 2^n}$,
where $f_{ji} \in E_{u_i}$ ($1\le j\le 2^n$), $u_i \in \NN$, such that for any $i\le j$, we have $u_i \le u_j$.
Thus, naturally, we need to consider the following diagram:
\[\xymatrix{
  E_{u_i}^{\oplus_p 2^n}
  \ar[r]^{\gamma_{\xi}\big|_{E_{u_i}^{\oplus_p 2^n}}}_{\cong}
  \ar@{^(->}[d]_{e_{u_i}^{\oplus 2^n}}
  \ar@/_4pc/[dd]_{\theta_{u_i}^{\oplus 2^n}}
& E_{u_i+1}
  \ar@{^(->}[d]^{e_{u_i+1}}
  \ar@/^4pc/[dd]^{\theta_{u_i}}
\\
  \w{\bfS_{\tau}(\II_{\itLamb})}{}^{\oplus_p 2^n}
  \ar[r]^{\gamma_{\xi}}
  \ar[d]_{\theta^{\oplus 2^n}}
& \w{\bfS_{\tau}(\II_{\itLamb})}
  \ar[d]^{\theta}
\\
  V^{\oplus_p 2^n}
  \ar[r]_{\delta} & V,
}\]
where $\big(e_{u_i}: E_{u_i} \to \w{\bfS_{\tau}(\II_{\itLamb})}\big)$
is the embedding induced by $E_{u_i}\subseteq \w{\bfS_{\tau}(\II_{\itLamb})}$.
Since
\begin{align*}
    \theta(\gamma_{\xi}(\pmb{f}))
& 
  = \underleftarrow{\lim}\ \theta(\gamma_{\xi}( e^{\oplus 2^n}_{u_i}(\pmb{f}_i)))
    && \\
& = \underleftarrow{\lim}\ \theta(e_{u_i+1}(\gamma_{\xi}\big|_{E^{\oplus_p 2^n}}(\pmb{f}_i)))
    &&
        (
        \gamma_{\xi} e^{\oplus 2^n}_{u_i} = e_{u_i+1} \gamma_{\xi}\big|_{E^{\oplus_p 2^n}}) \\
& = \underleftarrow{\lim}\ \theta_{u_i}(\gamma_{\xi}\big|_{E^{\oplus_p 2^n}}(\pmb{f}_i))
    &&
        (
        \theta e =\theta_{u_i}) \\
& = \underleftarrow{\lim}\ \delta(\theta_{u_i}^{\oplus 2^n}(\pmb{f}_i))
    &&
        (
        \theta_{u_i} \gamma_{\xi}\big|_{E^{\oplus_p 2^n}} = \delta\theta_{u_i}^{\oplus 2^n}) \\
& = \underleftarrow{\lim}\ \delta(\theta^{\oplus 2^n}(e^{\oplus 2^n}_{u_i}(\pmb{f}_i)))
    &&
        (
        \theta_u^{\oplus 2^n} = \theta^{\oplus 2^n}e^{\oplus 2^n}_{u_i}) \\
& = \delta(\theta^{\oplus 2^n}(\underleftarrow{\lim}\ e^{\oplus 2^n}_{u_i+1}(\pmb{f}_i)))
  = \delta(\theta^{\oplus 2^n}(\pmb{f})),
    &&
       (\text{by } (\ref{formula:comm}))
\end{align*}
the assertion follows.
\end{proof}

\subsection{\texorpdfstring{The uniqueness of morphism from $(\w{\bfS_{\tau}(\II_{\itLamb})}, \id, \w{\gamma}_{\xi})$}{}}

Now, we show that, for any object $(V, v, \delta)$ in $\scrA^p$, if the morphism in the category $\scrA^p$
from $(\w{\bfS_{\tau}(\II_{\itLamb})}, \id, \w{\gamma}_{\xi})$ exists, then it is unique.

\begin{lemma} \label{lemm:uniqueness}
Let $(V, v, \delta) \in \scrA^p$ be an object in $\scrA^p$.
If \[\Hom_{\scrA^p}((\w{\bfS_{\tau}(\II_{\itLamb})}, \id, \w{\gamma}_{\xi}),(V, v, \delta))\ne \varnothing,\]
then $\Hom_{\scrA^p}((\w{\bfS_{\tau}(\II_{\itLamb})}, \id, \w{\gamma}_{\xi}),(V, v, \delta))$ contains a unique morphism.
\end{lemma}

\begin{proof}
Let $\theta$ and $\theta'$ be two $\itLamb$-homomorphisms
from $(\bfS_{\tau}(\II_{\itLamb}), \id, \gamma_{\xi})$ to $(V,v,\delta)$ in $\scrA^p$.
Then $\theta(\id) = v = \theta'(\id)$.
Since $\theta$ and $\theta'$ are maps in $\scrA^p$, the square
\[\xymatrix@C=1.5cm{
  E_u^{\oplus_p 2^n}
  \ar[r]^{\gamma_{\xi}|_{E_u^{\oplus 2^n}}}_{\cong}
  \ar[d]_{(\theta|_{E_u}-\theta'|_{E_u})^{\oplus 2^n}}
& E_{u+1}
  \ar[d]^{\theta|_{E_{u+1}}-\theta'|_{E_{u+1}}} \\
  V^{\oplus_p 2^n}
  \ar[r]_{\delta}
& V
}\]
commutes. Then for any $f\in E_{u+1}$, we have
\[ (\theta|_{E_{u+1}}-\theta'|_{E_{u+1}})(f)
= (\delta\circ (\theta|_{E_u}-\theta'|_{E_u})^{\oplus 2^n}\circ(\gamma_{\xi}|_{E_u^{\oplus 2^n}})^{-1})(f), \]
\checks{i.e.,} $\theta|_{E_{u+1}}-\theta'|_{E_{u+1}}$ is determined by $\theta|_{E_u}-\theta'|_{E_u}$.
Consider the case for $u=0$, since $\theta|_{E_0}$ and $\theta'|_{E_0}: E_0\to V$ are defined by $\theta_0(\id_{E_0})=v$,
we have \[(\theta|_{E_0}-\theta'|_{E_0})(k\id_{E_0}) = k(\theta|_{E_0}(\id_{E_0})-\theta'|_{E_0}(\id_{E_0})) = k(v-v)=0.\]
Therefore $\theta|_{E_u}-\theta'|_{E_u}=0$ for all $u\in\NN$ by induction.

On the other hand, consider the embeddings $e_u: E_u\to \w{\bfS_{\tau}(\II_{\itLamb})}$ and $e_{uv}:E_u\to E_v$ ($u\le v$)
induced by $\subseteq$ and the direct system
\[\big((E_u^{\oplus_p 2^n})_{u\in\NN}, (e_u^{\oplus 2^n} :E_u^{\oplus_p 2^n}\to \w{\bfS_{\tau}(\II_{\itLamb})}{}^{\oplus_p 2^n})_{u\in\NN}\big),\]
we have the following commutative diagram
\[\xymatrix{
   \w{\bfS_{\tau}(\II_{\itLamb})}{}^{\oplus_p 2^n}  & & V \\
 & E_{i}^{\oplus_p 2^n}
   \ar@{^(->}[lu]_{e_i^{\oplus 2^n}}
   \ar[ru] \ar@{}[ru]|
       {\rotatebox{40}{\tiny$\begin{smallmatrix}
                             \theta|_{E_i}-\theta'|_{E_i}
                             \\ \
                             \\ (=0)
                             \\ \ \end{smallmatrix}$}}
   \ar@{^(->}[d]^{e_{ij}^{\oplus 2^n}}
 & \\
 & E_{ij}^{\oplus_p 2^n}
   \ar@{^(->}@/^1pc/[luu]^{e_j^{\oplus 2^n}}
   \ar@/_1pc/[ruu]_{\theta|_{E_j}-\theta'|_{E_j}=0. }&
}\]
Since
\[ \underrightarrow{\lim} E_i^{\oplus_p 2^n} \cong (\underrightarrow{\lim} E_i)^{\oplus_p 2^n}
\cong \w{\bfS_{\tau}(\II_{\itLamb})}{}^{\oplus_p 2^n}, \]
there is a unique $\itLamb$-homomorphism
$\phi: \w{\bfS_{\tau}(\II_{\itLamb})}{}^{\oplus_p 2^n}\to V$ such that the following diagram
\[\xymatrix{
   \w{\bfS_{\tau}(\II_{\itLamb})}{}^{\oplus_p 2^n} \ar@{-->}[rr]^{\phi} & & V \\
 & E_{i}^{\oplus_p 2^n}
   \ar@{^(->}[lu]_{e_i^{\oplus 2^n}}
   \ar[ru] \ar@{}[ru]|
       {\rotatebox{40}{\tiny$\begin{smallmatrix}
                             \theta|_{E_i}-\theta'|_{E_i}
                             \\ \
                             \\ (=0)
                             \\ \ \end{smallmatrix}$}}
   \ar@{^(->}[d]^{e_{ij}^{\oplus 2^n}}
 & \\
 & E_{ij}^{\oplus_p 2^n}
   \ar@{^(->}@/^1pc/[luu]^{e_j^{\oplus 2^n}}
   \ar@/_1pc/[ruu]_{\theta|_{E_j}-\theta'|_{E_j}=0 }&
}\]
commutes. Since $(\theta-\theta')e_u^{\oplus 2^n} = \theta|_{E_i}-\theta'|_{E_j}$,
we know that the case for $\phi=\theta-\theta'$ makes the above diagram commute.
On the other hand, the case for $\phi=0$ makes the above diagram commute.
Thus $\theta-\theta'=0$ and $\theta=\theta'$.
\end{proof}

\subsection{\texorpdfstring{The $\scrA^p$-initial object in $\scrN^p$}{}}

\begin{lemma} \label{lemm:initial}
Let $\mathcal{C}$ be a category and $\mathcal{D}$ a subcategory of $\mathcal{C}$, and let $D'$ be an initial object in $\mathcal{D}$.
If an object $C$ is a subobject of $D'$ in $\mathcal{C}$, then $C$ is a $\mathcal{D}$-initial object.
\end{lemma}

\begin{proof}
For any object $D$ in $\mathcal{D}$, there is a unique morphism $h' \in \Hom_{\mathcal{D}}(D',D)$
since $D'$ an initial object in $\mathcal{D}$.
Let $e$ be the embedding $C \to D'$ obtained by $C$ being a subobject of $D'$.
Then we obtain a morphism $h'e\in\Hom_{\mathcal{C}}(C,D)$.
Next, assume that $h_0$ is any morphism in $\Hom_{\mathcal{C}}(C, D)$ such that the following diagram
\[
\xymatrix@C=1.5cm{
 C \ar[r]^{h_0} \ar[d]_{e} & D \\
 D' \ar[ru]_{h_0'} &
}
\]
commutes, where $h_0'$ is a morphism in $\mathcal{D}$.
Since $D'$ an initial object in $\mathcal{D}$,
we have $h_0'=h'$, and thus $h_0 = h_0' e = h'e$.
\end{proof}

Now, we can prove the following main result of this paper.

\begin{theorem} \label{thm:main1}
The triple $(\bfS_{\tau}(\II_{\itLamb}), \id, \gamma_{\xi})$ is an $\scrA^p$-initial object in $\scrN^p$.
\end{theorem}

\begin{proof}
For any object $(V, v, \delta)$ in $\scrA^p$, the existence of morphisms in
$\Hom_{\scrA^p}((\w{\bfS_{\tau}(\II_{\itLamb})}, \id, \w{\gamma}_{\xi}),$ $(V, v, \delta))$
is proved in Lemma \ref{lemm:existence}\point, and its uniqueness is proved in Lemma \ref{lemm:uniqueness}.
Thus, the triple $(\w{\bfS_{\tau}(\II_{\itLamb})}, \id, \w{\gamma}_{\xi})$,
as an object in $\scrA^p$, is an initial object in $\scrA^p$.
It follows that $(\bfS_{\tau}(\II_{\itLamb}), \id, \w{\gamma}_{\xi})$ is an $\scrA^p$-initial object in $\scrN^p$ by Lemma \ref{lemm:initial}.
\end{proof}

We give a remark for Theorem \ref{thm:main1}.

\begin{remark} \label{rmk:main1} \rm
For any object $(V,v,\delta)$ in $\scrA^p$, there is a unique morphism
\[h: (\bfS_{\tau}(\II_{\itLamb}), \id, \gamma_{\xi}) \to (V,v,\delta)\]
in $\scrN^p$, which can be extended to $\w{h}: (\w{\bfS_{\tau}(\II_{\itLamb})}, \id, \w{\gamma}_{\xi}) \to (V,v,\delta)$.
In other words, if there exists a morphism $h$ making the diagram
\[\xymatrix@C=2cm{
  (\bfS_{\tau}(\II_{\itLamb}), \id, \gamma_{\xi}) \ar[r]^h \ar[d]_{\subseteq}
& (V,v,\delta) \\
  (\w{\bfS_{\tau}(\II_{\itLamb})}, \id, \w{\gamma}_{\xi}) \ar[ru]_{\w{h}}
& }\]
commute, then the existence of $h$ is guaranteed to be unique.
\end{remark}

\section{\texorpdfstring{The categorification of integration}{}} \label{sect:7}

Take $\kk=(\kk,|\cdot|,\preceq)$ to be an extension of $\RR$ and $p=1$.
Recall the symbols given in Notation \ref{notation}\point, any step function $f$ in $E_u$ can be written as
\[f=\sum_{(u_iv_i)_i} k_{(u_iv_i)_i}\id_{(u_iv_i)_i}.  \]
We define the map $T_u:E_u \to \kk$ by
\begin{align}\label{formula:Tu(1)}
  T_u(f)= \sum_{(u_iv_i)_i} k_{(u_iv_i)}\mu\bigg(\prod\nolimits_{i}\II^{(u_iv_i)}\bigg)
\end{align}
\begin{center}
 (note that if all coefficients $k_{(u_iv_i)}$ are equal to $1$, then $T_u(f)=\mu(E_u)$).
\end{center}

Then the $\itLamb$-isomorphism $\gamma_{\xi}$ shown in Lemma \ref{lemm:inverse gamma}\point~points out the following fact:
there is a map $m_u: \kk^{\oplus_p 2^n}\to\kk$ such that the following diagram
\begin{align}\label{diagram:m}
\xymatrix{
  E_u^{\oplus_p 2^n} \ar[r]^{\gamma_{\xi}} \ar[d]_{T_u^{\oplus 2^n}}
& E_{u+1} \ar[d]^{T_{u+1}} \\
  \kk^{\oplus_p 2^n} \ar[r]_{m_u}
& \kk
}
\end{align}
commutes.
Indeed, for the function $f_k = \frac{k}{\mu(\II_{\itLamb})}\id_{\II_{\itLamb}}$ with $k\in\kk$, we have
\begin{center}
  $ T_u(f_k)=T_u(\frac{k}{\mu(\II_{\itLamb})} \id_{\II_{\itLamb}})
    = \frac{k}{\mu(\II_{\itLamb})} T_u(\id_{\II_{\itLamb}}) = k$
\end{center}
by (\ref{formula:Tu(1)}).
Then for any $\pmb{k}=(k_1,\ldots, k_{2^n})\in \kk^{\oplus_p 2^n}$,
$\pmb{f}_{\pmb{k}} = (f_{k_1},\ldots,f_{k_{2^n}})\in E_u^{\oplus_p 2^n}$
is a preimage of $\pmb{k}$ under the $\kk$-linear map $T_u^{\oplus 2^n}$.
We define $\mu_u$ as follows:
\[
m_u(\pmb{k}) = T_{u+1}(\gamma_{\xi}(\pmb{f}_{\pmb{k}})).
\]
It is easy to see that $m_u$ is a $\kk$-linear map.
In particular, for the constant function $\id_{\II_{\itLamb}}$ given by the measure $\mu(\II_{\itLamb})$ of $\II_{\itLamb}$,
we obtain that $f_{\mu(E_u)}$ is a preimage of $\mu(\II_{\itLamb})\in\kk$, and then
\[ m_u(\mu(\II_{\itLamb}),\ldots,\mu(\II_{\itLamb})) = T_{u+1}\gamma_{\xi}(\id_{\II_{\itLamb}},\ldots,\id_{\II_{\itLamb}})
= \sum_{(u_iv_i)_i} 1 \cdot \mu\bigg(\prod\nolimits_{i}\II^{(u_iv_i)}\bigg)
= \mu(\II_{\itLamb}). \]

\begin{lemma} \label{lemm:Tu}
Let $\kk=(\kk,|\cdot|,\preceq)$ be an extension of $\RR$.
Then $T_u: E_u \to \kk$ is a $\itLamb$-homomorphism.
\end{lemma}

\begin{proof}
Note that $\kk$ is a $\itLamb$-module given by
\[\itLamb\times\kk\to \kk, (\lambda, k)\mapsto \lambda\cdot k:=\tau(\lambda)k. \]
For arbitrary two elements $\lambda_1,\lambda_2 \in\itLamb$ and arbitrary two functions
$f=\sum_ik_i\id_{I_i}$, $g=\sum_j k_j'\id_{I_j'}\in E_u$, we have
\begin{align*}
  T_u(\lambda_1\cdot f+\lambda_2\cdot g)
& = T_u\bigg(\sum\nolimits_i\tau(\lambda_1)k_i\id_{I_i}+\sum\nolimits_j \tau(\lambda_2)k_j'\id_{I_j'}\bigg) \\
& = \tau(\lambda_1)T_u\bigg(\sum\nolimits_ik_i\id_{I_i}\bigg)+\tau(\lambda_2)T_u\bigg(\sum\nolimits_jk_j'\id_{I_j'}\bigg) \\
& = \tau(\lambda_1)T_u(f)+\tau(\lambda_2)T_u(g) \\
& = \lambda_1\cdot T_u(f) + \lambda_2\cdot T_u(g).
\end{align*}
\end{proof}

\begin{lemma} \label{lemm:mu}
Let $\kk=(\kk,|\cdot|,\preceq)$ be an extension of $\RR$
and let $m_u$ be the $\kk$-linear map given in the diagram (\ref{diagram:m}).
Then $m_u$ is a $\itLamb$-homomorphism.
\end{lemma}

\begin{proof}
We can prove that $m_u$ is a $\itLamb$-homomorphism by using an argument similar to
proving that $T_u$ is a $\kk$-linear mapping as in Lemma \ref{lemm:Tu}.
\end{proof}

\begin{remark} \label{rmk:m} \rm
Since $E_0\subseteq E_1\subseteq \cdots \subseteq E_u \subseteq \cdots \subseteq \bfS_{\tau}(\II_{\itLamb})
\subseteq \w{\bfS_{\tau}(\II_{\itLamb})} = \underrightarrow{\lim}E_i$, we have that $\mu$ is independent on $u$.
Thus, we can use $m$ to present all maps $m_u$ ($u\in \NN$) because $m_0=m_1=m_2=\ldots$.
\end{remark}

\begin{proposition}
Let $\kk=(\kk,|\cdot|,\preceq)$ be an extension of $\RR$.
Then the triple $(\kk, \mu(\II_{\itLamb}), m)$ is an object in $\scrN^p$.
Furthermore, since $\itLamb$ is complete, so is $\kk$. Then $\kk^{\oplus_p 2n}$ is a Banach $\itLamb$-module,
and so $(\kk, \mu(\II_{\itLamb}), m)$ is an object in $\scrA^p$.
\end{proposition}

\begin{proof}
It follows from Lemmas \ref{lemm:Tu}\point~and \ref{lemm:mu}\point~and Remark \ref{rmk:m}.
\end{proof}

The following proposition shows that $T_u$ satisfies the triangle inequality.

\begin{proposition} \label{prop:tri inq}
If $\kk=(\kk,|\cdot|,\preceq)$ is an extension of $\RR$,
then for any $f\in E_u$, the following inequality holds for all $u\in\NN$.
\begin{align}\label{formula:tri inq}
  |T_u(f)|  \le T_u(|f|).
\end{align}
\end{proposition}

\begin{proof}
Assume that $f=\sum_i k_i\id_{I_i} \in E_u$, where $I_i\cap I_j = \varnothing$ for all $i\ne j$.
Then $|f| = |\sum_i k_i\id_{I_i}|$ is also a step function in $E_u$,
and we have
\begin{align*}
  T_u(|f|)
& = T_u \Big(\Big| \sum\nolimits_i k_i\id_{I_i} \Big|\Big)
  \mathop{=\!=}\limits^{(\star)} T_u\Big( \sum\nolimits_i |k_i| \id_{I_i}\Big) \\
& = \sum\nolimits_i |k_i|\mu\Big(\prod\nolimits_i\II^{(u_iv_i)_i}\Big) \\
& \ge \Big| \sum\nolimits_i k_i \mu\Big(\prod\nolimits_i\II^{(u_iv_i)_i}\Big) \Big|
  = |T_u(f)|,
\end{align*}
where $(\star)$ is given by $I_i\cap I_j = \varnothing$.
\end{proof}

\begin{theorem} \label{thm:main2}
If $\kk=(\kk,|\cdot|,\preceq)$ is an extension of $\RR$, then there exists a unique morphism
\[T: (\bfS_{\tau}(\II_{\itLamb}), \id, \gamma_{\xi}) \to (\kk, \mu(\II_{\itLamb}), m)\]
in $\Hom_{\scrN^p}((\bfS_{\tau}(\II_{\itLamb}), \id, \gamma_{\xi}), (\kk, \mu(\II_{\itLamb}), m))$ such that the diagram
\[\xymatrix@C=2cm{
  (\bfS_{\tau}(\II_{\itLamb}), \id, \gamma_{\xi}) \ar[r]^{T} \ar[d]_{\subseteq}
& (\kk, \mu(\II_{\itLamb}),m) \\
  (\w{\bfS_{\tau}(\II_{\itLamb})}, \id, \w{\gamma}_{\xi}) \ar[ru]_{\w{T}}
& }\]
commutes, where $\w{T}$ is the is the unique extension of $T$ lying in
$\Hom_{\scrA^p}((\w{\bfS_{\tau}(\II_{\itLamb})}, \id, \w{\gamma}_{\xi}),$ $(\kk, \mu(\II_{\itLamb}), m))$.
Furthermore, $\w{T}$ is given by the direct limit $\underrightarrow{\lim}T_i: \underrightarrow{\lim} E_i \to \kk$.
\end{theorem}

\begin{proof}
Denote by $\alpha_{ij}:E_i\to E_j$ ($i\le j$) and $\alpha_i:E_i\to  \underrightarrow{\lim}E_i$ the monomorphism induced by
$E_i\subseteq E_j\subseteq \underrightarrow{\lim}E_i$. Then there is a unique morphism $\underrightarrow{\lim}T_i: \underrightarrow{\lim} E_i \to \kk$
such that the following diagram
\[\xymatrix{
   \underrightarrow{\lim}E_i \ar@{-->}[rr]^{\underrightarrow{\lim}T_i} & & \kk \\
 & E_{i} \ar@{^(->}[lu]^(0.35){\alpha_i} \ar[ru]_{T_i} \ar@{^(->}[d]|{\alpha_{ij}}  & \\
 & E_{j} \ar@{^(->}@/^1pc/[luu]^{\alpha_j} \ar@/_1pc/[ruu]_{T_j}  &
}\]
commutes. By Lemma \ref{lemm:limE}\point, we have $\underrightarrow{\lim}E_i \cong \w{\bfS_{\tau}(\II_{\itLamb})}$,
then $\underrightarrow{\lim}T_i$ induces a morphism in $\scrA^p$ from
$(\bfS_{\tau}(\II_{\itLamb}), \id, \gamma_{\xi})$ to $(\kk, \mu(\II_{\itLamb}), m)$.
Theorem \ref{thm:main1}\point~and Remark \ref{rmk:main1} show that $\underrightarrow{\lim}T_i=\w{T}$ and $T=\w{T}|_{\bfS_{\tau}(\II_{\itLamb})}$.
\end{proof}

\begin{definition} \rm \label{def:cat-int}
Let $\kk$ be a field with the norm $|\cdot|:\kk\to\RR^{\ge0}$ and the total ordered $\preceq$,
and let $f: \II_{\itLamb}\to\kk$ be a function in $\w{\bfS_{\tau}(\II_{\itLamb})}$.
We call that $f$ is an {\defines integrable function} on $\II_{\itLamb}$
and {its} integral, denoted by $\displaystyle (\scrA^1)\int_{\II_{\itLamb}}f\mathrm{d}\mu$, is defined as follows:
\[ (\scrA^1)\int_{\II_{\itLamb}} f \mathrm{d}\mu := \w{T}(f). \]
\end{definition}

By using the limit $\underrightarrow{\lim}T_i: \underrightarrow{\lim} E_i \to \kk$ given in Theorem \ref{thm:main2}\point,
the formula (\ref{formula:Tu(1)}), Lemma \ref{lemm:Tu}\point~and Proposition \ref{prop:tri inq}\point~show that
\[ (\scrA^1) \int_{\II_{\itLamb}} 1\dd\mu = \mu(\II_{\itLamb}) , \]
\[ (\scrA^1) \int_{\II_{\itLamb}} (\lambda_1 \cdot f_1 + \lambda_2 \cdot f_2) \mu
= \lambda_1 \cdot (\scrA^1)\int_{\II_{\itLamb}} f_1 \mu  + \lambda_2 \cdot (\scrA^1)\int_{\II_{\itLamb}} f_2 \mu \ \ (\lambda_1,\lambda_2\in \itLamb)  \]
and
\[ \bigg|(\scrA^1)\int_{\II_{\itLamb}} f \dd\mu \bigg| \le (\scrA^1)\int_{\II_{\itLamb}} |f| \dd\mu, \]
respectively.

In Subsection \ref{subsect:app1}\point, we point out that Theorem \ref{thm:main2}\point~
and Definition \ref{def:cat-int}\point~provide a categorification of Lebesgue integration.

\section{Series expansions of functions} \label{sect:series}

Set $n:=\dim_{\kk}\itLamb$, and define the {\defines $n$ variables polynomial ring} $\kk[X_1, \cdots, X_N]$ ($=\kk[\pmb{X}]$ for short)
over $\kk$ to be the set of all $N$ variables polynomial rings ($N\ge n$).
Then $\kk[\pmb{X}]$ is a left $\itLamb$-module whose left $\itLamb$-action is defined as
\[ \itLamb \times \kk[\pmb{X}] \to \kk[\pmb{X}], (a,P(x)) \mapsto \tau(a)P(x). \]

\subsection{Realizing power series expansions of functions as morphisms in $\scrA^p$} \label{subsubsect:power series}
Take $N=n$. In this part, we define the map
\begin{align}\label{formula: poly norm 1}
  \Vert \cdot \Vert: \kk[\pmb{X}] \to \RR^{\ge 0}, \
   P \mapsto \left((\scrA^p)\int_{\II_{\itLamb}}|P|^p\dd\mu \right)^{\frac{1}{p}},
\end{align}
where $|P|$ is defined by the norm $|\cdot|: \kk \to \RR^{\ge 0}$ defined on $\kk$ and $|P(\pmb{X})|$ for any $\pmb{X}\in\II_{\itLamb} \subseteq \itLamb$.

\begin{lemma} \label{lemm:power series}
The polynomial ring $\kk[\pmb{X}]$ with the map (\ref{formula: poly norm 1}) is a normed left $\itLamb$-module.
\end{lemma}

\begin{proof}
Each polynomial can be seen as a function lying in $\w{\bfS_{\tau}(\II_{\itLamb})}$.
Then, by using the norm $|\cdot|: \kk \to \RR^{\ge 0}$, the map (\ref{formula: poly norm 1})
induces a norm as required since $\Vert a\cdot P\Vert = \Vert \tau(a)P\Vert = |\tau(a)|\cdot \Vert P\Vert$.
\end{proof}

By using Lemma \ref{lemm:power series}\point, the Banach left $\itLamb$-module $\w{\kk[\pmb{X}]}$, as the completion of $\kk[\pmb{X}]$,
provides a triple $(\w{\kk[\pmb{X}]}, \id, \w{\gamma_{\xi}}|_{\w{\kk[\pmb{X}]}})$ which is an object of $\scrA^1$.
Thus, by Theorem \ref{thm:main2}\point, the following result holds.

\begin{corollary}[Weierstrass Approximation Theorem] \label{coro:power series}
There exists a unique morphism
\[ \expaP: (\bfS_{\tau}(\II_{\itLamb}), \id, \gamma_{\xi})
\to (\w{\kk[\pmb{X}]}, \id, \w{\gamma_{\xi}}|_{\w{\kk[\pmb{X}]}}) \]
in $\Hom_{\scrN^1}((\bfS_{\tau}(\II_{\itLamb}), \id, \w{\gamma}_{\xi}),
(\w{\kk[\pmb{X}]}, \id, \w{\gamma_{\xi}}|_{\w{\kk[\pmb{X}]}}) )$ such that the following diagram
\[\xymatrix@C=2cm{
  (\bfS_{\tau}(\II_{\itLamb}), \id, \gamma_{\xi}) \ar[r]^{\expaP} \ar[d]_{\subseteq}
& (\w{\kk[\pmb{X}]}, \id, \w{\gamma_{\xi}}|_{\w{\kk[\pmb{X}]}}) \\
  (\w{\bfS_{\tau}(\II_{\itLamb})}, \id, \w{\gamma}_{\xi}) \ar[ru]_{\w{\expaP}}
& }\]
commutes, where $\w{\expaP}$ is the unique extension of $\expaP$ lying in
$\Hom_{\scrA^1}((\w{\bfS_{\tau}(\II_{\itLamb})}, \id, \w{\gamma}_{\xi}),$ $(\w{\kk[\pmb{X}]},$ $\id, \w{\gamma_{\xi}}|_{\w{\kk[\pmb{X}]}}))$.
\end{corollary}

The above corollary shows that for any function $f \in \w{\bfS_{\tau}(\II_{\itLamb})}$,
there exists a sequence $\{P_i\}_{i\in\NN}$ of polynomials such that
\[ \w{\expaP}(f) = \underleftarrow{\lim} P_i \in \w{\kk[\pmb{X}]} \subseteq \kk[[\pmb{X}]]. \]
This formula is called a {\defines power series expansion} of $f$.

\begin{remark} \rm
In the case for $N=2n$, if $\kk[\pmb{X}] = \kk[Y_j, Y_j^{-1} \mid 1\le j \le n]$,
where $X_u = Y_u$ holds for any $1\le u\le n$, and $X_{n+v}=Y_v^{-1}$ holds for any $1\le v\le n$,
then we can obtain the Laurent series of analytic functions.
\end{remark}

\subsection{Realizing Fourier series expansions of functions as morphisms in $\scrA^p$} \label{subsubsect:Fourier series}
Consider the case for $N=2n$ and $\kk=\CC$ in this part.
Let $\tri$ be the $\CC$-linear map
\begin{center}
  $\tri: \CC[\pmb{X}] \to \CC[\power^{\pm2\pi\rmi\pmb{X}}]
    := \CC[\power^{\pm 2\pi \rmi X_j}\mid 1\le j\le n]$
\end{center}
induced by
\begin{align}\label{formula:Fourier}
  X_j \mapsto
\begin{cases}
 \power^{ 2\pi \rmi X_j}, & \text{if  } 1\le j\le n; \\
 \power^{-2\pi \rmi X_j}, & \text{if  } n+1\le j\le 2n,
\end{cases}
\end{align}
and define the map
\begin{align}\label{formula: poly norm 2}
  \Vert \cdot \Vert: \CC[\pmb{X}] \to \RR^{\ge 0}, \
   P \mapsto \left((\scrA^p)\int_{\II_{\itLamb}}|\tri(P)|^p\dd\Leb \right)^{\frac{1}{p}}.
\end{align}

\begin{lemma} \label{lemm:Fourier series}
The $\CC$-linear map $\tri$ is a $\itLamb$-isomorphism,
and $\CC[\pmb{X}] \cong \CC[\power^{\pm2\pi\rmi\pmb{X}}]$ with the map (\ref{formula: poly norm 2}) is a normed left $\itLamb$-module.
\end{lemma}

\begin{proof}
It is trivial that $\tri$ is a $\CC$-linear isomorphism by (\ref{formula:Fourier}).
Thus, the assertion that $\tri$ is a $\itLamb$-isomorphism follows from the fact that the following formula
\[ \tri(a\cdot P) = \tri(\tau(a)P) = \tau(a)\tri(P) = a\cdot \tri(P) \]
holds for any $a\in \itLamb$.
Furthermore, we can prove that the polynomial ring $\kk[\pmb{X}]$ with
the map (\ref{formula: poly norm 2}) is a normed left $\itLamb$-module
by the way similar to that in Lemma \ref{lemm:power series}.
\end{proof}

Next, by Lemma \ref{lemm:Fourier series}\point, we obtain that
\[(\w{\CC[\pmb{X}]}, \id,
    \w{\gamma_{\xi}}|_{\w{\CC[\power^{\pm2\pi\rmi\pmb{X}}] }}
         )
  \cong
  (\w{\CC[\power^{\pm2\pi\rmi\pmb{X}}] }, \id,
    \w{\gamma_{\xi}}|_{\w{\CC[\power^{\pm2\pi\rmi\pmb{X}}] }}
          )
  \]
is an object in $\scrA^p$. Then the following corollary follows from Theorem \ref{thm:main2}.

\begin{corollary} \label{coro:Fourier series}
There exists a unique morphism
\[ \expaF: (\bfS_{\tau}(\II_{\itLamb}), \id, \gamma_{\xi})
\to (\w{\CC[\power^{\pm2\pi\rmi\pmb{X}}]}, \id,
     \w{\gamma_{\xi}}|_{\w{\CC[\power^{\pm2\pi\rmi\pmb{X}}] }}) \]
in $\Hom_{\scrN^1}((\bfS_{\tau}(\II_{\itLamb}), \id, \w{\gamma}_{\xi}),
(\w{\CC[\power^{\pm2\pi\rmi\pmb{X}}] }, \id, \w{\gamma_{\xi}}|_{\w{\CC[\power^{\pm2\pi\rmi\pmb{X}}] }}) )$
such that the following diagram
\[\xymatrix@C=2cm{
  (\bfS_{\tau}(\II_{\itLamb}), \id, \gamma_{\xi}) \ar[r]^{\expaF} \ar[d]_{\subseteq}
& (\w{\CC[\power^{\pm2\pi\rmi\pmb{X}}] }, \id, \w{\gamma_{\xi}}|_{\w{\CC[\power^{\pm2\pi\rmi\pmb{X}}] }}) \\
  (\w{\bfS_{\tau}(\II_{\itLamb})}, \id, \w{\gamma}_{\xi}) \ar[ru]_{\w{\expaF}}
& }\]
commutes, where $\w{\expaF}$ is the unique extension of $\expaF$ lying in
$\Hom_{\scrA^1}((\w{\bfS_{\tau}(\II_{\itLamb})}, \id, \w{\gamma}_{\xi}),$
$(\w{\CC[\power^{\pm2\pi\rmi\pmb{X}}] },$ $\id, \w{\gamma_{\xi}}|_{\w{\CC[\power^{\pm2\pi\rmi\pmb{X}}] }}))$.
\end{corollary}

The above corollary shows that for any function $f \in \w{\bfS_{\tau}(\II_{\itLamb})}$,
there exists a sequence $\{P_i\}_{i\in\NN}$ of triangulated polynomials such that
\[ \w{\expaF}(f) = \underleftarrow{\lim} P_i \in \w{\CC[\power^{\pm2\pi\rmi\pmb{X}}] }. \]
This formula is called a {\defines Fourier series expansion} of $f$.

\subsection{\texorpdfstring{Stone$-$Weierstrass Theorem in $\scrA^p$}{}} \label{subsect:S-W thm}

Let $\bfW_0$ be a normed left $\itLamb$-module generated by some functions lying in $\w{\bfS_{\tau}(\II_{\itLamb})}$
such that $\w{\bfW_0}$ and $\w{\bfS_{\tau}(\II_{\itLamb})}$, as left $\itLamb$-modules, are isomorphic preserving $\id$.
For any $u\in \NN$, define
\[ \bfW_u = \{ \w{\gamma}_{\xi}|_{\bfW_{u-1}}(\pmb{f})
   \mid \pmb{f}=(f_1,\cdots, f_{2^n})
      \in \bfW_{u-1}^{\oplus_p 2^n} \}. \]
Then we obtain a family of canonical embedding
\[ \bfW_0 \To{\subseteq} \bfW_1 \To{\subseteq} \cdots \To{\subseteq} \bfW_u \To{\subseteq} \cdots \ (\subseteq \w{\bfS_{\tau}(\II_{\itLamb})}), \]
which induced a direct limits
\[ \underrightarrow{\lim} \bfW_u =: \bfW. \]

\begin{lemma} \label{lemma:SW1}
For any complete extension $\bfW_{\ext}$ of $\bfW$, i.e., the Banach $\itLamb$-module satisfying $\bfW\subseteq \bfW_{\ext}$,
there exists a $\itLamb$-monomorphism
\[ \w{\expaSW}: \w{\bfS_{\tau}(\II_{\itLamb})} \to \bfW_{\ext} \]
between two left $\itLamb$-modules $\w{\bfS_{\tau}(\II_{\itLamb})}$ and $\bfW$
such that $\expaSW(\id) = \id$ holds in the case for $\id\in\bfW$.
\end{lemma}

\begin{proof}
Since $\bfW_i \subseteq \bfW_j \subseteq \bfW$ for any $i, j\in\NN$ with $i\le j$, we have
$\w{\bfW}_i \subseteq \w{\bfW}_j \subseteq \w{\bfW}$.
On the other hand, $\bfW \subseteq \w{\bfS_{\tau}(\II_{\itLamb})}$ yields $\w{\bfW} \subseteq \w{\bfS_{\tau}(\II_{\itLamb})}.$
It follows that
\[ \w{\bfS_{\tau}(\II_{\itLamb})} \cong \w{\bfW}_0 \subseteq \w{\bfW}_u \subseteq \w{\bfW} \subseteq  \w{\bfS_{\tau}(\II_{\itLamb})}.\]
Therefore, we get a $\itLamb$-isomorphism $\w{\bfS_{\tau}(\II_{\itLamb})} \cong \bfW$ ($=\w{\bfW}$)
since the isomorphism $\w{\bfS_{\tau}(\II_{\itLamb})} \cong \w{\bfW}_0$ preserves $\id$.
The composition
\[ \w{\bfS_{\tau}(\II_{\itLamb})} \To{\cong} \bfW \To{\subseteq} \bfW_{\ext} \]
is the desired $\itLamb$-monomorphism.
\end{proof}

\begin{lemma} \label{lemma:SW2}
There exists a $\itLamb$-homomorphism $\w{\gamma}_{\xi\ext}: \bfW_{\ext}^{\oplus_p 2^n} \to \bfW_{\ext}$
such that the following diagram
\[
\xymatrix{
   \w{\bfS_{\tau}(\II_{\itLamb})}^{\oplus_p 2^n}
   \ar[r]^{\w{\gamma}_{\xi}}
   \ar[d]_{\expaSW^{\oplus 2^n}}
 & \w{\bfS_{\tau}(\II_{\itLamb})}
   \ar[d]^{\expaSW}
 \\
   \bfW_{\ext}^{\oplus_p 2^n}
   \ar[r]_{\w{\gamma}_{\xi\ext}}
 & \bfW_{\ext}
}
\]
commutes and $\w{\gamma}_{\xi\ext}(\id, \ldots, \id) = \id$ holds.
\end{lemma}

\begin{proof}
The composition $\w{\gamma}_{\xi\ext} := \expaSW \compos \w{\gamma}_{\xi} \compos (\expaSW^{\oplus 2^n})^{-1}$ is the desired $\itLamb$-homomorphism.
\end{proof}

\begin{corollary}[Stone$-$Weierstrass Approximation Theorem] \label{coro:S-W approx}
There exists a unique morphism
\[ \expaSW: (\bfS_{\tau}(\II_{\itLamb}), \id, \gamma_{\xi})
\to (\bfW, \id, \w{\gamma_{\xi\dagger}}) \]
in $\Hom_{\scrN^1}((\bfS_{\tau}(\II_{\itLamb}), \id, \w{\gamma}_{\xi}),
(\bfW, \id, \w{\gamma_{\xi\dagger}}) )$ such that the diagram
\[\xymatrix@C=2cm{
  (\bfS_{\tau}(\II_{\itLamb}), \id, \gamma_{\xi}) \ar[r]^{\expaSW} \ar[d]_{\subseteq}
& (\bfW, \id, \w{\gamma_{\xi\dagger}}) \\
  (\w{\bfS_{\tau}(\II_{\itLamb})}, \id, \w{\gamma}_{\xi}) \ar[ru]_{\w{\expaSW}}
& }\]
commutes, where $\w{\expaSW}$ is the unique extension of $\expaSW$ lying in $\Hom_{\scrA^1}((\w{\bfS_{\tau}(\II_{\itLamb})}, \id, \w{\gamma}_{\xi}),$
$(\bfW, \id, \w{\gamma_{\xi\dagger}}) )$.
\end{corollary}

\section{Differentiations} \label{sect:deriva}
In this section, let $\scrA^p$ satisfy $\itLamb = \kk$ which is an extension of $\RR$,
and take $\tau = \ident$, $\xi = \frac{1}{2}$, $\mu=\Leb$, $\II_{\itLamb}=[0,1]$, and $\xi=\frac{1}{2}$.
In this case, the initial object of $\scrA^p$ is $(\w{\bfS}, \id, \gamma_{\frac{1}{2}})$,
where $\w{\bfS}=\w{\bfS_{\ident}([0,1])}$.

\subsection{Realizing variable upper limit integration as a morphism in $\scrA^1$}
\label{subsect:u l int}

We recall some works of Leinster in \cite[Section 2]{Lei2023FA}.
Let $C_*([0,1])$ be the set of all continuous functions $F:[0,1]\to \kk$ such that $F(0)=0$, with the sup norm
\[\Vert\cdot\Vert: C_*([0,1]) \to \RR^{\ge 0}, f \mapsto \sup_{x\in[0,1]} |f(x)|. \]
Then the triple $(C_*([0,1]), \ident, \kappa)$ of the $\kk$-module $C_*([0,1])$,
the identity function $\ident(x)=x$, and
the $\kk$-homomorphism $\kappa: C_*([0,1])^{\oplus 2} \to C_*([0,1])$ defined by
\[\kappa(F_1, F_2) =
\begin{cases}
 \frac{1}{2}F_1(2x), & 0\le x\le \frac{1}{2}; \\
 \frac{1}{2}(F_1(1)+F_2(2x-1)) & \frac{1}{2}\le x\le 1
\end{cases}\]
is an object in $\scrA^1$. Then the following proposition, first proved by Mark Meckes, holds.

\begin{proposition}[{\!\!\cite[Proposition 2.4]{Lei2023FA}}]\label{prop-9.1}
There exists a unique morphism
\[\w{T}_{[0,x]} : (\w{\bfS}, \id, \gamma_{\frac{1}{2}}) \to (\w{C_*([0,1])}, \ident, \w{\kappa}) \]
in $\Hom_{\scrA^1}((\w{\bfS}, \id, \gamma_{\frac{1}{2}}), (\w{C_*([0,1])}, \ident, \w{\kappa}))$
sending each function $f \in \w{\bfS}$ to the variable upper limit integration
$\displaystyle F(x) = (\rmL)\int_0^x f\dd\Leb$.
\end{proposition}

\subsection{Realizing differentiation as a preimage of a morphism in $\scrA^1$}
\label{subsect:deri 1}

It follows from Proposition \ref{prop-9.1}\point~that for any function $F\in \im(\w{T}_{[0,x]})$,
there exists an element $f \in \w{\bfS}$ such that
\begin{itemize}
  \item[(1)] If $F$ is a differentiable function (in the classical sense), then
    \begin{center}
      $\displaystyle \frac{\dd F}{\dd x} = f$ holds for all $x\in [0,1]$.
    \end{center}
    Here, $f$ is seen as a function in some equivalence class lying in $\w{\bfS}$, and, strictly speaking, $\frac{\dd F}{\dd x}$ is an element lying in the equivalence class containing $f$.

  \item[(2)] Otherwise, there exists a function $f$ such that
    \[ \int_0^1 F(x) \phi(x) \dd\Leb =- \int_0^1 f(x)\itPhi(x) \dd\Leb \]
    holds for any differentiable function $\itPhi: [0,1]\to\kk$ (in the classical sense)
    satisfying $\itPhi(0) = \itPhi(1) = 0$.
\end{itemize}
Thus, we can define the weak derivatives for functions lying in $\im(\w{T}_{[0,x]})$
by using the preimage of the $\kk$-homomorphism $\w{T}_{[0,x]}$ as follows.

\begin{definition} \rm
All functions lying in the preimage $\w{T}_{[0,x]}^{-1}(F)$ of $F \in \im(\w{T}_{[0,x]})$
are called {\defines weak derivatives} of $F$, and written $\w{T}_{[0,x]}^{-1}(F(x))$ as $\frac{\dd F}{\dd x}$.
\end{definition}

The following theorem shows that we cannot define the weak derivatives of
a function by using the morphism in $\scrA^1$ starting with $(\w{\bfS},\id,\w{\gamma}_{\frac{1}{2}})$.

\begin{theorem} \label{thm:diff} \
\begin{itemize}
  \item[{\rm(1)}] A morphism in $\Hom_{\scrA^1}((\w{\bfS},\id, \w{\gamma}_{\frac{1}{2}}), (N,v,\delta))$ is zero if and only if $v=0$.
  \item[{\rm(2)}] Furthermore, there is no morphism $D$ in $\scrA^1$ starting with $(\w{\bfS},\id, \w{\gamma}_{\frac{1}{2}})$
  such that $D$ sends any almost everywhere differentiable function $f(x)$ to its weak derivative $\frac{\dd f}{\dd x}$.
\end{itemize}
\end{theorem}

\begin{proof}
(1) For any $h\in \Hom_{\scrA^1}((\w{\bfS},\id, \w{\gamma}_{\frac{1}{2}}), (N,v,\delta))$,
the following diagram
\[\xymatrix{
  \w{\bfS}^{\oplus 2} \ar[r]^{\w{\gamma}}
  \ar[d]_{h^{\oplus 2}}
& \w{\bfS} \ar[d]^{h} \\
  N^{\oplus 2} \ar[r]_{\delta}
& N
}\]
commutes.

If $v=0$, then $h(\id)=v=0$, and the map $0:\w{\bfS} \to N, f\mapsto 0$ is a $\kk$-homomorphism such that the above diagram commutes.
By using $\Hom_{\scrA^1}((\w{\bfS},\id, \w{\gamma}_{\frac{1}{2}}), (N,v,\delta))$ to be a set containing only one morphism
(see Theorem \ref{thm:main1}\point), we obtain $h=0$.

Conversely, if $h=0$, then by the definition of morphism in $\scrA^1$, we have $v=h(\id)=0$.

(2) If there is an object $(N,v,\delta)$ such that $D:(\w{\bfS},\id, \w{\gamma}_{\frac{1}{2}})$ $\to$ $(N,v,\delta)$
is a morphism in $\scrA^1$ sending each almost everywhere differentiable function $f(x)$ to its weak derivative $\frac{\dd f}{\dd x}$,
then, by the definition of morphism in $\scrA^1$, we have $v=D(\id)=\frac{\dd\id}{\dd x}=0$.
It follows from (1) that $D=0$, which is a contradiction.
\end{proof}

\subsection{Realizing differentiation as a morphism in $\scrA^1$}
\label{subsect:deri 2}

In this subsection, we provide a description of differentiation by another morphism in $\scrA^1$.

Consider the triple $(\w{\bfS}, \ident, \w{\kappa})$,
where $\ident: [0,1] \to \kk, x\mapsto x$ is the function given in Subsection \ref{subsect:u l int}\point,
and $\w{\kappa}: \w{\bfS}^{\oplus 2} \to \w{\bfS}$ is a $\kk$-homomorphism defined as
\[\w{\kappa}(F_1, F_2) =
\begin{cases}
 \frac{1}{2}F_1(2x), & 0\le x\le \frac{1}{2}; \\
 \frac{1}{2}(F_1(1)+F_2(2x-1)), & \frac{1}{2}\le x\le 1
\end{cases}\]
which is a natural extension of the $\kk$-homomorphism $\kappa$
(the definition of $\kappa$ is given in Subsection \ref{subsect:u l int}\point).

\begin{lemma}
The triple $(\w{\bfS}, \ident, \w{\kappa})$ is an object in $\scrA^1$.
\end{lemma}

\begin{proof}
It is clear that $\w{\kappa}$ sends $(\ident, \ident)$ to $\ident$ by using the definition of $\w{\kappa}$.
Now, let $\{(F_{1,n}, F_{2,n})\}_{n\in\NN}$ be any Cauchy sequence in $\w{\bfS}^{\oplus 2}$ whose limits is $(F_1, F_2)$.
We need to prove $\underleftarrow{\lim} \kappa(F_{1,n}, F_{2,n}) =
\kappa (\underleftarrow{\lim} (F_{1,n}, F_{2,n}) )$.
Indeed, we have
\begin{align*}
  \underleftarrow{\lim} \kappa(F_{1,n}, F_{2,n})
& = \begin{cases}
      \frac{1}{2} \underleftarrow{\lim} F_{1,n}(2x), & 0\le x\le \frac{1}{2}; \\
      \frac{1}{2}(F_{2,n}(1) + \underleftarrow{\lim} F_{2,n}(2x-1)), & \frac{1}{2}\le x\le 1
    \end{cases} \\
& = \begin{cases}
      \frac{1}{2} F_1(2x), & 0\le x\le \frac{1}{2}; \\
      \frac{1}{2}(F_2(1) + F_2(2x-1)), & \frac{1}{2}\le x\le 1
    \end{cases} \\
& = \kappa(F_1, F_2) \\
& = \kappa (\underleftarrow{\lim} (F_{1,n}, F_{2,n}) ),
\end{align*}
as required.
\end{proof}

\begin{theorem} \label{thm:diff-non initial}
There exists a morphism
\[D \in \Hom_{\scrA^1}( (\w{\bfS}, \ident, \w{\kappa}), (\w{\bfS}, \id, \w{\gamma}_{\frac{1}{2}}) )\]
in $\scrA^1$ sending each element $f \in \w{\bfS}$ to its weak derivative.
\end{theorem}

\begin{proof}
First of all, the following diagram
\[\xymatrix{
  \w{\bfS}^{\oplus 2} \ar[r]^{\w{\kappa}}
  \ar[d]_{D^{\oplus 2}}
& \w{\bfS} \ar[d]^{D} \\
  \w{\bfS}^{\oplus 2} \ar[r]_{\w{\gamma}}
& \w{\bfS}
}\]
commutes, since for any $F_1(x), F_2(x) \in \w{\bfS}$, we have
\begin{align*}
    D \compos \w{\kappa}(F_1,F_2)
& = {
     \begin{cases}
      \displaystyle \frac{1}{2}\frac{\dd}{\dd x}F_1(2x), & 0\le x\le \frac{1}{2}; \\
      \displaystyle \frac{1}{2}\frac{\dd}{\dd x}(F_1(1)+F_2(2x-1)), & \frac{1}{2}\le x\le 1
     \end{cases}
    } \\
& = {
     \begin{cases}
      \displaystyle f_1(2x), & 0\le x\le \frac{1}{2}; \\
      \displaystyle f_2(2x-1), & \frac{1}{2}\le x\le 1
     \end{cases}
    } \\
& = \w{\gamma}(f_1, f_2) = \w{\gamma}\compos D^{\oplus 2}(F_1, F_2),
\end{align*}
where $\frac{\dd}{\dd x} F_i(x) = f_i(x)$ and $i\in\{1,2\}$.
Moreover, it is obvious that $D(\ident) = \frac{\dd}{\dd x}\ident = 1$,
thus $D$ is a morphism in $\scrA^1$.
\end{proof}

\section{Applications and examples} \label{sect:9}

\subsection{Lebesgue integration} \label{subsect:app1}

We assume the following assumptions hold in this subsection.

\begin{assumption} \label{assumption} \rm
Take $\kk=\RR$, $(\itLamb,\prec,\Vert\cdot\Vert_{\itLamb})=(\RR,\le,\Vert\cdot\Vert_{\RR})$, $B_{\RR}=\{1\}$
and $\norm{}: B_{\RR}\to \{1\}\subseteq\RR^{\ge 0}$.
Then $\dim\RR=1$, $\RR$ is a normed $\RR$-algebra with the norm $\Vert\cdot\Vert_{\RR}=|\cdot|: \RR\to\RR^{\ge 0}$
sending each real number $r$ to its absolute value $|r|$,
and any normed $\RR$-module is a normed $\kk$-vector space.
Take $\II_{\RR}=[0,1]$, $\xi=\frac{1}{2}$, $\kappa_0(x)=\frac{x}{2}$, $\kappa_1(x)=\frac{x+1}{2}$
and $\tau=\text{id}_{\RR}:\RR\to \RR$.
Then any object $(N, v, \delta)$ in $\scrN^p$ is a triple of a
normed $\kk$-module $N=(N,h_N,\Vert\cdot\Vert)$, an element $v\in N$ with $\Vert v\Vert_1$ and
the $\kk$-linear map $\delta: N\oplus_1 N \to N$,
where the norm $\Vert\cdot\Vert$ satisfies
$\Vert rx\Vert = |\tau(r)| \cdot\Vert x\Vert  = |r|\cdot\Vert x\Vert$
for any $r\in\itLamb=\RR$ and $x\in N$.
\end{assumption}

Under the above Assumption, we have the following properties for $\scrN^p$.

\begin{itemize}
  \item[\rm(L1)]
    The normed $\kk$-module $\bfS_{\tau}(\II_{\itLamb}) = \bfS_{\id_{\RR}}([0,1])$ ($=\bfS$ for short)
    is a $\kk$-vector space of all elementary simple functions which are of the form
    \[ f = \sum_{x=i}^t k_i\id_{[x_i,y_i]}, \]
    where $[x_i,y_i] \cap [x_j,y_j] = \varnothing$ for any $i\ne j$,
    and for any $f(r), g(r)\in \bfS$, it holds that
    \[\gamma_{\frac{1}{2}}(f, g)
      = \begin{cases}
          f(2r),   & 1\le r< \frac{1}{2}, \\
          g(2r-1), & \frac{1}{2}< r\le 1,
        \end{cases} \]
    by the definition of $\gamma_{\xi}$, see (\ref{map:gamma}).

  \item[\rm(L2)] $\scrA^p$ is a full subcategory,
    $(\bfS, \id_{[0,1]}, \gamma_{\frac{1}{2}})$ is an object in $\scrN^p$,
    but is not an object in $\scrA^p$ because $\bfS$ is not complete.

  \item[\rm(L3)] Let $\w{\bfS}$ be the completion of $\bfS$,
    and let $\w{\gamma}_{\frac{1}{2}}$ be the map $\w{\bfS}\oplus_1 \w{\bfS} \to \w{\bfS}$
    induced by $\gamma_{\frac{1}{2}}$.
    Then $(\w{\bfS}, \id_{[0,1]}, \w{\gamma}_{\frac{1}{2}})$ is an object in $\scrA^p$.
\end{itemize}

By Theorem \ref{thm:main1}\point, we obtain the following result directly.

\begin{corollary} \label{coro:main1}
  The triple $(\w{\bfS}, \id_{[0,1]}, \w{\gamma}_{\frac{1}{2}})$ is an $\scrA^p$-initial object in $\scrN^p$.
\end{corollary}

\begin{remark} \rm
It follows from Theorem \ref{thm:main1}\point~that $(\w{\bfS}, \id[0,1], \w{\gamma}_{\frac{1}{2}})$
is an initial object in $\scrA^p$, and then Corollary \ref{coro:main1}\point~holds.
In \cite{Lei2023FA}, Leinster showed that the initial object in $\scrA^p$ is
$(L^p([0,1]), \id_{[0,1]}, \gamma_{\frac{1}{2}})$. Then we obtain $L^p([0,1])\cong\w{\bfS}$
by uniqueness (up to isomorphism) of initial objects in arbitrary categories.
\end{remark}

Consider the triple $(\RR, 1, m)$ of the normed $\RR$-module $\RR$, the constant function and the map
\[ m: \RR \oplus_p \RR \to \RR \]
sending $(x,y)$ to $\frac{1}{2}(x+y)$. Then $(\RR, 1, m)$ is an object in $\scrA^p$,
and there is a family of $\RR$-linear maps $(L_i: E_i \to \kk)_{i\in\NN}$ such that the diagram
\[ \xymatrix{
  E_i\oplus_p E_i \ar[r]^{\gamma_{\frac{1}{2}}} \ar[d]_{\left(\begin{smallmatrix} L_i & 0 \\ 0 & L_i \end{smallmatrix}\right)}
& E_{i+1} \ar[d]^{L_{i+1}} \\
  \kk\oplus_p \kk \ar[r]_{m_i}
& \kk
}\]
commutes, where $E_i$ is the set of all step function constants on each $(\frac{t-1}{2^i}, \frac{t}{2^i})$,
$L_i$ sends $f=\sum_i k_i\id_{[a_i,b_i]}$ to $\sum_i k_i|b_i-a_i|$,
and $m = \underrightarrow{\lim} m_i$. Furthermore, we have the following result.

\begin{corollary}
There exists a unique morphism
\[L: (\bfS, \id_{[0,1]}, \gamma_{\frac{1}{2}}) \to (\RR, 1, m)\]
in $\Hom_{\scrN^1}((\bfS, \id_{[0,1]}, \gamma_{\frac{1}{2}}), (\RR, 1, m))$ such that the diagram
\[\xymatrix@C=2cm{
  (\bfS, \id_{[0,1]}, \gamma_{\frac{1}{2}}) \ar[r]^{L} \ar[d]_{\subseteq}
& (\RR, 1, m) \\
  (\w{\bfS}, \id_{[0,1]}, \w{\gamma}_{\frac{1}{2}}) \ar[ru]_{\w{L}}
& }\]
commutes, where $\w{L}$ is the unique extension of $L$ lying in $\Hom_{\scrA^p}((\w{\bfS}, \id_{[0,1]}, \w{\gamma}_{\frac{1}{2}}), (\RR, 1, m))$.
Furthermore, $\w{L}$ is given by the direct limit $\underrightarrow{\lim}L_i: \underrightarrow{\lim} E_i \to \kk$.
\end{corollary}

\begin{proof}
It is an immediate consequence of Theorem \ref{thm:main2}.
\end{proof}

The morphism $\w{L}$ induces a $\kk$-linear map sending $f$ to $\w{L}(f)$.
Furthermore, if $\mu=\Leb$ is a Lebesgue measure,
then $\w{L}(f)$ is Lebesgue integration $\displaystyle (\rmL)\int$ of $f$, \checks{i.e.,}
\[ \w{L}(f) = (\rmL)\int_0^1 f \dd\Leb,\]
where $\Leb$ is the Lebesgue measure in this case, see \cite[Proposition 2.2]{Lei2023FA}.

\vspace{0.25cm}

Next, as an application, we establish the Cauchy-Schwarz inequality for the morphism $\w{T}$ in $\scrN^1$.
We need the following lemma for arbitrary complete finite-dimensional $\RR$-algebras.

\begin{lemma} \label{lemm:positive}
If $f \in \w{\bfS_{\tau}(\II_{\itLamb})}$ is non-negative, then so is $\w{T}(f)$, i.e., $f\ge 0$ yields
\[(\scrA^1)\int_{\II_{\itLamb}} f\dd\mu \ge 0. \]
\end{lemma}

\begin{proof}
By $\w{\bfS_{\tau}(\II_{\itLamb})} = \underrightarrow{\lim} E_u$, there is a monotonically increasing sequence $\{f_t\}_{t\in\NN^+}$
of non-negative functions with $f_t = \sum_{i=1}^{2^{u_t}} k_{ti} \id_{I_{ti}} \in E_{u_t}$,
such that $I_{ti} \cap I_{tj} = \varnothing$ for any $i\ne j$;
$t_1<t_2$ yields $u_{t_1}<u_{t_2}$ and $f_{t_1}\le f_{t_2}$;
and $f = \underrightarrow{\lim} f_t$.
Thus, for any $1\le i \le 2^{u_t}$ and $t\in\NN^+$, we have $k_{ti\ge 0}$, and then the following inequality
\[\w{T}(f_t) = T_{u_t}(f_t) = \sum_{i=1}^{2^{u_t}} k_{ti} \mu(I_{ti}) \ge 0\]
holds. Furthermore, we obtain
\[\w{T}(f) = \underrightarrow{\lim} T_{u_t}(f_t) = \underrightarrow{\lim} T|_{E_{u_t}}(f_t) =  \underrightarrow{\lim} T(f_t) \ge 0\]
as required, where $\underrightarrow{\lim} T(f_t) = \lim\limits_{t\to+\infty} T(f_t)$ is the usual limit in $\RR$ in analysis.
\end{proof}

\begin{proposition}[Cauchy-Schwarz inequality]
Let $f$ and $g$ be two functions lying in $\w{\bfS_{\tau}(\II_{\itLamb})}$. Then
\begin{align}\label{formula:C-S ineq}
  \bigg((\scrA^1)\int_{\II_{\itLamb}} fg \dd\mu \bigg)^2
\le \bigg((\scrA^1)\int_{\II_{\itLamb}}f^2\dd\mu\bigg) \bigg((\scrA^1)\int_{\II_{\itLamb}}g^2\dd\mu\bigg).
\end{align}
\end{proposition}

\begin{proof}
Indeed, consider the quadratic function
\begin{align*}
  \varphi(t)
& = \w{T}(f^2) \cdot t^2 + 2\w{T}(fg) \cdot t + \w{T}(g^2) \ (t\in\RR).
\end{align*}
Notice that $\w{T}$ is a $\itLamb$-homomorphism, thus it is also an $\RR$-linear map.
Then
\begin{align*}
  \varphi(t) & = \w{T}(f^2 \cdot (t\id_{\RR})^2+2fg \cdot (t\id_{\RR}) + g^2)  = \w{T}((f \cdot (t\id_{\RR}) + g )^2).
\end{align*}
Notice that $(f \cdot (t\id_{\RR}) + g )^2$, written as $h$, is also a function defined
on $\II_{\itLamb}$ lying in $\bfS_{\tau}(\II_{\itLamb})$,
thus for any $x\in\II_{\itLamb}$, we have $h(x) = (tf(x)+g(x))^2 \ge 0$.
Then $\varphi(t) \ge 0$ by Lemma \ref{lemm:positive}. It follows that the discriminant
$(2\w{T}(fg))^2-4\w{T}(f^2)\w{T}(g^2)$ of $\varphi(x)$ is at most zero,
\checks{i.e.,} (\ref{formula:C-S ineq}) holds.
\end{proof}

The above inequality yields that if $\scrN^p$ satisfies the conditions (L1)--(L3) given in the subsection, then
the Cauchy-Schwarz inequality
\[  \bigg((\rmL)\int_0^1 f g \dd\Leb \bigg)^2
\le \bigg((\rmL)\int_0^1 f^2 \dd\Leb \bigg)
    \bigg((\rmL)\int_0^1 g^2 \dd\Leb \bigg) \]
holds.

\subsection{Series expansions of functions} \label{subsect:app2}

We provide two examples for Corollaries \ref{coro:power series}\point~
and \ref{coro:Fourier series}\point~in this subsection.

\begin{example}[{Taylor series}] \rm \label{examp:Taylor}
Assume that $\scrA^1$ satisfies Assumption \ref{assumption}. Then the $\itLamb$-homomorphism $\w{\expaP}$ in Corollary \ref{coro:power series}\point~is
\[ \w{\expaP}: (\w{\bfS}, \id, \w{\gamma}_{\frac{1}{2}})
\to (\w{\kk[x]}, \id, \w{\gamma_{\frac{1}{2}}}|_{\w{\kk[x]}}). \]
Now we show that for any analytic function $f(x) \in \w{\bfS}$, we have
\[\w{\expaP}: f(x) \mapsto \sum_{k=0}^{+\infty} \frac{1}{k!}
\left.\frac{\dd^k f}{\dd x^k}\right|_{x=0} x^k. \]
To do this, let $\mathbf{A}_0$ be the set of all analytic functions defined on $[0,1]$,
and define
\[\mathbf{A}_u = \{ \w{\gamma}_{\frac{1}{2}}(f,g) \mid (f,g)\in \mathbf{A}_{u-1}^{\oplus 2} \} \]
for any $u\in \NN$. Then we have
\[\kk[x] \subseteq \mathbf{A}_0 \subseteq \mathbf{A}_1 \subseteq \cdots \subseteq \w{\bfS} \cong L^1([0,1]). \]
Let $\mathfrak{E}_0:\mathbf{A}_0 \to \w{\bfS}$ be the map sending each analytic function $f(x)$ to its Taylor series $\displaystyle \sum_{k=0}^{+\infty} \frac{1}{k!}
\left.\frac{\dd^k f}{\dd x^k}\right|_{x=0} x^k \in \w{\kk[x]}$.
Then $\mathfrak{E}_0$ is a $\kk$-linear map since for any $a,b\in\itLamb=\RR$ and $f,g\in\w{\bfS}$,
the $\RR$-linear formula $\mathfrak{E}_0(a\cdot f + b\cdot g) = a\mathfrak{E}_0(f) + b \mathfrak{E}_0(g)$ holds.
Furthermore, one can check that $\mathfrak{E}_0$ is a $\itLamb$-homomorphism.
For any $u\in\NN$, any function $f$ in $\mathbf{A}_u$ can be seen as two functions $f_1$ and $f_2$ lying in $\mathbf{A}_{u-1}$ such that
\[ f = \w{\gamma}_{\frac{1}{2}}(f_1, f_2)
= \begin{cases}
f_1(2x), & 0\le x < \frac{1}{2};  \\
f_2(2x-1), & \frac{1}{2}< x \le 1.
\end{cases} \]
Thus, we can inductively define
\[\mathfrak{E}_u: \mathbf{A}_u \to \w{\kk[x]},
f \mapsto \w{\gamma}_{\frac{1}{2}}(\mathfrak{E}_{u-1}(f_1), \mathfrak{E}_{u-1}(f_2)). \]
Let $\mathbf{A}$ be the direct limit $\underrightarrow{\lim}{\mathbf{A}_u}$ given by $\mathbf{A}_0 \subseteq \mathbf{A}_1 \subseteq \cdots$.
The following statements (a) and (b) show that $\mathfrak{E} := \underrightarrow{\lim}\mathfrak{E}_u: \mathbf{A} \to \w{\kk[x]}$, induced by $\underrightarrow{\lim}\mathbf{A}_u = \mathbf{A}$, is a homomorphism in $\scrN^1$.
\begin{itemize}
  \item[(a)] First of all, it is obvious that $\mathfrak{E}(\id) =
    \underrightarrow{\lim}\mathfrak{E}_u(\id) = \underrightarrow{\lim} \id = \id$.
  \item[(b)] Next, for any two functions $f_1(x),f_2(x)$ in $\mathbf{A}$,
    the following diagram
    \[\xymatrix@R=1.25cm@C=1.25cm{
        \mathbf{A}^{\oplus 2}
        \ar[r]^{\w{\gamma}_{\frac{1}{2}}\big|_{\mathbf{A}}}
        \ar[d]_{ \left(\begin{smallmatrix}
          \mathfrak{E} & 0 \\
          0 & \mathfrak{E}
         \end{smallmatrix}\right) }
      & \mathbf{A}
        \ar[d]^{\mathfrak{E}} \\
        \w{\kk[x]}^{\oplus 2}
        \ar[r]_{\w{\gamma}_{\frac{1}{2}}}
      & \w{\kk[x]}
    }\]
    commutes since
    \begin{align*}
          \mathfrak{E}(\w{\gamma}_{\frac{1}{2}}|_{\mathbf{A}}(f(x),g(x)))
      & = \begin{cases}
            \mathfrak{E}(f(2x)), & 0\le x< \frac{1}{2}; \\
            \mathfrak{E}(g(2x-1)), & \frac{1}{2}<x\le 1
          \end{cases} \\
      & = \w{\gamma}_{\frac{1}{2}}(\mathfrak{E}(f(x)),\mathfrak{E}(g(x))).
    \end{align*}
\end{itemize}
Thus, the completion $\w{\mathbf{A}}$ of $\mathbf{A}$ induces a $\itLamb$-homomorphism
$\w{\mathfrak{E}}: \w{\mathbf{A}} \to \w{\kk[x]}$ which provides a morphism
\[ \w{\mathfrak{E}} \in \Hom_{\scrA^1}
    (
      (\w{\mathbf{A}}, \id, \w{\gamma}_{\frac{1}{2}}|_{\mathbf{A}}),
      (\w{\kk[x]}, \id, \w{\gamma}_{\frac{1}{2}}|_{\kk[x]})
    ) \]
in the category $\scrA^1$.

On the other hand,
for any polynomial $P(x) \in \kk[x]$, there exists a monotonically increasing sequence $\{s_i(x)\}_{i=0}^{+\infty}$
of elementary simple functions such that $\underleftarrow{\lim}s_i(x) = P(x)$.
Then we obtain that $\kk[x]$ is dense in $\w{\bfS}$.
It follows that $\w{\mathbf{A}}$ is dense in $\w{\bfS}$ by $\kk[x]\subseteq \w{\mathbf{A}}$.
Thus, we have an isomorphism
$\eta: (\w{\mathbf{A}}, \id, \w{\gamma}_{\frac{1}{2}}|_{\mathbf{A}})
  \mathop{\longrightarrow}\limits^{\cong}
  (\w{\bfS}, \id, \w{\gamma}_{\frac{1}{2}})$
and an isomorphism
\[ \w{\mathfrak{E}}\eta^{-1}:
   (\w{\mathbf{A}}, \id, \w{\gamma}_{\frac{1}{2}}|_{\mathbf{A}})
   \mathop{\longrightarrow}\limits^{\cong}
   (\w{\kk[x]}, \id, \w{\gamma}_{\frac{1}{2}}|_{\kk[x]})
   \]
in the category $\scrA^1$ such that
\[\w{\expaP}(f) = (\w{\mathfrak{E}}\eta^{-1})|_{\mathbf{A}_0}(f)
= \mathfrak{E}_0(f) =
\sum_{k=0}^{+\infty} \frac{1}{k!}
     \left.
       \frac{\dd^k f}{\dd x^k}
     \right|_{x=0} x^k \]
holds for any analytic function $f$ by using $(\w{\bfS}, \id, \w{\gamma}_{\frac{1}{2}})$ to be an initial object of $\scrA^1$ (see Theorem \ref{thm:main1}\point).
\end{example}

\begin{example}[{Fourier series}] \rm \label{examp:Fourier}
Assume that $\scrA^1$ satisfies Assumption \ref{assumption}. Then the $\itLamb$-homomorphism $\w{\expaF}$ in Corollary \ref{coro:Fourier series}\point~is
\[ \w{\expaF}: (\w{\bfS}, \id, \w{\gamma}_{\frac{1}{2}})
\to (\w{\CC[\power^{\pm 2\pi\rmi x}]}, \id,
  \w{\gamma_{\frac{1}{2}}}|_{\w{\CC[\power^{\pm 2\pi\rmi x}]}}), \]
which sends each function $f$ satisfying the Dirichlet condition to its Fourier series.
The proof of the above statement is similar to that of Example \ref{examp:Taylor}\point~by using $\CC[\power^{\pm 2\pi\rmi x}]$ to be a dense subspace of $\w{\bfS}$.
In particular, $\w{\expaF}$ induces an isomorphism in $\scrA^1$.
\end{example}

\section{Conclusions}

In this paper, we have significantly expanded the theoretical landscape of mathematical analysis
by extending the domain of classical Lebesgue integration beyond the real numbers and establishing
a robust framework for the major branches of analysis--differentiation, integration,
and series--over finite-dimensional $\kk$-algebras. By developing the categories $\scrN^p$ and $\scrA^p$,
we have introduced a structured methodology for examining norms and integration within an algebraic context.
This approach not only enhances our understanding of these processes but also provides
a unified perspective across various analytical branches.

Our study has not only reinforced existing mathematical theories within a generalized algebraic setting,
but has also paved the way for exploring how these concepts interact within the realms of category theory.
The categorification of key analytical operations such as differentiation and integration through normed modules
and their morphisms in $\scrA^p$ illustrates a significant theoretical advance,
bridging various analytical disciplines through a common categorical framework.

The implications of this work extend beyond the theoretical, suggesting applications in fields that
benefit from a deep understanding of the algebraic underpinnings of analysis, such as
computational mathematics and theoretical physics. Looking forward, the exploration of
higher-dimensional normed modules within this categorical framework promises to
open new research avenues in areas such as quantum field theory and numerical methods
for differential equations.

In summary, our research not only deepens the mathematical understanding of the interplay
between algebra and analysis, but also lays a solid foundation for further explorations.
Future work can extend these methods to more complex algebraic structures and
explore their practical applications in science and engineering, thereby continuing to
bridge the gap between abstract theory and real-world problem-solving.




\section*{Acknowledgements}
{
Yu-Zhe Liu is supported by the National Natural Science Foundation of China (Grant No. 12401042),
Guizhou Provincial Basic Research Program (Natural Science) (Grant Nos. ZD[2025]085 and ZK[2024]YiBan066)
and Scientific Research Foundation of Guizhou University (Grant Nos. [2022]53, [2022]65).
Shengda Liu is supported by the National Natural Science Foundation of China (Grant No. 62203442).
Zhaoyong Huang is supported by the National Natural Science Foundation of China (Grant Nos. 12371038, 12171207).
Panyue Zhou is supported by the National Natural Science Foundation of China (Grant No. 12371034)
and Scientific Research Fund of Hunan Provincial Education Department (Grant No. 24A0221).
We are greatly indebted to Tiwei Zhao, Yucheng Wang, Jian He, and Yuan Yuan for helpful suggestions; and are extremely grateful to the referees for his/her meticulous review and valuable comments and suggestions, which greatly improved the quality of our paper.
}




\def\cprime{$'$}

\end{document}